\documentclass[3p]{elsarticle}

\usepackage{amssymb,amsfonts,amsthm,amsmath,mathrsfs,sectsty,graphicx}
\usepackage{algorithm}
\usepackage{ifthen}
\usepackage{subfig}
\usepackage{enumitem}
\usepackage{nameref}
\usepackage{hyperref}
\usepackage{xcolor}
\usepackage{placeins}
\usepackage{multirow}
\usepackage{comment}

\usepackage{siunitx}
\usepackage{pdfpages}
\usepackage{epstopdf}

\graphicspath{ {figures/} }

\newtheorem{theo}{Theorem}
\newtheorem{lem}{Lemma}

\newtheorem{prop}{Proposition}

\theoremstyle{remark}

\newtheorem{remark}{Remark}

\sectionfont{\normalsize}
\subsectionfont{\normalsize}

\hypersetup{
    colorlinks,
    linkcolor={red!50!black},
    citecolor={blue!50!black},
    urlcolor={blue!80!black}
}

\newcommand{\der}[1]{\frac{\mathrm{d}#1}{\mathrm{d}t}}
\newcommand{\derpar}[2]{\frac{\partial #1}{\partial #2}}

\newcommand{\norz}[2]{\left\|#1\right\|_{0,#2}}
\newcommand{\noro}[2]{\left|#1\right|_{1,#2}}

\newcommand{\inp}[3]{\left( #1,\,#2\right)_{#3}}

\newcommand{\vertiii}[1]{{\left\vert\kern-0.25ex\left\vert\kern-0.25ex\left\vert #1 
    \right\vert\kern-0.25ex\right\vert\kern-0.25ex\right\vert}}

\newcommand{\intT}[4]{\int_{#1}^{#2}#3\,\mathrm{d}#4}

\DeclareMathOperator\sech{sech}

\makeatletter
\def\namedlabel#1#2{\begingroup
    #2%
    \def\@currentlabel{#2}%
    \label{#1}\endgroup
}
\makeatother

\begin{document}

\begin{frontmatter}
  \title{Anisotropic space-time adaptation for reaction-diffusion problems}
  \author[rvt]{Edward Boey\corref{cor}}
  \ead{eboey041@uottawa.ca}
  
  \author[rvt]{Yves Bourgault}
  \author[rvt]{Thierry Giordano}
  
  \
  \cortext[cor]{Corresponding author}
  \address[rvt]{Department of Mathematics and Statistics, University of Ottawa,
    585 King Edward Avenue, Ottawa, ON, Canada, K1N 6N5}

  \begin{abstract}
    
    A residual error estimator is proposed for the energy norm of the error for
    a scalar reaction-diffusion problem and for the monodomain model used in
    cardiac electrophysiology. The problem is discretized using $P_1$ finite
    elements in space, and the backward difference formula of second order
    (BDF2) in time. The estimator for space makes use of anisotropic
    interpolation estimates, assuming only minimal regularity. Reliability of
    the estimator is proven under certain mild assumptions on the convergence of
    the approximate solution. The monodomain model couples a nonlinear parabolic
    partial differential equation (PDE) with an ordinary differential equation
    (ODE) and this setting presents challenges theoretically as well as
    numerically. A space-time adaptation algorithm is proposed to control the
    global error, using a non-Euclidean metric for mesh adaptation and a simple
    method to adjust the time step. Numerical examples are used to verify the
    reliability and efficiency of the estimator, and to test the adaptive
    algorithm. The potential gains in efficiency of the proposed algorithm
    compared to methods using uniform meshes is discussed.

  \end{abstract}


\end{frontmatter}

\section{Introduction}
\label{sec:intro-react}


Reaction-diffusion equations model physical problems in a large variety of
situations. Numerically, these problems tend to be stiff, and can be demanding
in terms of computational resources, particularly in regions that develop sharp
wave fronts, solitons, etc., and for solutions that exhibit multiscale
behaviour. Achieving an accurate solution with uniform spatial and temporal
resolution can be impractical or even impossible. This paper addresses the
problem of improving efficiency and accuracy by the use of adaptive
techniques. The adaptation should be based, when possible, on theoretically
justified \textit{a posteriori} error estimates that are computable from the
approximate solution. We consider in detail two problems: first, a parabolic
scalar reaction-diffusion equation, and following, the more complex monodomain
system used to model cardiac electrophysiology, which couples a parobolic
equation with an ordinary differential equation.


A variety of approaches have been considered to estimate the error for
reaction-diffusion systems. Energy techniques are applied in \cite{burpic03} to
derive explicit error estimates. The error is bounded by a sum of residual
terms, together with interpolation estimates. A similar class of estimators in
various norms is found in \cite{erijoh95}, \cite{estlarwil00}, where the error
is estimated using duality techniques. The residual terms are weighted by
stability constants, obtained by approximating a dual problem, which indicate
the rate of accumulation of error. While such techniques are generally more
expensive, owing to the requirement of solving a dual problem, they are useful
if one is interested in controlling the value of an arbitrary functional of the
error. Another approach can be found in \cite{coldeuerdlanpav06},
\cite{huakamlan10}, \cite{jens-multi}, where the error is approximated by
solving an auxiliary problem with a hierarchical basis.


Equally important to error estimation is the application of adaptation
techniques. Mesh adaptation based on a posteriori error estimates for
reaction-diffusion problems can be found in \cite{biebab82b}, \cite{erijoh95},
\cite{estlarwil00}, \cite{coldeuerdlanpav06} in an isotropic context. Here the
primary mesh operations performed are refinement and coarsening respectively in
regions where the estimator is large or small.  These estimators generally
employ finite element interpolation or projection estimates, which rely on a
mesh regularity assumption such as the minimum (or maximum) angle
condition. With the aid of anisotropic interpolation estimates, for instance
from \cite{cao07}, \cite{forper01}, \cite{kun99}, the classical a posteriori
estimates have been applied to mesh adaptation in anisotropic framework in
\cite{huakamlan10}, \cite{kun99}, \cite{micper10}, \cite{pic03a} for linear
problems, and in \cite{boupicalalos09}, \cite{burpic03} for nonlinear
problems. The adaptation is driven by constructing a non-Euclidean metric from
the estimator, often employing a gradient and Hessian recovery technique. See
for instance \cite{fregeo08}, \cite{habdombouaitforval00}, \cite{losala11a} for
details on metric mesh adaptation in a general context. The use of anisotropic
methods is generally found to result in significantly lower error for a given
number of elements compared to isotropic methods; see for instance
\cite{forper03}, \cite{huakamlan10}.


The majority of the adaptive techniques applied to the models in cardiac
electrophysiology have been dedicated to isotropic meshes. A heuristic method is
employed in \cite{lingrotve03}, where mesh elements containing the wave-front
are successively refined, based on the observation that the variation of the
solution is low outside this region. A similar heuristic method is employed in
\cite{whi07}, where elements are refined based on the magnitude of the gradient
of the transmembrane potential, which again, is expected to occur primarily in
the wave front. Additionally, the time step is adapted based on the variation of
certain ionic variables, specific to the ionic model. Both \cite{lingrotve03}
and \cite{whi07} report an increase in computational efficiency compared to
uniform refinement. In \cite{chegrehen00} the mesh and time step are adapted
based on estimating the local truncation error for a finite difference scheme
using a Richardson extrapolation, and a similar technique is applied to a finite
element method in \cite{trakim04}, while in \cite{moore00}, the interpolation
error is approximated for trilinear elements. In \cite{coldeuerdlanpav06}, the
authors use a hierarchical error estimator, which approximates the residual in a
higher-order space, for a multilevel finite element discretization in space and
a Rosenbrock time-stepping scheme. This work is applied to computations on a
realistic heart geometry, simulating fibrillation dynamics in
\cite{deuerdroi09}. The theoretical foundation of the method can be found in
\cite{jens-multi}. A different approach is taken in \cite{ABK13}, where they
apply a $p$-adaptive method. The error estimator is based on an approximation of
the error in space only, between the semidiscretization in time and the full
discretization. An advantage of their method is a relatively quick reassembly of
the matrices involved, since the mesh connectivity is preserved. While the
adaptive approach taken in this paper means that we cannot avoid the issue of
matrix reassembly, the general methodology we take is to not adapt the mesh too
often.


Work on adaptive methods in an anisotropic setting applied to electrophysiology
has only begun recently. The first work in that direction can be found in
\cite{bel08}. The mesh adaptation is based on a simple hierarchical estimator,
constructed from gradient and Hessian recovery techniques. Results are presented
for 2D spiral waves, where the elements of the mesh are aligned for the
minimization of the gradient of the error, capturing the anisotropic features of
the solution. Results in 3D are presented in \cite{belforbou09}, where the
authors apply a Riemannian metric adaptive technique using Hessian recovery, and
extend their results to 3D scroll waves in \cite{belforbou14}. In these works,
however, mesh adaptation is performed after every time step. The potential gains
in CPU time using the adaptive mesh could be offset by the increased overhead
required to perform the adaptation steps and to recompute the matrices to solve
the system. In \cite{sougorpigfarberpit10}, the adaptation step is performed
after fixed intervals, and a speedup of up to 11.2 compared to the uniform
method is observed. However, a significant percentage, up to $79$\%, of the
total computation time is still spent adapting the mesh. A parallel version was
presented in \cite{sougorpigfar10}. In \cite{rio12}, a similar method was used
while the portion of time spent adapting the mesh was reported to be about
$25$\% of the total time when adapting every $10$ time steps.


The application of adaptation in both space and time has been explored for a
number of problems and time-stepping methods. To give some examples, estimates
for arbitrary order continuous and discontinuous Galerkin methods are considered
in \cite{estlarwil00}, space-time adaptation in an anisotropic setting for the
first-order discontinuous Galerkin method is given in \cite{micper10}, while the
Crank-Nicolson method is used in \cite{lozpicpra09}, for linear problems and in
\cite{prapicgij10} for nonlinear ones. A popular choice for reaction-diffusion
problems is the fully implicit backward difference formula of second order
(BDF2). The method is second-order accurate, and has good stability properties
applied to nonlinear problems, for instance see \cite{EB08}, \cite{tho06} for
constant time step, \cite{belforbou09} for use in mesh adaptation, and
\cite{bec98}, \cite{emm04} for stability of the variable time-step method. As
far as the author is aware, a space-time adaptation algorithm driven by a
posteriori error estimates has not been considered for the BDF2 discretization.
In \cite{akrcha10} optimal order a posteriori error estimates were derived for
the method applied to linear ODEs. However, their theoretical and numerical
results considered only the use of a constant time step, and they did not
consider time-dependent PDEs.


The first goal of this paper is to extend the estimator in \cite{akrcha10} to
the full discretization of nonlinear problems. First-order simplicial elements
will be used for the spatial discretization, and we will employ a simple
explicit a posteriori error estimator using energy techniques. The general
framework uses the anisotropic interpolation estimates for piecewise linear
elements found in \cite{forper01} and \cite{forper03}, which is combined with a
gradient recovery operator to obtain an a posteriori error estimator as is done
in \cite{micper06} and \cite{pic03a}. As in \cite{pic03a}, mesh adaptation is
only performed when the estimated error is above or below a certain threshold,
therefore systematically avoiding the issue of too frequent adaptation found in
previous anisotropic methods in electrophysiology. For the monodomain system, we
find that it is necessary to treat the ODE variable differently as it does not
benefit from parabolic smoothing. We propose a modified estimator that does not
make use of a residual. In some previous work, for instance in
\cite{lingrotve03}, \cite{sougorpigfarberpit10}, \cite{whi07} the adaptation
takes into account only the variation of the transmembrane potential, likely
based on the observation that it varies more rapidly than the ODE
variables. Here we illustrate numerically that all variables should be taken
into account, especially when simulating a heartbeat with realistic duration
scales. We use the fully implicit backward difference formula of second order
(BDF2) for the discretization in time. The error due to the time discretization
is approximated with an extension of the estimator from \cite{akrcha10} to the
nonlinear setting and with a variable time step. A space-time adaptation
algorithm is employed to control the error of the full discretization. In
addition to the residual estimator, for the monodomain system we consider a
simplified estimator for the recovery variable based only on the recovered
gradient. The simplified estimator is found to be more useful in practice.


The rest of the paper will proceed as follows. In Section \ref{sec:model-react}
we introduce the problem studied and all functional notation. In Section
\ref{sec:apost} we introduce the a posteriori error estimators and prove some
upper bounds: first for the semidiscrete problem in space in Section
\ref{subsec:semidisc}, followed by the full BDF2 discretization in Section
\ref{subsec:gear}. In Section \ref{sec:numer-react} we verify numerically the
equivalence of the estimator with exact error, and validate the numerical
algorithms. It is shown that applying the algorithm leads to optimal
second-order behaviour in time. The efficiency of the method is considered in
detail.


\section{Functional spaces and model problems}
\label{sec:model-react}

Let $\Omega$ be a bounded polygonal domain in $\mathbb{R}^2$ with finitely many
edges. For two measurable functions $u,\,v:\Omega\rightarrow\mathbb{R}^n$ whose
inner product is integrable we will denote by
$\inp{u}{v}{\Omega} =\int_\Omega (u, v)\,\mathrm{d}x$ their $L^2$ inner product.
Let $V=H^1(\Omega),\,H=L^2(\Omega),$ with topological duals $V',\,H'.$ There
exist continuous dense embeddings $V\subseteq H=H'\subseteq V'.$ Therefore,
defining $\inp{u}{f}{V,V'}=\inp{u}{f}{\Omega}$ for $u\in V,\,f\in H,$ the
duality between $V$ and $V'$ can be expressed in terms of the duality between
$H$ and itself.

Let $T>0$ and define
\begin{equation*}
  \mathcal{W}(V,V')
  =\{w:[0,T]\rightarrow V\,:\,w\in L^2(0,T;V),\,\partial_tw\in L^2(0,T;V')\},
\end{equation*}
where the derivative is meant in the vector-valued distributional sense. There
is a continuous embedding $\mathcal{W}(V,V')\subseteq C([0,T];H),$ so that in
particular $w(0),\,w(T)\in H$ are well-defined, see for instance \cite{lio69}.

Let $f\in C(\mathbb{R})$. We will assume that there exists $2\leq p<\infty$,
with conjugate exponent $1<q\leq2$ such that $f(u)\in L^q(0,T;L^q(\Omega))$ for
$u\in\mathcal{W}(V,V')\cap L^p(0,T;L^p(\Omega))$, for instance if $f$ satisfies
the growth conditions from \cite{mar89}. As $\Omega\subseteq\mathbb{R}^2$, the
Sobolev embedding $V\subseteq L^p(\Omega)$ holds, so that the integral
$\int_0^T\inp{f(u)}{v}{\Omega}\,\mathrm{d}t$ is defined for $v\in
V$. Additionally, $f$ is assumed to satisfy one of the following:
\begin{description}
  \item[\namedlabel{cond1}{(F1)}] $f$ is continuously differentiable and for some
    $\alpha\geq0$ its derivative satisfies $f'(x)\geq-\alpha,\,\forall
    x\in\mathbb{R},$
  \item[\namedlabel{cond2}{(F2)}] $f$ is locally Lipschitz continuous.
\end{description}
A typical example of \ref{cond1} is $f(x)=\sum_{k=0}^{2p-1}b_kx^k$, that is, a
polynomial of odd degree with real coefficients such that $b_{2p-1}>0.$ We can
now define the model initial value problem. For
$u_0\in H,\,\,g\in L^2(0,T;L^2(\partial\Omega))$ and $s\in L^2(0,T;H)$, let
$u\in\mathcal{W}(V,V')\cap L^p(0,T;L^p(\Omega))$ be the solution to the initial
value problem
\begin{equation}
  \label{eqn:react}
  \left\{
    \begin{array}{ll}
      \dfrac{\partial u}{\partial t}-\Delta u+f(u)=s,&\quad\text{in }(0,T)\times\Omega,\\[1ex]
      \nabla u\cdot n=g,&\quad \text{in }(0,T)\times\partial\Omega,\\
      u(0)=u_0,&\quad \text{in }\Omega.
    \end{array}
  \right.
\end{equation}
Problem (\ref{eqn:react}) will be considered in the following variational formulation:
\begin{align}
  \der{}\inp{u}{v}{\Omega}+\inp{\nabla u}{\nabla v}{\Omega}+\inp{f(u)}{v}{\Omega}
  &=\inp{s}{v}{\Omega}+\inp{g}{v}{\partial\Omega},\quad\forall v\in V\nonumber\\
  u(0)&=u_0\label{eqn:react-var}.
\end{align}
For existence and uniqueness of solutions to the initial value problem
(\ref{eqn:react-var}), see for instance \cite{lio69} or \cite{tem88}.


We now introduce monodomain system, used to model problems in cardiac
electrophysiology. For $u_0,\,w_0\in H$, let
$u\in\mathcal{W}(V,V'),\,w\in \mathcal{W}(H,H')$ be the solution to the
following initial value problem:
\begin{equation}
  \label{eqn:mono}
  \left\{
    \begin{array}{ll}
      \dfrac{\partial u}{\partial t}-\Delta u+F(u,w)=0,&\quad\text{in }(0,T)\times\Omega,\\[1.5ex]
      \dfrac{\partial w}{\partial t}+G(u,w)=0,&\quad\text{in }(0,T)\times\Omega, \\[1ex]
      \nabla u\cdot n=0,&\quad \text{on }(0,T)\times\partial\Omega,\\
      u(0)=u_0,\\
      w(0)=w_0,
    \end{array}
  \right.
\end{equation}
where $F,\,G:\mathbb{R}^2\rightarrow\mathbb{R}$ are continuous.  Moreover, we
assume that $F,\,G$ satisfy one of the following conditions for every bounded
domain $\mathcal{D}\subseteq\mathbb{R}^2$,
\begin{description}
  \item[\namedlabel{mono-cond1}{(M1)}] there exists $\alpha_\mathcal{D}>0$ such
    that for
    $x=(x_1,x_2),\,y=(y_1,y_2)\in \mathcal{D}$
    \begin{align*}
      (F(x)-F(y))(x_1-y_1)+(G(x)-G(y))(x_2-y_2)
      &\geq-\alpha_{\mathcal{D}}\|x-y\|_2^2,
    \end{align*}
  \item[\namedlabel{mono-cond2}{(M2)}] there exists $\alpha_{\mathcal{D}}>0$ such that
    for $x,\,y\in \mathcal{D}$
    \begin{align*}
      \max\{|F(x)-F(y)|,\,|G(x)-G(y)|\}
      &\leq \alpha_{\mathcal{D}}\|x-y\|_2.
    \end{align*}
\end{description}
To improve the estimates, we also consider a stronger form of \ref{mono-cond1}
\begin{description}
  \item[\namedlabel{mono-cond3}{(M3)}] there exist positive constants
    $\alpha_\mathcal{D},$ $\beta_\mathcal{D}$ such that for
    $x=(x_1,x_2),\,y=(y_1,y_2)\in \mathcal{D}$
    \begin{align*}
      (F(x)-F(y))(x_1-y_1)+&(G(x)-G(y))(x_2-y_2)\\
                           &\geq-\alpha_{\mathcal{D}}(x_1-y_1)^2 + \beta_{\mathcal{D}}(x_2-y_2)^2.
    \end{align*}
\end{description}
The domain $\mathcal{D}$ will be implicitly assumed, and we will denote
$\alpha_{\mathcal{D}}$ by $\alpha$ and $\beta_{\mathcal{D}}$ by $\beta$ in
\ref{mono-cond1}, \ref{mono-cond2} and \ref{mono-cond3}. In particular, we
will assume the solution, and approximate solutions, are uniformly bounded in
$L^\infty(\Omega)$. Some details on the boundedness assumption follow in Section
\ref{sec:mono-est}.

There are several well-posedness results for (\ref{eqn:mono}). In
\cite{boucoupie09}, a weak solution is proven to exist globally in time provided
the reaction terms satisfy mild growth conditions, which apply for instance to
the FitzHugh-Nagumo and Aliev-Panfilov models. Uniqueness is proven for the
FitzHugh-Nagumo model. Additionally, local existence of more regular solutions
is proven provided further regularity of the reaction terms and the initial
data. In \cite{BFGZ08} existence is proven for a regularized version of the
Mitchell-Schaeffer model. Note that the results in the references above apply to
the bidomain problem, for which the monodomain problem is a simplification.

\begin{sloppypar}
  
  The error estimates in Section \ref{sec:mono-est} require additionally that
  $w$ belongs to $L^2(0,T;V)$. We can obtain more regular solutions by
  considering so-called strong solutions of (\ref{eqn:mono}) in the setting of
  \cite{hen81}. For instance, if we modify the definition of the spaces $Z$ and
  $Z^\alpha$ in \cite[Section 4]{boucoupie09} to use the space
  $B=V\cap L^\infty(\Omega)$, then on a maximal interval $[0,\tau_{\max})$,
  there exists a unique strong solution in the sense of \cite[Definition
  18]{boucoupie09} provided the function
  $(u,w)\in Z^\alpha\mapsto (F(u,w),G(u,w))\in Z$ is locally Lipschitz
  continuous and $(u_0,w_0)\in Z^\alpha$. In particular, we have
  $w\in C([0,\tau_{max});V)$. For the solution to exist globally, we require
  that the local Lipschitz condition be replaced by a global one. For $F$, this
  can be achieved provided there exists a priori bounds on $u,w$ in the
  $L^\infty(\Omega)$ norm. Such bounds will be discussed for specific ionic
  models in Section \ref{sec:mono-est} due to the existence of invariant
  rectangles. For $G$ this is not quite enough since the norm on $V$ involves
  the gradient. However, it suffices to assume that $G$ and its first
  derivatives are globally Lipschitz continuous, which is satisfied for the
  ionic models considered in this paper.
  
\end{sloppypar}


\section{A posteriori estimates}
\label{sec:apost}

\subsection{Semidiscrete problem}
\label{subsec:semidisc}

In this section we consider error estimates for the semidiscrete in space
approximation of problem (\ref{eqn:react}).

\subsubsection{Notation and background}
\label{subsubsec:notation}


Let $\mathcal{T}_h$ be a conformal triangulation of the domain $\Omega$ with
elements $K$ of diameter $h_K$ and consider the finite element approximation
space
$V_h=\{v_h\in C(\overline{\Omega}): v_h|_K\in P_1(K),\,\forall
K\in\mathcal{T}_h\}$. Let $\overline{u}_0\in V_h$ be an approximation of $u_0$.
The semidiscrete approximation $u_h\in C^1(0,T;V_{h})$ of $u$ satisfies
\begin{align}
  \der{}\inp{u_h}{v_h}{\Omega}+\inp{\nabla u_h}{\nabla v_h}{\Omega}+\inp{f(u_h)}{v_h}{\Omega}
  &=\inp{s}{v_h}{\Omega}+\inp{g}{v_h}{\partial\Omega},\quad v_h\in V_h,\nonumber\\
  u_h(0)&=\overline{u}_0.
\end{align}
Define the error $e_h=u-u_h.$ Below, we introduce the notation required to
derive an anisotropic residual estimator for the energy norm of the
error. Relevant results are cited from the literature.


The estimator combines information on the residual with anisotropic
interpolation estimates. Define the local residual $R_K(u_h)\in L^2(0,T;L^2(K))$ by
\begin{equation*}
  R_K(u_h)=\derpar{u_h}{t}-\Delta u_h+f(u_h)-s.
\end{equation*}
Here $\Delta$ denotes the Laplacian operator on $K$. The jump of the
derivative of $u_h$ for an element $K$ with edges $e_i$ is defined by
\begin{equation*}
  r_K(u_h)=\sum_{i=1}^3[\nabla u_h]_{e_i},
\end{equation*}
where the jump $[\nabla v_h]_{e_i}$ over $e_i$ is defined as follows: denoting
the outward unit normal by $n_i$ and the adjacent element (if it exists) by
$K'$, then
\begin{equation*}
  [\nabla u_h]_{e_i}=\left\{\begin{array}{ll}
      2\left(g-\nabla u_h\cdot n\right),
                              &\quad e_i\text{ is a boundary edge},\\
      \nabla (u_h)|_{K}\cdot n_i - \nabla (u_h)|_{K'}\cdot n_i,
                              &\quad e_i\text{ is an interior edge}.
  \end{array}\right.
\end{equation*}


Next, we introduce the interpolation estimates from \cite{forper01} and
\cite{forper03}. For a triangular element $K$, the anisotropic information comes
from the affine mapping $F_K:\hat{K}\rightarrow K$. The reference element
$\hat{K}$ is taken to be the equilateral triangle centred at the origin with
vertices at the points
$(0,1),\,(\frac{-\sqrt{3}}{2},\frac{-1}{2}),\,(\frac{\sqrt{3}}{2},\frac{-1}{2})$.
The Jacobian $J_K$ of $F_K$ is non-degenerate, so the singular value
decomposition (SVD) $J_K=\mathcal{R}_K^T\Lambda_K\mathcal{R}_K \mathcal{Z}_K$
consists of orthogonal matrices $\mathcal{R}_K,\,\mathcal{Z}_K$, and a positive
definite diagonal matrix $\Lambda_K$. The matrices $\mathcal{R}_K,\,\Lambda_K$
take the form
\begin{align*}
  \mathcal{R}_K=\begin{pmatrix}r_{1,K}^T\\r_{2,K}^T\end{pmatrix},
  \quad\quad\quad
  \Lambda_K=\begin{pmatrix}\lambda_{1,K} && 0\\0 && \lambda_{2,K}\end{pmatrix},
\end{align*}
where $\lambda_{1,K}\geq\lambda_{2,K}>0$, and $r_{1,K},\,r_{2,K}$ are orthogonal
unit vectors. Geometrically, the SVD represents the deformation of the unit
ball in $\mathbb{R}^2$ to an ellipse with axes of length
$\lambda_{1,K},\,\lambda_{2,K}$ in directions $r_{1,K},\,r_{2,K}$ respectively.
Moreover, the SVD also represents $K$ in the sense that the ellipse
circumscribes the element.

Let $I_h:H^1(\Omega)\rightarrow V_h$ denote a Scott-Zhang interpolation
operator, see \cite{scozha90}. Define the following ``Hessian'' type matrix:
\begin{equation*}
  \tilde{G}_K(v)=
  \left(\int_{\Delta_K}\frac{\partial v}{\partial x_i}\frac{\partial v}{\partial
    x_j}\,\,\mathrm{d}x\right)_{i,j},\quad\quad\text{for }v\in V,
\end{equation*}
and let
\begin{equation*}
  \tilde{\omega}_K(v)=\left(\lambda_{1,K}^2r_{1,K}^T \tilde{G}_K(v)r_{1,K}+\lambda_{2,K}^2r_{2,K}^T \tilde{G}_K(v)r_{2,K}\right)^{1/2}.
\end{equation*}
Here $\Delta_K$ is the patch of elements containing a vertex common to
$K$. Recall that as in \cite{micperpic03}, the usual minimum-angle condition is
not required, but instead, the uniform bound of the interpolation operator $I_h$
requires a mild patch condition to hold. In what follows, a constant
$C_{\hat{K}}$ will denote a positive constant which relies on such a patch
condition.

Define the local anisotropic residual estimator by
\begin{align}
  \eta_{K,t}^2
  &=\left(\|R_K(u_h)\|_{0,K} +
    \left(\frac{h_K}{\lambda_{1,K}\lambda_{2,K}}\right)^{1/2}\|r_K(u_h)\|_{0,\partial
      K}\right)\tilde{\omega}_K(e_h),\label{eq-est-aniso-rdiff}
\end{align}
and the global estimator
\begin{align}
  \eta_t^2
  =\sum_K\eta_{K,t}^2,
  \quad\quad\quad
  \eta
  =\left(\int_0^T\eta_t^2\,\mathrm{d}t\right)^{1/2}.
\end{align}




\begin{lem}\label{lem:distr}
  There exists a constant $C_{\hat{K}}>0$ independent of the mesh such that the following
  distributional inequality holds:
  \begin{align}
    \frac{1}{2}\der{}\norz{e_h}{\Omega}^2+\noro{e_h}{\Omega}^2
    &\leq C_{\hat{K}}\eta_t^2 - \inp{f(u)-f(u_h)}{e_h}{\Omega}.\label{eqn:react-res}
  \end{align}
\end{lem}
\begin{proof}

  From the variational formulation (\ref{eqn:react-var}) of $u$,
  \begin{align*}
    &\inp{\derpar{e_h}{t}}{e_h}{\Omega} + \inp{\nabla e_h}{\nabla e_h}{\Omega}\\
    &=\inp{s-\derpar{u_h}{t}}{e_h}{\Omega} - \inp{\nabla u_h}{\nabla e_h}{\Omega}
      + \inp{g}{e_h}{\partial\Omega} -
      \inp{f(u)}{e_h}{\Omega}\\
    &=\inp{s-\derpar{u_h}{t}-f(u_h)}{e_h}{\Omega} - \inp{\nabla u_h}{\nabla e_h}{\Omega} +
      \inp{g}{e_h}{\partial\Omega} -
      \inp{f(u)-f(u_h)}{e_h}{\Omega}.\\
  \end{align*}
  Using the fact that $I_h(e_h)\in V_h$ a.e. $t\in(0,T)$ and applying integration by parts for
  each triangle $K$
  \begin{align*}
    &\inp{s-\derpar{u_h}{t}-f(u_h)}{e_h}{\Omega} - \inp{\nabla u_h}{\nabla e_h}{\Omega} +
      \inp{g}{e_h}{\partial\Omega}\\
    &=\inp{s-\derpar{u_h}{t}-f(u_h)}{e_h-I_h(e_h)}{\Omega} - \inp{\nabla u_h}{\nabla (e_h-I_h(e_h))}{\Omega} +
      \inp{g}{e_h-I_h(e_h)}{\partial\Omega}\\
    &=\sum_K\left(\inp{R_h(u_h)}{e_h-I_h(e_h)}{K} + \frac{1}{2}\inp{r_h(u_h)}{\nabla
      (e_h-I_h(e_h))}{\partial K}\right).
  \end{align*}
  Therefore, applying the interpolation estimates from \cite{forper01} and
  \cite{forper03}, and the Cauchy-Bunyakowsky-Schwartz inequality
  \begin{align*}
    &\frac{1}{2}\der{}\norz{e_h}{\Omega}^2+\noro{e_h}{\Omega}^2\\
    &\leq\sum_K\left(\norz{R_h(u_h)}{K}\norz{e_h-I_h(e_h)}{K}+\frac{1}{2}\norz{r_h(u_h)}{\partial
      K}\norz{\nabla (e_h-I_h(e_h))}{\partial K}\right) \\
    &\quad - \inp{f(u)-f(u_h)}{e_h}{\Omega}\\
    &\leq C_{\hat{K}}\eta_{t}^2 - \inp{f(u)-f(u_h)}{e_h}{\Omega}.
  \end{align*}
\end{proof}

From Lemma \ref{lem:distr}, it follows that the main difficulty to proceed is to
deal with the last term on the right side of (\ref{eqn:react-res}), which
depends on the nonlinear function $f$.

\subsubsection{Upper bounds}
\label{subsubsec:monotone}

Here we derive two theoretical upper bounds for the energy norm of the error in
terms of the estimator. Recall that the energy norm for $v\in L^2(0,T;V)$ is
given by
\begin{align}
  \vertiii{v}
  =\left(\int_0^T\noro{v(t)}{\Omega}^2\,\mathrm{d}t\right)^{1/2}.
\end{align}
We would like to find an upper bound for the error of the
form
\begin{align}\label{ineq-upper}
  \vertiii{e_h}
  &\leq C\eta,
\end{align}
where $C>0$ is close to $1,$ and does not depend on the choice of mesh.

Propositions \ref{prop-gron} and \ref{prop-AN} can be proved if $f$ satisfies
\ref{cond1} or \ref{cond2}. To simplify the presentation, we only prove the
results in terms of \ref{cond1}. To see how \ref{cond2} can be used instead, see
\textit{Remark} \ref{rem:lipschitz}.


\begin{prop}
  \label{prop-gron}
  Suppose that $f$ satisfies \emph{\ref{cond1}}. Then for $C_{\hat{K}}>0$, independent
  of the mesh, we have
  \begin{align}
    \frac{1}{2}\norz{e_h(T)}{\Omega}^2+\int_{0}^{T}e^{2\alpha (T-t)}\noro{e_h(t)}{\Omega}^2\,\mathrm{d}t
    &\leq\frac{1}{2}e^{2\alpha T}\norz{e_h(0)}{\Omega}^2+C_{\hat{K}}\int_{0}^{T}e^{2\alpha (T-t)}\eta_{t}^2\,\mathrm{d}t.\label{apost-est2}
  \end{align}
\end{prop}

\begin{proof}

  From \ref{cond1} and the mean value theorem we conclude
  $(f(u)-f(u_h))(e_h)\geq -\alpha(e_h)^2$. Inequality (\ref{eqn:react-res})
  implies
  \begin{align}\label{}
    \frac{1}{2}\der{}\norz{e_h}{\Omega}^2+\noro{e_h}{\Omega}^2
    &\leq C_{\hat{K}}\eta_{t}^2 + \alpha\norz{e_h}{\Omega}^2.\label{eqn:react-gron1}
  \end{align}
  Taking the term $\noro{e_h}{\Omega}^2$ to the right hand side, we can apply
  Gronwall's inequality to (\ref{eqn:react-gron1}) to get
  \begin{align}\label{eqn:react-gron2}
    \frac{1}{2}\norz{e_h(T)}{\Omega}^2
    &\leq \frac{1}{2}e^{2\alpha T}\norz{e_h(0)}{\Omega}^2+\int_0^T e^{2\alpha(T-t)}\left(C_{\hat{K}}\eta_{t}^2-\noro{e_h(t)}{\Omega}^2\right)\,\mathrm{d}t,
  \end{align}
  and the result follows.
\end{proof}

\begin{sloppypar}
  In the context of (\ref{ineq-upper}), using the inequality
  $1\leq e^{2\alpha(T-t)}\leq e^{2\alpha T}$ and supposing that we may ignore
  the initial error term, we are led to consider the upper bound
  \begin{align}\label{ineq-gron-exp}
    \vertiii{e_h}
    &\leq e^{\alpha T}C_{\hat{K}}^{1/2}\eta.
  \end{align}
  While the constant on the right hand side of (\ref{ineq-gron-exp}) will not be
  close to $1$, we can at least conclude that an upper bound holds for fixed
  $T>0.$ For long time scales, however, we do not expect the upper bound to be
  sharp. Furthermore, the value $\alpha$ can be quite large as will be seen in
  Section \ref{sec:numer-react}. For instance, solutions to the bistable
  equation are traveling waves with width scaled proportional to
  $\alpha^{-1/2}$. Therefore, the stiffer the solution, the less optimistic is
  the theoretical result. Recall that classical a priori estimates for
  (\ref{eqn:react}) are typically of the form $C(u;T)e^{M T}O(h^k)$, where
  $C(u;T)$ depends on various norms of $u$, and where $M$ is a large constant
  depending on the derivatives of $f$. It should be remarked, however, that
  (\ref{apost-est2}) is essentially a worst case estimate in the sense that it
  is valid for any solution of the initial value problem (\ref{eqn:react}),
  including solutions with finite time blowup. In what follows, we derive a
  stronger estimate under the assumption that the approximate solution converges
  at an optimal rate.
\end{sloppypar}


We now get to the alternative upper bound in Proposition \ref{prop-AN}. The
following result is based on the assumption that the error converges faster in
the $L^2$ norm than the $H^1$ seminorm. We use a condition similar to
\cite[(3.4)]{burpic03}. For the general anisotropic case, we have some partial
results. We borrow from \cite{boiforfor12}, where the idea is that given a
$P_1$ approximation $u_h$, by utilizing the recovered gradient one may construct
a $P_2$ approximation $u_{h,2}$ which converges faster to $u$, both in $L^2$
norm and $H^1$ seminorm. While this superconvergence property would be
difficult to prove in general, test cases in \cite{boiforfor12} suggests that it
is true in practice. Condition (\ref{eqn:direct-equi-react}) is similar to
conditions used in \cite{pic06} and \cite{micper06}.
\begin{lem}[\cite{2016arXiv160900854B}, Proposition 1]\label{lem-l2-upper}
  Using the notation above, suppose that there exists $C_{\hat{K},1}>0$ such
  that, for all $K\in\mathcal{T}_h$ and a.e. $t\in [0,T]$,
  \begin{align}\label{eqn:direct-equi-react}
    \lambda_{1,K}\|\nabla (u_h-u_{h,2})\cdot r_{1,K}\|_{0,K}
    &\leq C_{\hat{K},1}\lambda_{2,K}\|\nabla (u_h-u_{h,2})\cdot r_{2,K}\|_{0,K}.
  \end{align}
  Then there exists $C_{\hat{K},2}>0$ independent of $t$ such that
  \begin{align}\label{eq-l2-upper-rdiff}
    \|e_h\|_{0,K}
    &\leq C_{\hat{K},2}\left(\lambda_{2,K}|e_h|_{1,K}
      + \|u-u_{h,2}\|_{0,K}
      + \lambda_{2,K}|u-u_{h,2}|_{1,K}\right).
  \end{align}
\end{lem}
From (\ref{eq-l2-upper-rdiff}), we conclude that there exists $C_{\hat{K}}>0$
for a.e. $t\in [0,T]$, such that
\begin{align*}
  \|e_h\|_{0,\Omega}
  &\leq C_{\hat{K}}\left(\left(\sum_K\lambda_{2,K}^2|e_h|_{1,K}^2\right)^{1/2}
  + \|u-u_{h,2}\|_{0,\Omega}
  + \left(\sum_K\lambda_{2,K}^2|u-u_{h,2}|_{1,K}^2\right)^{1/2}\right).
\end{align*}
Under superconvergence assumptions for $u_{h,2}$, the last two terms on the
right are assumed to be higher order, so neglected. Note that this is stronger
than the usual superconvergence assumptions, which are for stationary problems.
Additionally, we will assume that there exists a constant $C>0$ such that for
all $K\in\mathcal{T}_h$ and a.e. $t\in [0,T]$
\begin{align}\label{eq:h1-dist}
  |e_h|_{1,K}&\leq C N_T^{-1/2}|e_h|_{1,\Omega},
\end{align}
where $N_T$ is the number of elements. The mesh adaptation algorithm will
attempt to equidistribute the error over all elements, so that
(\ref{eq:h1-dist}) should hold in practice. Then noting that
$\sum_K\lambda_{2,K}^2\leq\sum_K\lambda_{1,K}\lambda_{2,K}=|\Omega||\hat{K}|^{-1}$,
there exists a constant $C_{AN}>0$ such that up to higher-order terms
\begin{align}\label{eq:h1-l2}
  \|e_h\|_{0,\Omega}^2
  &\leq C_{AN}N_T^{-1}|e_h|_{1,\Omega}^2.
\end{align}

\begin{prop}
  \label{prop-AN}
  Suppose that $f$ satisfies \emph{\ref{cond1}} and the error satisfies
  \emph{(\ref{eq:h1-l2})}. Then there exists $C_{\hat{K}}>0$ such that
  \begin{align}
    \norz{e_h(T)}{\Omega}^2 + \left(1-\alpha C_{AN}N_T^{-1}\right)\vertiii{e_h}^2
    &\leq \norz{e_h(0)}{\Omega}^2
    + C_{\hat{K}}\eta^2.\label{eqn:est-AN}
  \end{align}
\end{prop}

\begin{proof}

  Integrate (\ref{eqn:react-gron1}) over $t$ and apply (\ref{eq:h1-l2}).

\end{proof}

Ignoring the error on the initial condition $\norz{e_h(0)}{\Omega}$, estimate
(\ref{eqn:est-AN}) implies that as $N_T\rightarrow\infty$,
\begin{align*}
  \vertiii{e_h}
  \leq C_{\hat{K}}^{1/2}\left(1-\alpha
  C_{AN}N_T^{-1}\right)^{-1/2}\eta\rightarrow C_{\hat{K}}^{1/2}\eta.
\end{align*}
Therefore, we achieve an upper bound of the form
(\ref{ineq-upper}) asymptotically with respect to the mesh size. However, note
that the value of $\alpha$ can be large, and the constant $C_{AN}$ is not known
a priori, so it is not clear how fine the mesh needs to be so that
$1-\alpha C_{AN}N_T^{-1}>0.$

\begin{remark} \normalfont
  In general the constant $C_{AN}$ depends on the class of meshes considered.
  For isotropic meshes, where
  $\lambda_{1,K}\simeq\lambda_{2,K}\simeq h_K\simeq h$, the error is expected
  to converge $O(h^2)$ in the $L^2$ norm and $O(h)$ in the $H^1$ seminorm. Then
  from the relation $N_T\approx Ch^{-2}$, this translates to $O(N_T^{-1})$ for
  the $L^2$ norm and $O(N_T^{-1/2})$ for the $H^1$ seminorm, and we obtain
  (\ref{eq:h1-l2}).
\end{remark}

\begin{remark} \normalfont\label{rem:lipschitz} Propositions \ref{prop-gron} and
  \ref{prop-AN} were proven under the assumption of \ref{cond1}, which is used
  to prove estimate (\ref{eqn:react-gron1}). On the other hand, if $f$ satisfies
  \ref{cond2}, if we assume that $u\in L^\infty(0,T;L^\infty(\Omega))$ and
  moreover, that the collection $\{u_h\}_h$ is uniformly bounded in
  $L^\infty(0,T;L^\infty(\Omega)),$ then there exists a Lipschitz constant
  $C_f(u)>0$ for $f$ such that $|(f(u)-f(u_h))(e_h)|\leq C_f(u)(e_h)^2$ a.e.,
  and we obtain an analogue to (\ref{eqn:react-gron1}). The rest of the proof
  follows as before. The boundedness of $u$ is easy to prove, provided that the
  initial data is smooth enough and that the equation admits an invariant region
  \cite{Smoller-shock}. For uniform boundedness of the approximate solutions,
  see for instance \cite{estlarwil00}. If the solution blows up in finite time
  (in the $L^\infty$ sense), then the estimates can only be local in time.
\end{remark}


\subsection{Space-time discretization}
\label{subsec:gear}

\subsubsection{Scalar problem}
\label{subsubsec:gear-scal}

To simplify the presentation, for the remainder of the section we assume that
$g$ and $s$ in (\ref{eqn:react}) are both $0.$ Let $0=t_0<t_1<\cdots<t_N=T$,
with time steps $\tau_k=t_k-t_{k-1}$. We use the following notation for backward
finite difference formulas
\begin{align*}
  \partial_n^0u_h
  &=u_h^n,\\
  \partial_n^{k}u_h
  &=\frac{\partial_n^{k-1}u_h-\partial_{n-1}^{k-1}u_h}{\left(\frac{\tau_n+\ldots+\tau_{n-k+1}}{k}\right)},\quad k\geq1.
\end{align*}
We will also need the Newton polynomials of degree $1$ and $2$, given by
\begin{align*}
  &u_{h\Delta t}
  =u_h^n+(t-t_n)\partial_n^1 u_h,\quad\quad t\in[t_{n-1},t_n],\\
  &\tilde{u}_{h\Delta t}
  =u_h^n+(t-t_n)\partial_n^1 u_h +
  \frac{1}{2}(t-t_{n-1})(t-t_n)\partial_n^2 u_h,\quad\quad t\in[t_{n-2},t_n],
\end{align*}
that satisfy $u_{h\Delta t}(t_k)=u_h^k$ for $k=n-1,\,n$ and
$\tilde{u}_{h\Delta t}(t_k)=u_h^k$ for $k=n-2,\,n-1,\,n$. We define the Gear derivative
$\partial_n^Gu_h=\frac{\partial \tilde{u}_{h\Delta t}}{\partial t}(t_n)$, which
is a second-order accurate approximation of the first derivative. In practical
computation, it is a two step approximation of the following form:
\begin{align*}
  \partial_n^Gu_h
  &=\frac{1}{\tau_n}\left[\frac{1+2\gamma_n}{1+\gamma_n}u_h^n-(1+\gamma_n)u_h^{n-1}+\frac{\gamma_n^2}{1+\gamma_n}u_h^{n-2}\right],
\end{align*}
where $\gamma_n=\frac{\tau_n}{\tau_{n-1}}$ is the step-size ratio. Denote by
$\pi_h:C(\overline{\Omega})\cap H^1(\Omega)\rightarrow V_h$ the Lagrange
interpolation operator. The variable time step BDF2 method starting with
backward Euler is defined as follows: find $\{u_h^n\}_n\in V_h$ that solve
\begin{equation}\label{eq:bdf2-solve}
  \left\{\begin{array}{ll}
  u_h^0=\pi_h(u_0)\\
  \inp{\partial_1^1u_h}{v_h}{\Omega}+\inp{\nabla u_h^1}{\nabla v_h}{\Omega}
           +\inp{f(u_h^1)}{v_h}{\Omega}=0,\\
  \inp{\partial_n^Gu_h}{v_h}{\Omega}+\inp{\nabla u_h^n}{\nabla v_h}{\Omega}
  +\inp{f(u_h^n)}{v_h}{\Omega}=0,&\quad n\geq 2.
  \end{array}\right.
\end{equation}
As with constant step BDF2, the method is second-order accurate in time. It is
shown in \cite{bec98} that the method is stable for linear problems provided
$\gamma_n\leq (2+\sqrt{13})/3\approx 1.86$, and that the constant
$\Gamma_N=\sum_{n=2}^{N-2}[\gamma_n-\gamma_{n+2}]_+$ remains bounded, where
$[\cdot]_+$ denotes the non-negative part. In \cite{emm04}, this result is
extended to semilinear parabolic problems provided
$\gamma_n\leq \gamma_{max}\approx 1.910$, and provided $\tau_n\leq\tau_{max}$,
where $\tau_{max}$ depends only on the nonlinear term $f$. For the purposes of
implementing adaptive step-size control, these restrictions are not too
severe. The maximum for the step-size ratio allows the step size to increase by
an order of magnitude in three steps, a rapid transition. Furthermore, the
boundedness of $\Gamma_N$ will hold provided the variation in $\gamma_n$ is not
too erratic.


For the next lemma we need a linear reconstruction of $f(u_h)$:
\begin{align*}
  I_1(f(u_h))(t)
  &=f(u_h^n)+\frac{t-t_n}{\tau_n}(f(u_h^n)-f(u_h^{n-1})).
\end{align*}
For convenience, we will define the function
$\overline{\partial}^3u_h:[0,T]\rightarrow V_h$ to be $0$ for $t\leq t_2$ and
otherwise to be $\partial_n^3u_h$ on $(t_{n-1},t_n]$. The following lemma
extends results from \cite{akrcha10}, where only a constant step size was
considered.

\begin{lem}\label{lem-res}
  For $v_h\in V_h$ and $t\in(t_{n-1},t_n]$ with $n\geq3$
  \begin{align*}
    \inp{\derpar{\tilde{u}_{h\Delta t}}{t}}{v_h}{\Omega}
    +\inp{\nabla u_{h\Delta t}}{\nabla v_h}{\Omega}
    &=-\inp{I_1(f(u_h))}{v_h}{\Omega}-Q_n(t)\inp{\partial_n^3u_h}{v_h}{\Omega},
  \end{align*}
  where
  $Q_n(t)=\frac{\tau_{n-1}(\tau_n+\tau_{n-1}+\tau_{n-2})}{6\tau_n}(t-t_n).$
  Furthermore, if the collection $\left\{\overline{\partial}^3u_h\right\}_h$ is
  uniformly bounded in $L^2(0,T;H)$, then
  $\|Q_n(t)\partial_n^3u_h\|_{L^2(0,T;H)}$ is order $O(\tau_*^2)$ where
  $\tau_*=\max_{n\geq2}\tau_n.$
\end{lem}

\begin{proof}

  If $n\geq 3$, we can apply the two step variational equality for $u_h^{n-1}$
  and $u_h^n$ so that
  \begin{align*}
    &\inp{\nabla u_{h\Delta t}}{\nabla v_h}{\Omega}\\
    &=-\inp{I_1f(u_h)}{v_h}{\Omega}
      -\inp{\partial_n^Gu_h+\frac{t-t_n}{\tau_n}(\partial_n^Gu_h-\partial_{n-1}^Gu_h)}{v_h}{\Omega}.
  \end{align*}
  Then combined with the relation $\derpar{\tilde{u}_{h\Delta t}}{t}=\partial_n^G
  u_h+(t-t_n)\partial_n^2 u_h$
  \begin{align*}
    &\inp{\derpar{\tilde{u}_{h\Delta t}}{t}}{v_h}{\Omega} + \inp{\nabla u_{h\Delta t}}{\nabla v_h}{\Omega}\\
    &=-\inp{I_1f(u_h)}{v_h}{\Omega}
      + (t-t_n)\inp{\partial_n^2u_h-\frac{1}{\tau_n}(\partial_n^Gu_h-\partial_{n-1}^Gu_h)}{v_h}{\Omega}.
  \end{align*}
  Since
  $\partial_n^Gu_h=\partial_n^1u_h+\frac{\tau_n}{2}\partial_n^2u_h$, we get
  \begin{align*}
    \partial_n^2u_h-\frac{1}{\tau_n}(\partial_n^Gu_h-\partial_{n-1}^Gu_h)
    &=\partial_n^2u_h-\frac{1}{\tau_n}\left(\frac{\tau_n+\tau_{n-1}}{2}\partial_n^2u_h+\frac{\tau_n}{2}\partial_{n}^2u_h-\frac{\tau_{n-1}}{2}\partial_{n-1}^2u_h\right)\\
    &=-\frac{\tau_{n-1}}{2\tau_n}(\partial_{n}^2u_h-\partial_{n-1}^2u_h)\\
    &=-\frac{\tau_{n-1}(\tau_n+\tau_{n-1}+\tau_{n-2})}{6\tau_n}\partial_{n}^3u_h
  \end{align*}
  and we conclude the result. Finally, for the order of convergence we note that
  \begin{align}\label{eqn-conv-third}
    &\int_{t_{n-1}}^{t_n}Q_n(t)^2\norz{\partial_n^3u_h}{\Omega}^2\,\mathrm{d}t\nonumber\\
    &=\frac{1}{108}\tau_n\tau_{n-1}^2\left(\tau_n+\tau_{n-1}+\tau_{n-2}\right)^2\norz{\partial_n^3u_h}{\Omega}^2\nonumber\\
    &\leq C\tau_*^4\|\partial_n^3u_h\|_{L^2(t_{n-1},t_n;H)}^2.
  \end{align}

\end{proof}

\begin{remark} \normalfont In \cite{akrcha10} it is shown that the uniform
  boundedness assumption of $\overline{\partial}^3u_h$ need not hold. The
  optimality of the order of convergence of the estimator will be addressed in
  Section \ref{sec:numer-react} for numerical examples.
\end{remark}

For convenience, we will denote by $p_n=p_n(\tau_n,\tau_{n-1},\tau_{n-2})$ the
coefficient for $\norz{\partial_n^3u_h}{\Omega}^2$ appearing in the second line
of (\ref{eqn-conv-third}).


The techniques used in the proof of Theorem \ref{theorem-bdf2} are essentially
from \cite[Theorem 4.4]{lozpicpra09}, which is an estimate for the heat equation
solved with Crank-Nicolson. We include the proof for the sake of
completeness. In what follows, we denote $e_h(t)=u-u_{h\Delta t}$ and
$\tilde{e}_h(t)=u-\tilde{u}_{h\Delta t}$. For the fully discrete problem, the
element residual defined on $(t_{n-1},t_n)\times K$ is
\begin{align*}
  R^S_{K,n}(u_h)
  &=f(\tilde{u}_{h\Delta
    t})+\partial_n^Gu_h-\Delta u_{h\Delta
    t},
\end{align*}
while the edge residual defined on $(t_{n-1},t_n)\times \partial K$  is
\begin{align*}
  r^S_{K,n}(u_h)
  &=[\nabla u_{h\Delta t}].
\end{align*}

\begin{theo}\label{theorem-bdf2}
  Suppose that \emph{(\ref{eq:h1-l2})} holds. There exists a constant $C_{\hat{K}}>0$
  independent of the mesh and step size such that for $n\geq 3$
\begin{align}\label{bdf2est-full}
  &\norz{e_h(t_n)}{\Omega}^2
  + \int_{t_{n-1}}^{t_n}\noro{e_h(t)}{\Omega}^2\,\mathrm{d}t\leq \norz{e_h(t_{n-1})}{\Omega}^2\nonumber\\
  & + C_{\hat{K}}\left(\sum_K\int_{t_{n-1}}^{t_n}\left(\|R^S_{K,n}(u_h)\|_{0,K}
  +\frac{1}{2}\left(\frac{h_K}{\lambda_{1,K}\lambda_{2,K}}\right)^{1/2}\|r^S_{K,n}(u_h)\|_{0,\partial K}\right)\tilde{\omega}_K(\tilde{e}_h)\,\mathrm{d}t\right.\nonumber\\
  &\left.+ \sum_K\lambda_{2,K}^2\tau_n^3\norz{\partial_n^2u_h}{K}^2
  + p_n\norz{\partial_n^3u_h}{\Omega}^2
    + \frac{1}{120}\tau_n^5\noro{\partial_n^2u_h}{\Omega}^2\right.\nonumber\\
  &\left.+\int_{t_{n-1}}^{t_n}\|f(\tilde{u}_{h\Delta
    t})-I_1(f(u_h))\|_{0,\Omega}^2\,\mathrm{d}t\right).
\end{align}
\end{theo}

\begin{proof}

For $v\in V$, using the variational formulation for $u$
\begin{align}\label{eq-gear1}
  &\inp{\derpar{\tilde{e}_{h}}{t}}{v}{\Omega} + \inp{\nabla e_h}{\nabla v}{\Omega}\nonumber\\
  &=-\inp{\derpar{\tilde{u}_{h\Delta t}}{t}}{v}{\Omega} - \inp{\nabla u_{h\Delta
      t}}{\nabla v}{\Omega} - \inp{f(u)}{v}{\Omega}\nonumber\\
  &=-\inp{\derpar{\tilde{u}_{h\Delta t}}{t}+f(\tilde{u}_{h\Delta t})}{v}{\Omega}
  - \inp{\nabla u_{h\Delta t}}{\nabla v}{\Omega}
  - \inp{f(u)-f(\tilde{u}_{h\Delta t})}{v}{\Omega}.
\end{align}
For $v_h\in V_{h},$ we apply Lemma \ref{lem-res} to conclude
\begin{align}\label{eq-integ-full}
  &\inp{\derpar{\tilde{e}_{h}}{t}}{v}{\Omega} + \inp{\nabla e_h}{\nabla
    v}{\Omega}\nonumber\\
  &=-\inp{\derpar{\tilde{u}_{h\Delta t}}{t} + f(\tilde{u}_{h\Delta t})}{v-v_h}{\Omega} - \inp{\nabla u_{h\Delta
    t}}{\nabla
    (v-v_h)}{\Omega}+Q_n(t)\inp{\partial_n^3u_h}{v_h}{\Omega}\nonumber\\
  &-\inp{f(u)-f(\tilde{u}_{h\Delta t})}{v}{\Omega}-\inp{f(\tilde{u}_{h\Delta
    t})-I_1(f(u_h))}{v_h}{\Omega}.
\end{align}
Choose $v=\tilde{e}_h,\,v_h=I_h(\tilde{e}_h).$ Note that
\begin{align}\label{eq-pars}
  \inp{\nabla e_h}{\nabla\tilde{e}_h}{\Omega}
  &=\frac{1}{2}\noro{e_h}{\Omega}^2+\frac{1}{2}\noro{\tilde{e}_h}{\Omega}^2 -
    \frac{1}{2}\noro{e_{h} - \tilde{e}_{h}}{\Omega}^2\nonumber\\
  &=\frac{1}{2}\noro{e_h}{\Omega}^2+\frac{1}{2}\noro{\tilde{e}_h}{\Omega}^2 - \frac{1}{2}\noro{\tilde{u}_{h\Delta t} - u_{h\Delta t}}{\Omega}^2\nonumber\\
  &=\frac{1}{2}\noro{e_h}{\Omega}^2+\frac{1}{2}\noro{\tilde{e}_h}{\Omega}^2 -
    \frac{1}{4}(t-t_n)^2(t-t_{n-1})^2\noro{\partial_n^2u_h}{\Omega}^2,
\end{align}
so substituting (\ref{eq-pars}) into (\ref{eq-integ-full}), applying integration
by parts over each element $K,$ and integrating from $t_{n-1}$ to $t_n$, and
applying Cauchy-Bunyakowsky-Schwartz we get
\begin{align}\label{ineq-var-full}
  &\frac{1}{2}\norz{e_h(t_n)}{\Omega}^2-\frac{1}{2}\norz{e_h(t_{n-1})}{\Omega}^2
    + \int_{t_{n-1}}^{t_n}\frac{1}{2}\left(\noro{e_h(t)}{\Omega}^2+\noro{\tilde{e}_h(t)}{\Omega}^2\right)\,\mathrm{d}t\nonumber\\
  &\leq\underbrace{\sum_K\int_{t_{n-1}}^{t_n}\norz{f(\tilde{u}_{h\Delta
    t})+\partial_n^Gu_h-\Delta u_{h\Delta
    t}}{K}\norz{\tilde{e}_h-I_h(\tilde{e}_h)}{K}\,\mathrm{d}t}_\text{I}\nonumber\\
  &+ \sum_K\int_{t_{n-1}}^{t_n}\underbrace{\frac{1}{2}\norz{[\nabla u_{h\Delta
    t}]}{\partial K}\norz{\tilde{e}_h-I_h(\tilde{e}_h)}{\partial
    K}\,\mathrm{d}t}_\text{II} + \frac{1}{120}\tau_n^5\noro{\partial_n^2u_h}{\Omega}^2\nonumber\\
  &+\underbrace{\sum_K\int_{t_{n-1}}^{t_n}|t-t_n|\norz{\partial_n^2u_h}{K}\norz{\tilde{e}_h-I_h(\tilde{e}_h)}{K}\,\mathrm{d}t}_\text{III}
    \nonumber\\
  &+\underbrace{\int_{t_{n-1}}^{t_n}\norz{Q_n(t)\partial_n^3u_h}{\Omega}\norz{I_h(\tilde{e}_h)}{\Omega}\,\mathrm{d}t}_\text{IV}\nonumber\\
  &-\underbrace{\int_{t_{n-1}}^{t_n}\inp{f(u)-f(\tilde{u}_{h\Delta t})}{\tilde{e}_h}{\Omega}\,\mathrm{d}t}_\text{V}-\underbrace{\int_{t_{n-1}}^{t_n}\inp{f(\tilde{u}_{h\Delta
    t})-I_1(f(u_h))}{I_h(\tilde{e}_h)}{\Omega}\,\mathrm{d}t}_\text{VI}
\end{align}
(Note that $e_h=\tilde{e}_h$ at $t=t_n,\,t_{n-1}$).

We will deal with each term of (\ref{ineq-var-full}). Applying the interpolation
estimates we get
\begin{align*}
  \text{I}+\text{II}
  &\leq \int_{t_{n-1}}^{t_n}C_{\hat{K}}\sum_K\left(\|R^S_{K,n}(u_h)\|_{0,K}
    +\frac{1}{2}\left(\frac{h_K}{\lambda_{1,K}\lambda_{2,K}}\right)^{1/2}\|r^S_{K,n}(u_h)\|_{0,\partial K}\right)\tilde{\omega}_K(\tilde{e}_h)\,\mathrm{d}t.
\end{align*}
As in \cite{lozpicpra09}, we assume there exists $C_{eq}>0$ independent of the
mesh such that
$\tilde{\omega}_K(\tilde{e}_h)\leq C_{eq}\lambda_{2,K}|\tilde{e}_h|_{1,\Omega}$. Then applying the interpolation estimates and Young's
inequality, for any $\gamma>0$
\begin{align*}
  \text{III}
  &\leq \sum_K
    \frac{\gamma}{2}\int_{t_{n-1}}^{t_n}\lambda_{2,K}^2(t-t_n)^2\norz{\partial_n^2u_h}{K}^2\,\mathrm{d}t
    +\frac{C_{\hat{K}}^2C_{eq}^2}{2\gamma}\int_{t_{n-1}}^{t_n}\noro{\tilde{e}_h}{\Omega}^2\,\mathrm{d}t\\
  &\leq \sum_K
    \frac{\gamma}{6}\lambda_{2,K}^2\tau_n^3\norz{\partial_n^2u_h}{K}^2
    +\frac{C_{\hat{K}}^2C_{eq}^2}{2\gamma}\int_{t_{n-1}}^{t_n}\noro{\tilde{e}_h}{\Omega}^2\,\mathrm{d}t.
\end{align*}
Let $C_{I_h}>0$ be a constant such that for all $v\in H^1(\Omega),$
$\|I_h(v)\|_{0,\Omega}\leq C_{I_h}\|v\|_{0,\Omega}$. Then
\begin{align*}
  \text{IV}
  &\leq
    \frac{\gamma}{2}p_n\norz{\partial_n^3u_h}{\Omega}^2
    +\frac{C_{I_h}^2}{2\gamma}\int_{t_{n-1}}^{t_n}\norz{\tilde{e}_h}{\Omega}^2\,\mathrm{d}t.
\end{align*}
Assuming \ref{cond1} we get
\begin{align*}
  \text{V}
  &\leq\alpha\int_{t_{n-1}}^{t_n}\norz{\tilde{e}_h}{\Omega}^2\,\mathrm{d}t.
\end{align*}
Finally,
\begin{align*}
  \text{VI}
  &\leq \frac{\gamma}{2}\int_{t_{n-1}}^{t_n}\|f(\tilde{u}_{h\Delta
    t})-I_1(f(u_h))\|_{0,\Omega}^2\,\mathrm{d}t+\frac{C_{I_h}^2}{2\gamma}\int_{t_{n-1}}^{t_n}\norz{\tilde{e}_h}{\Omega}^2\,\mathrm{d}t\\
\end{align*}
If we apply assumption (\ref{eq:h1-l2}) and collect the coefficients to
$\int_{t_{n-1}}^{t_n}\noro{\tilde{e}_h}{\Omega}^2\,\mathrm{d}t$ on the left
side we have
\begin{align*}
  \frac{1}{2}-\frac{C_{\hat{K}}^2C_{eq}^2}{2\gamma}-\frac{C_{AN}}{N_T}\left(\alpha+\frac{C_{I_h}^2}{2\gamma}\right).
\end{align*}
This can be made non-negative, for instance choosing
$\gamma\geq\max\{2C_{\hat{K}}^2C_{eq}^2,C_{I_h}^2\}$ and provided
$N_T\geq 4C_{AN}(\alpha+\frac{1}{2})$. The proof is then completed by gathering
all terms in (\ref{ineq-var-full}).

\end{proof}


Note that the right hand side of the inequality in Theorem \ref{theorem-bdf2}
still depends on the unknown $\nabla u$. However, as is done for instance in
\cite{pic03a}, we let $\Pi_h:V_h\rightarrow V_h\times V_h$ denote a gradient
recovery operator. In general, we expect $\Pi_h(u_h)$ to converge faster than
$\nabla u_h$. In this paper, we found it sufficed to apply the simplified
Zienkiewicz-Zhu operator studied in \cite{rod94}.  We have the following
estimator for the error in space
\begin{align}
  \eta_{K,n}^{S}
  &=\left(\int_{t_{n-1}}^{t_n}\left(\|R^S_{K,n}(u_h)\|_{0,K}
  +\frac{1}{2}\left(\frac{h_K}{\lambda_{1,K}\lambda_{2,K}}\right)^{1/2}\|r^S_{K,n}(u_h)\|_{0,\partial K}\right)\omega_K(\tilde{u}_{h\Delta
    t})\,\mathrm{d}t\right)^{1/2},\label{bdf-est-space}
\end{align}
where for $v_h\in V_h$
\begin{align}
  \omega_K(v_h)&=\left(\lambda_{1,K}^2r_{1,K}^T
                 G_K(v_h)r_{1,K}+\lambda_{2,K}^2r_{2,K}^T
                 G_K(v_h)r_{2,K}\right)^{1/2},\label{omega-recovered}\\
  G_K(v_h)&=
  \left(\int_{\Delta_K}\left(\frac{\partial v_h}{\partial x_i}-\Pi_h(v_h)_i\right)\left(\frac{\partial v_h}{\partial
    x_j}-\Pi_h(v_h)_j\right)\,\,\mathrm{d}x\right)_{i,j}.
\end{align}
The estimator for the error in time is given by
\begin{align}\label{bdf-est-time}
  \eta_n^{T}
  &=\left((\eta_n^{1,T})^2+(\eta_n^{2,T})^2+(\eta_n^{3,T})^2+(\eta_n^{4,T})^2\right)^{1/2},
\end{align}
where
\begin{align}\label{bdf-est-time-terms}
  \eta_n^{1,T}&=\left(\frac{1}{120}\tau_n^5\noro{\partial_n^2u_h}{\Omega}^2\right)^{1/2},
                \quad\quad\eta_n^{2,T}=\left(\frac{1}{12}\sum_K\lambda_{2,K}^2\tau_n^3\norz{\partial_n^2u_h}{K}^2\right)^{1/2},\nonumber\\
  \eta_n^{3,T}&=\left(p_n\norz{\partial_n^3u_h}{\Omega}^2\right)^{1/2},
                \quad\quad\eta_n^{4,T}=\left(\int_{t_{n-1}}^{t_n}\|f(\tilde{u}_{h\Delta
                t})-I_1(f(u_h))\|_{0,\Omega}^2\,\mathrm{d}t\right)^{1/2}.
\end{align}


\subsubsection{Monodomain problem}
\label{sec:mono-est}

We retain the notation from Section \ref{subsubsec:gear-scal}. For $n\geq2$, the
fully discrete approximations $u_h^n,\,w_h^n$ to (\ref{eqn:mono}) are obtained
by solving the variational problem
\begin{equation}\label{eq:beul-mono-solve}
  \left\{
    \begin{array}{ll}
      \inp{\partial_n^G u_h}{\phi_h}{\Omega}+\inp{\nabla u_h^n}{\nabla \phi_h}{\Omega}
      +\inp{F(u_h^n,w_h^n)}{\phi_h}{\Omega}=0,\\
      \inp{\partial_n^G w_h}{\psi_h}{\Omega}
      +\inp{G(u_h^n,w_h^n)}{\psi_h}{\Omega}=0,
    \end{array}\right.
\end{equation}
for all $\phi_h,\,\psi_h\in V_h$. The solution for $n=1$ is obtained by the
backward Euler method.

Define the element residuals
\begin{align*}
  R_{K,n}^U
  &=\derpar{\tilde{u}_{h\Delta t}}{t} - \Delta u_{h\Delta t} + F(\tilde{u}_h,\tilde{w}_h),\\
  R_{K,n}^W
  &=\derpar{\tilde{w}_{h\Delta t}}{t} + G(\tilde{u}_h,\tilde{w}_h),
\end{align*}
and edge residual
\begin{equation*}
  r_{K,n}^U=[\nabla u_{h\Delta t}].
\end{equation*}
Define $e_u=u-u_{h\Delta t},$ $\tilde{e}_u=u-\tilde{u}_{h\Delta t}$ and
similarly define $e_w$ and $\tilde{e}_w$. Define the local space estimators
\begin{align}
  \tilde{\eta}_{K,n}^{S,U}(e_u)
  &=\left(\int_{t_{n-1}}^{t_n}\left(\|R_{K,n}^U\|_{0,K} +
    \left(\frac{h_K}{\lambda_{1,K}\lambda_{2,K}}\right)^{1/2}\|r_{K,n}^U\|_{0,\partial
    K}\right)\tilde{\omega}_K(\tilde{e}_u)
    \,\mathrm{d}t\right)^{1/2},\nonumber\\
  \tilde{\eta}_{n,K}^{S,W}(e_w)
  &=\left(\int_{t_{n-1}}^{t_n}\|R_{K,n}^W\|_{0,K}\tilde{\omega}_K(\tilde{e}_w)
    \,\mathrm{d}t\right)^{1/2}\label{eta-mono-tilde-K},
\end{align}
and global space estimators
\begin{align}\label{eta-mono-tilde}
  \tilde{\eta}_{n}^{S,U}
  =\left(\sum_K(\tilde{\eta}_{K,n}^{S,U})^2\right)^{1/2},
  \quad\quad
  \tilde{\eta}_{n}^{S,W}
  =\left(\sum_K(\tilde{\eta}_{K,n}^{S,W})^2\right)^{1/2}.
\end{align}
Define the time estimators
\begin{align}\label{eta-mono-terms}
  \eta_n^{T,U}
  &=\left((\eta_n^{1,T,U})^2 + (\eta_n^{3,T,U})^2 + (\eta_n^{4,T,U})^2\right)^{1/2},\nonumber\\
  \eta_n^{T,W}
  &=\left((\eta_n^{3,T,W})^2 + (\eta_n^{4,T,W})^2\right)^{1/2},
\end{align}
where the terms appearing in (\ref{eta-mono-terms}) are defined analogous to
those in (\ref{bdf-est-time-terms}). Note that there is no corresponding
estimator for $\eta^{1,T,U}$ for $w$ since there is no Laplacian operator. Also,
note that in (\ref{eta-mono-terms}) there are no terms corresponding to
$\eta^{2,T}$ from (\ref{bdf-est-time-terms}). Here we use the time derivatives
of the quadratic reconstructions $\derpar{\tilde{u}_h}{t}$ and
$\derpar{\tilde{w}_h}{t}$ in the element residuals, instead of retaining only
the constant part $\partial_n^G u_h$ and $\partial_n^G w_h$ as is done in
Section \ref{subsubsec:gear-scal}. The following lemma is an easy extension of
Lemma \ref{lem-res}.

\begin{lem}\label{lem-bdf-mono}
  For $v_h\in V_h$ and $t\in(t_{n-1},t_n]$ with $n\geq3$
  \begin{align*}
    \inp{\derpar{\tilde{u}_{h\Delta t}}{t}}{v_h}{\Omega}
    +\inp{\nabla u_{h\Delta t}}{\nabla v_h}{\Omega}
    &=-\inp{I_1(F(u_h,w_h))}{v_h}{\Omega}-Q_n(t)\inp{\partial_n^3u_h}{v_h}{\Omega},\\
    \inp{\derpar{\tilde{w}_{h\Delta t}}{t}}{v_h}{\Omega}
    &=-\inp{I_1(G(u_h,w_h))}{v_h}{\Omega}-Q_n(t)\inp{\partial_n^3w_h}{v_h}{\Omega},
  \end{align*}
  where
  $Q_n(t)=\frac{\tau_{n-1}(\tau_n+\tau_{n-1}+\tau_{n-2})}{6\tau_n}(t-t_n).$
\end{lem}

\begin{theo}\label{theorem-bdf2-mono}
  Suppose that \emph{\ref{mono-cond1}} or \emph{\ref{mono-cond2}} is satisfied. There
  exists a constant $C_{\hat{K}}>0$ independent of the mesh and the step size
  and $C(\alpha)>0$ depending linearly on $\alpha$ such that, for $n\geq3$,
  \begin{align}\label{bdf2est-full-mono}
    &\|\tilde{e}_u\|_{L^\infty(t_{n-1},t_n;H)}^2 + \|\tilde{e}_w\|_{L^\infty(t_{n-1},t_n;H)}^2
      + \vertiii{e_u}_n^2\nonumber\\
    &\leq C_{\hat{K}}e^{C(\alpha)\tau_n}\left(\norz{e_u(t_{n-1})}{\Omega}^2 + \norz{e_w(t_{n-1})}{\Omega}^2 +
      (\tilde{\eta}_n^{S,U})^2 + (\tilde{\eta}_n^{S,W})^2 + (\eta_n^{T,U})^2 + (\eta_n^{T,W})^2 \right).
  \end{align}
  If \emph{\ref{mono-cond3}} is satisfied, then there exists $C_{AN}>0$ depending on
  the superconvergence assumption \emph{(\ref{eq:h1-l2})} such that, for
  $n\geq3$,
  \begin{align}\label{bdf2est-full-mono-opt}
    &\|\tilde{e}_u\|_{L^\infty(t_{n-1},t_n;H)}^2 + \|\tilde{e}_w\|_{L^\infty(t_{n-1},t_n;H)}^2
      + \vertiii{e_u}_n^2
      + \frac{\beta}{2}\|\tilde{e}_w\|_{L^2(t_{n-1},t_n;H)}^2\nonumber\\
    &\leq C_{\hat{K}}\left(\norz{e_u(t_{n-1})}{\Omega}^2 + \norz{e_w(t_{n-1})}{\Omega}^2 +
      (\tilde{\eta}_n^{S,U})^2 + (\tilde{\eta}_n^{S,W})^2 + (\eta_n^{T,U})^2 + (\eta_n^{T,W})^2 \right).
  \end{align}
\end{theo}

\begin{proof}

  By inspecting the proof of Theorem \ref{theorem-bdf2} up to
  (\ref{eq-integ-full}), we get for the first equation
  \begin{align}\label{eq-integ-full-mono1}
    &\inp{\derpar{\tilde{e}_{u}}{t}}{\tilde{e}_{u}}{\Omega} + \inp{\nabla e_u}{\nabla
      \tilde{e}_{u}}{\Omega}\nonumber\\
    &=-\inp{\derpar{\tilde{u}_{h\Delta t}}{t} + F(\tilde{u}_{h\Delta
      t},\tilde{w}_{h\Delta t})}{\tilde{e}_{u}-I_h(\tilde{e}_{u})}{\Omega}
      - \inp{\nabla u_{h\Delta t}}{\nabla(\tilde{e}_{u}-I_h(\tilde{e}_{u}))}{\Omega}\nonumber\\
    &-Q_n(t)\inp{\partial_n^3u_h}{I_h(\tilde{e}_{u})}{\Omega}
      -\inp{F(\tilde{u}_{h\Delta t},\tilde{w}_{h\Delta t})-I_1(F(u_h,w_h))}{I_h(\tilde{e}_{u})}{\Omega}\nonumber\\
    &-\inp{F(u,w)-F(\tilde{u}_{h\Delta t},\tilde{w}_{h\Delta t})}{\tilde{e}_{u}}{\Omega}.
  \end{align}
  Therefore, applying Young's inequality and (\ref{eq-pars}), there exist
  positive constants $C_1,\,C_2$ depending on the interpolation operator $I_h$
  such that
  \begin{align}
    \frac{d}{dt}\norz{\tilde{e}_u}{\Omega}^2 +
    & \noro{e_u}{\Omega}^2 + \noro{\tilde{e}_u}{\Omega}^2\nonumber\\
    &\leq C_1\left(\sum_K\left(\|R_{K,n}^U\|_{0,K} +
      \left(\frac{h_K}{\lambda_{1,K}\lambda_{2,K}}\right)^{1/2}\|r_{K,n}^U\|_{0,\partial
      K}\right)\tilde{\omega}_K(\tilde{e}_u)\right.\nonumber\\
    & + \frac{1}{2}(t-t_n)^2(t-t_{n-1})^2\noro{\partial_n^2u_h}{\Omega}^2
      + \norz{Q_n(t)\partial_n^3u_h}{\Omega}^2\nonumber\\
    &  + \|F(\tilde{u}_{h\Delta t},\tilde{w}_{h\Delta t})-I_1(F(u_h,w_h))\|_{0,\Omega}^2\Bigg)\nonumber\\
    & + C_2 \|\tilde{e}_u\|_{0,\Omega}^2 - \inp{F(u,w)-F(\tilde{u}_{h\Delta t},\tilde{w}_{h\Delta t})}{\tilde{e}_u}{\Omega}\nonumber\\
    & = S_1(t)+ C_2 \|\tilde{e}_u\|_{0,\Omega}^2 - \inp{F(u,w)-F(\tilde{u}_{h\Delta
      t},\tilde{w}_{h\Delta t})}{\tilde{e}_u}{\Omega},\label{ineq:eq1-deriv}
  \end{align}
  where $S_1(t)$ contains the remainder of the terms. Similarly, for the second
  equation we get
  \begin{align}\label{eq-integ-full-mono2}
    \inp{\derpar{\tilde{e}_{w}}{t}}{\tilde{e}_{w}}{\Omega}
    &=-\inp{\derpar{\tilde{w}_{h\Delta t}}{t} + G(\tilde{u}_{h\Delta t},\tilde{w}_{h\Delta t})}{\tilde{e}_{w}-I_h(\tilde{e}_{w})}{\Omega}\nonumber\\
    &-Q_n(t)\inp{\partial_n^3w_h}{I_h(\tilde{e}_{w})}{\Omega}
      -\inp{G(\tilde{u}_{h\Delta t},\tilde{w}_{h\Delta t})-I_1(G(u_h,w_h))}{I_h(\tilde{e}_{w})}{\Omega}\nonumber\\
    &-\inp{G(u,w)-G(\tilde{u}_{h\Delta t},\tilde{w}_{h\Delta t})}{\tilde{e}_{w}}{\Omega}.
  \end{align}
  so there exist positive constants $C_3,\,C_4$, also depending on $I_h$, such that
  \begin{align}
    \frac{d}{dt}\norz{\tilde{e}_w}{\Omega}^2
    &\leq C_3\left(\sum_K\|R_{K,n}^W\|_{0,K}\omega_K(\tilde{e}_w)
      + \norz{Q_n(t)\partial_n^3w_h}{\Omega}^2\right.\nonumber\\
    & + \|G(\tilde{u}_{h\Delta t},\tilde{w}_{h\Delta t})-I_1(G(u_h,w_h))\|_{0,\Omega}^2\Bigg)\nonumber\\
    & + C_4 \|\tilde{e}_w\|_{0,\Omega}^2 
      - \inp{G(u,w)-G(\tilde{u}_{h\Delta t},\tilde{w}_{h\Delta t})}{\tilde{e}_w}{\Omega}.\nonumber\\
    & = S_2(t) + C_4 \|\tilde{e}_w\|_{0,\Omega}^2 
      - \inp{G(u,w)-G(\tilde{u}_{h\Delta t},\tilde{w}_{h\Delta t})}{\tilde{e}_w}{\Omega}.\label{ineq:eq2-deriv}
  \end{align}
  If we assume either
  \ref{mono-cond1} or \ref{mono-cond2}, then
  \begin{align}
    &C_2\|\tilde{e}_u\|_{0,\Omega}^2 + C_4\|\tilde{e}_w\|_{0,\Omega}^2\nonumber\\
    &- \inp{F(u,w)-F(\tilde{u}_{h\Delta t},\tilde{w}_{h\Delta t})}{\tilde{e}_u}{\Omega}
      - \inp{G(u,w)-G(\tilde{u}_{h\Delta t},\tilde{w}_{h\Delta t})}{\tilde{e}_w}{\Omega}\nonumber\\
    &\leq (C(\alpha))(\|\tilde{e}_u\|_{0,\Omega}^2 + \|\tilde{e}_w\|_{0,\Omega}^2),\label{ineq:applyconds}
  \end{align}
  where $C(\alpha)=\max\{C_2,C_4\}+\alpha.$ Taking the sum of
  (\ref{ineq:eq1-deriv}) and (\ref{ineq:eq2-deriv}), and applying Gronwall's
  inequality we obtain for a.e. $t\in (t_{n-1},t_n]$
  \begin{align*}
    &\norz{\tilde{e}_u(t)}{\Omega}^2 + \norz{\tilde{e}_w(t)}{\Omega}^2 + \int_{t_{n-1}}^te^{C(\alpha)(t-s)}\noro{e_u(s)}{\Omega}^2\,\mathrm{d}s\\
    &\leq e^{C(\alpha)(t-t_{n-1})}\left(\norz{\tilde{e}_u(t_{n-1})}{\Omega}^2 +
      \norz{\tilde{e}_w(t_{n-1})}{\Omega}^2\right)
      + \intT{t_{n-1}}{t}{e^{C(\alpha)(t-s)}(S_1(s)+S_2(s))}{s}\\
    &\leq e^{C(\alpha)\tau_n}\left(\norz{\tilde{e}_u(t_{n-1})}{\Omega}^2 +
      \norz{\tilde{e}_w(t_{n-1})}{\Omega}^2\right)
      + \intT{t_{n-1}}{t_n}{e^{C(\alpha)(t_{n}-t)}(S_1(t)+S_2(t))}{t}.
  \end{align*}
  Since $t$ is arbitrary, we have
  \begin{align*}
    &\|\tilde{e}_u\|_{L^\infty(t_{n-1},t_n;H)}^2 + \|\tilde{e}_w\|_{L^\infty(t_{n-1},t_n;H)}^2
      + \intT{t_{n-1}}{t_n}{e^{C(\alpha)(t_n-t)}\noro{e_h(t)}{\Omega}^2}{t}\\
    &\leq 2e^{C(\alpha)\tau_n}\left(\norz{\tilde{e}_u(t_{n-1})}{\Omega}^2 +
      \norz{\tilde{e}_w(t_{n-1})}{\Omega}^2\right)
      + 2\intT{t_{n-1}}{t_n}{e^{C(\alpha)(t_n-t)}(S_1(t)+S_2(t))}{t},
  \end{align*}
  and conclude (\ref{bdf2est-full-mono}) using the fact that $1\leq
  e^{C(\alpha)(t_n-t)}\leq e^{C(\alpha)\tau_n}$. On the other hand, suppose
  \ref{mono-cond3} holds. Then instead of (\ref{ineq:applyconds}) we have
  \begin{align}
    &C_2\|\tilde{e}_u\|_{0,\Omega}^2 + C_4\|\tilde{e}_w\|_{0,\Omega}^2\nonumber\\
    &- \inp{F(u,w)-F(\tilde{u}_{h\Delta t},\tilde{w}_{h\Delta t})}{\tilde{e}_u}{\Omega}
      - \inp{G(u,w)-G(\tilde{u}_{h\Delta t},\tilde{w}_{h\Delta t})}{\tilde{e}_w}{\Omega}\nonumber\\
    &\leq (C_2+\alpha)\|\tilde{e}_u\|_{0,\Omega}^2 +
      (C_4-\beta)\|\tilde{e}_w\|_{0,\Omega}^2.
      \label{eq-integ-opt}
  \end{align}
  The first term on the right of (\ref{eq-integ-opt}) can be dealt with the same
  way as in the proof of Theorem \ref{theorem-bdf2}. Moreover, from the proof we
  note that the constant $C_4$ can be made arbitrarily small at the cost of
  making the constant $C_3$ larger. In particular, we can take
  $C_4<\frac{\beta}{2}$. Then (\ref{bdf2est-full-mono-opt}) follows after
  applying (\ref{eq:h1-l2}) and integrating over $t$.
\end{proof}

As was done for the scalar problem, we remove the dependency on the exact
solution by replacing $\nabla u,\nabla w$ in the estimators
$\tilde{\eta}^{S,U},\tilde{\eta}^{S,W}$ with the recovered gradients
$\Pi_h(u_h),\Pi_h(w_h)$ respectively, and denote the resulting estimators by
$\eta^{S,U},\eta^{S,W}$.

In what follows, we look at specific ionic models. For convenience, we have the
following lemma.
\begin{lem}\label{lem:eigs}
  Suppose there exist positive constants $\mu_1,\mu_2,\,\mu_3$ such the following
  inequalities hold uniformly on the domain $\mathcal{D}$:
  \begin{align}
    \derpar{F}{x_1}\geq-\mu_1,\quad
    \left|\derpar{F}{x_2}\right|+\left|\derpar{G}{x_1}\right|\leq\mu_2,\quad
    \derpar{G}{x_2}\geq\mu_3.\label{eqn:der-bounds}
  \end{align}
  Then \emph{\ref{mono-cond3}} holds with $\alpha=\mu_1+\frac{\mu_2^2}{2\mu_3}$, and
  $\beta=\frac{\mu_3}{2}$.
\end{lem}

\begin{proof}

  For any $x,\,y\in\mathcal{D}$, by the mean value theorem there exists
  $\xi_1,\,\xi_2$ on the line segment between $x$ and $y$ such that
  \begin{align*}
    &(F(x)-F(y))(x_1-y_1)+(G(x)-G(y))(x_2-y_2)\\
    &=\derpar{F}{x_1}(\xi_1)(x_1-y_1)^2
      + \left(\derpar{F}{x_2}(\xi_1)+\derpar{G}{x_1}(\xi_2)\right)(x_1-y_1)(x_2-y_2)
      + \derpar{G}{x_2}(\xi_2)(x_2-y_2)^2.
  \end{align*}
  Applying (\ref{eqn:der-bounds}) and Young's inequality, for arbitrary $\gamma>0$
  \begin{align*}
    &(F(x)-F(y))(x_1-y_1)+(G(x)-G(y))(x_2-y_2)\\
    &\geq-\mu_1(x_1-y_1)^2
      - \mu_2|x_1-y_1||x_2-y_2|
      + \mu_3(x_2-y_2)^2\\
    &\geq-\left(\mu_1+\frac{\mu_2^2}{2\gamma}\right)(x_1-y_1)^2
      + \left(\mu_3-\frac{\gamma}{2}\right)(x_2-y_2)^2.
  \end{align*}
  The result follows choosing $\gamma=\mu_3$.
\end{proof}


\paragraph{FitzHugh-Nagumo model}
\label{subsec:FHN}

\begin{align}
  F(u,w)&=f_1(u)+w=u(u-a)(u-1)+w,\\
  G(u,w)&=-\epsilon(\kappa u-w),\label{FHN-recov-eq}
\end{align}
and $0<a<1,\,\epsilon,\,\kappa>0.$ It is clear that $F,\,G$ are locally
Lipschitz continuous. Invariant rectangles of arbitrary size exist for the
model, see \cite{Smoller-shock}, \cite{CP06}, so the solution remains bounded if
$u_0,\,w_0\in L^\infty(\Omega)$. Furthermore, letting $\mu>0$ be such that
$f_1'(x)\geq-\mu,$ then applying Lemma \ref{lem:eigs} we obtain
\ref{mono-cond3} with $\alpha=\mu+\frac{(1+\epsilon\kappa)^2}{2\epsilon}$ and
$\beta=\frac{\epsilon}{2}.$


\paragraph{Mitchell-Schaeffer model}
\label{subsec:mit-scha}

\begin{align}
  F(u,w) &= \frac{1}{\tau_{in}}wu^2(u-1)+\frac{1}{\tau_{out}}u\\
  G(u,w) &=
           \left\{
           \begin{array}{ll}
             \frac{1}{\tau_{open}}(w-1),&u<u_{gate},\\
             \frac{1}{\tau_{close}}w,&u\geq u_{gate},
           \end{array}\right.
\end{align}
where $\tau_{in},\,\tau_{out},\,\tau_{open},\,\tau_{close},\,u_{gate}$ are
positive constants, with $0<u_{gate}<1.$ The reaction term $G$ is discontinuous
on the line $u=u_{gate},$ and $F,\,G$ do not satisfy \ref{mono-cond1} or
\ref{mono-cond2} for a domain $\mathcal{D}$ crossing this line. We will use
the following regularized version:
\begin{align*}
  G(u,w)
  &=\frac{1}{\tau_u}((1-s)(w-1)+sw),\\
  \tau_u
  &=\tau_{open}+(\tau_{close}-\tau_{open})s,\\
  s(u,\kappa,u_{gate})
  &=\frac{1}{2}(1+\tanh(\kappa(u-u_{gate}))),
\end{align*}
where $\kappa>0.$ Both $F,\,G$ are now $C^1$ so that \ref{mono-cond2} holds.
The solutions remain bounded provided the initial conditions are
bounded. Moreover, applying the maximum principle from \cite{CP06}, for
arbitrary $\epsilon>0$, the region
$\{(u,\,w):-\epsilon\leq u\leq 1,\,0\leq w\leq 1\}$ is invariant, and we
conclude that the region $\{(u,\,w):0\leq u\leq 1,\,0\leq w\leq 1\}$ is
invariant as well.

We now show that \ref{mono-cond3} also holds. For the first reaction term,
\begin{align*}
  \derpar{F}{x_1}
  &=\frac{1}{\tau_{in}}x_2(3x_1^2-2x_1)+\frac{1}{\tau_{out}}.
\end{align*}
For a typical heartbeat, $(u,w)$ is in $[0,1]^2$ so we may take
$\mu_1=\frac{1}{3\tau_{in}} - \frac{1}{\tau_{out}}$ from the bound on $w.$ From
the bound on $u$ we get
\begin{align*}
  \left|\derpar{F}{x_2}\right|
  &\leq\max_{x_1\in[0,1]}\frac{1}{\tau_{in}}|x_1^2(x_1-1)|=\frac{1}{\tau_{in}}\frac{4}{27}.
\end{align*}
For the second term,
\begin{align*}
  \derpar{G}{x_2}
  &=\frac{1}{\tau_u}\geq\min\left\{\frac{1}{\tau_{open}},\frac{1}{\tau_{close}}\right\}.
\end{align*}
Finally,
\begin{align*}
  \derpar{G}{x_1}
  &=\kappa\,\sech^2(\kappa(x_1-u_{gate}))\frac{1}{\tau_u^2}\left(\tau_u-(x_2+s-1)(\tau_{close}-\tau_{open})\right),
\end{align*}
so assuming that $x_2$ remains in $[0,1]$,
\begin{align*}
  \left|\derpar{G}{x_1}\right|
  &\leq \kappa\frac{\tau_u+|\tau_{close}-\tau_{open}|}{\tau_u^2}\\
  &\leq \kappa\frac{\tau_{max}+|\tau_{close}-\tau_{open}|}{\tau_{min}^2},
\end{align*}
where $\tau_{min,max}=\min,\max\{\tau_{open},\tau_{close}\}$. To summarize, from
Lemma \ref{lem:eigs} we obtain \ref{mono-cond3} with
$\alpha=\frac{1}{3}\tau_{in}^{-1} - \tau_{out}^{-1} +
\frac{1}{2}\left(\frac{4}{27}\tau_{in}^{-1} +
  \kappa\frac{\tau_{max}+|\tau_{close}-\tau_{open}|}{\tau_{min}^2}\right)^2\tau_{max}$
and $\beta=\frac{1}{2\tau_{max}}.$ Under the assumption that
$\tau_{in}\ll\tau_{out}\ll\tau_{open},\tau_{close},$ we have
$\alpha=O\left(\tau_{in}^{-2}+\kappa^2\frac{\tau_{max}^3}{\tau_{min}^2}\right)$.
Note that when taking the limit $\kappa\rightarrow\infty$ to approach the
original discontinuous model, the constant $\alpha$ blows up. The proof above
relies on a uniform bound for $\derpar{G}{x_1}$, which clearly does not exist
for the discontinuous model, and therefore a different approach would have to be
taken.


\subsubsection{A modified estimator for the recovery variable}
\label{subsec:recov-mod}

We discuss some technical considerations for the use of the estimators
introduced in Theorem \ref{theorem-bdf2-mono} for mesh adaptation. A shortcoming
of the estimator is that it does not take into account the interpolation error
that occurs when the mesh is adapted, but implicitly assumes that the mesh is
left unchanged throughout. In what follows, we offer some heuristic arguments
for dealing with the interpolation error, in particular as it applies to the
second variable $w$. To simplify the discussion, we consider a scalar ODE model:
\begin{equation}\label{ode-lin}
  \left\{
    \begin{array}{ll}
      \dfrac{\partial w}{\partial t}+\mu w
      =f,\\
      w(0)=w_0.
    \end{array}
  \right.
\end{equation}
where $\mu\geq 0$ is a constant and $f\in L^2(0,T;H)$. Applying the
same arguments as in Theorem \ref{theorem-bdf2-mono} it can be shown that the
error satisfies, for $n\geq 3$,
\begin{align}\label{estim-ode}
  &\norz{e_w(t_n)}{\Omega}^2 + \mu\int_{t_{n-1}}^{t_n}\norz{e_w(t)}{\Omega}^2\,\mathrm{d}t\nonumber\\
  &\leq \norz{e_w(t_{n-1})}{\Omega}^2 + C_{\hat{K}}\left((\eta_n^{S,W})^2 + (\eta_n^{T,W})^2 \right),
\end{align}
where the residual defined on $K$ is
$R_{K,n}^W=\derpar{\tilde{w}_h}{t} + \mu\tilde{w}_h-f$. If we are in the special
situation that $f$ is in $L^2(0,T;V_h)$, then $R_{K,n}^W$ is as well. Therefore,
by Galerkin orthogonality, $R_{K,n}^W(t_n)=0$ and from this it is easy to
conclude that $R_{K,n}^W(t)\rightarrow 0$ uniformly in $t$ as
$\tau_n\rightarrow0$. By contrast, note that for parabolic problems, the
differential operator $A$ rarely satisfies $A(V_h)\subseteq V_h$. Moreover,
since the time estimator $\eta_n^{T,W}$ is $O(\tau_n^{5/2})$, then for $\tau_n$
small enough, the dominant term on the right side of (\ref{estim-ode}) is the
initial error $\norz{e_w(t_{n-1})}{\Omega}^2.$ Inspecting this term, let
$\pi_h^n$ be the interpolation operator for the new mesh obtained at time $t_n$,
and $w_h^{n-1,old}$ be the solution computed on the previous mesh so that
$w_h^{n-1}=\pi_h^n(w_h^{n-1,old})$. From the relation
$e_w(t_{n-1})=w(t_{n-1})-w_h^{n-1,old}+w_h^{n-1,old}-\pi_h^n(w_h^{n-1,old})$ we
get
\begin{align}\label{ode-interp}
  \|e_w(t_{n-1})\|_{0,\Omega}
  &\leq \|e_w^{old}(t_{n-1})\|_{0,\Omega}+\|w_h^{n-1,old}-\pi_h(w_h^{n-1,old})\|_{0,\Omega}.
\end{align}
If the interpolation error is not taken into account when adapting the mesh, the
last term on the right of (\ref{ode-interp}) will spoil the control of the
error. For instance, in regions of the domain where the solution varies rapidly,
the interpolation error remains large independent of the time step.

To control the interpolation error, we consider the estimator
\begin{align}\label{ode-recov-est-glob}
  \omega^{S,W}(t)
  =\left(\sum_K\omega_{K}(w_h)^2\right)^{1/2},
\end{align}
where $\omega_{K}(w_h)$ is from (\ref{omega-recovered}). Taking $\Pi(w_h)-\nabla u_h$
as a representation of the error $\nabla e_w$, then $\omega^{S,W}(t)$ estimates the
interpolation error of $e_w$ in the $L^2(\Omega)$-norm at time $t$. Moreover, by
Young's inequality we have
\begin{align*}
\eta^{S,W}(t)
&\leq\omega^{S,W}(t)+\left(\sum_K\|R^W_{K}(t)\|_{0,K}^2\right)^{1/2},
\end{align*}
so that $\omega^{S,W}$ dominates the estimator $\eta^{S,W}$ whenever the element
residual is small.

We illustrate with a numerical example. Let $\Omega=(0,100)\times(0,100)$,
$T=100$, and let $w$ solve (\ref{ode-lin}) with $f=-0.016875\tanh(x-0.25t-50)$,
$\mu=0.01$ and initial condition $w(0)=0$. The solution of this equation mimics
the behaviour of the recovery variable for the FitzHugh-Nagumo model. In
particular, the unknown $u$ in (\ref{FHN-recov-eq}) is replaced by
$\tanh(x-0.25t-50)$ to represent a traveling wave moving to the right with fixed
speed. The approximate solution is computed on a uniform mesh with $h=2.5$
($6400$ elements) and constant time step $\tau=1$. The exact error is estimated
by computing a reference solution on a much finer mesh with $h=0.15625$
($1638400$ elements). The approximate solution at time $t=100$ is plotted over
the line $y=50$ in Figure \ref{FHN-trav-errplots} (left), including the
superposition of the source term (right $y$-axis).

Figure \ref{FHN-trav-errplots} (right) plots the local distribution of the error
and estimators over the elements of the mesh, projected on the line $y=50$ at
time $t=100$. Note that the estimator $\eta^{S,W}_K(t)$ detects error primarily
in regions where the source term varies rapidly near $x=75$. This region
essentially acts as an ``activation'' region. Since there is no diffusion
involved in this equation, outside this activation region one does not expect a
large new contribution to the error as $t$ increases. On the other hand, we note
that the estimator $\omega^{S,W}(t)$ gives a local in time representation of the
exact error, and in this instance remains within an order of magnitude
throughout the domain. In particular, the relative size of the estimator
$\omega^{S,W}(t)$ in different regions of the domain reflects the local features
of the solution, such as at $x=45$ and $x=75$ where the solution transitions
from constant values.

\begin{figure}[ht]
  \centering
  \subfloat{
    \includegraphics[width=0.47\textwidth]{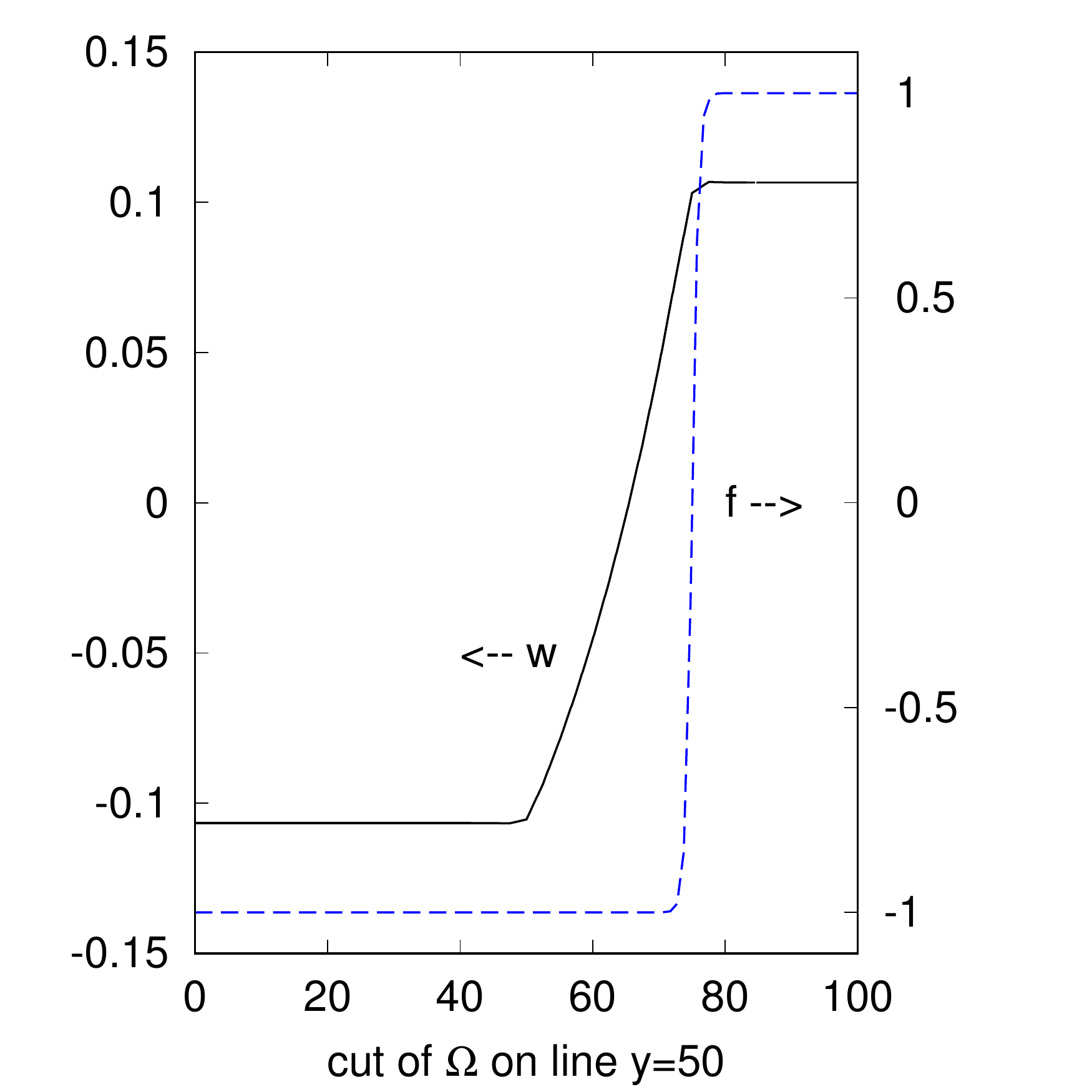}
  }
  \subfloat{
    \includegraphics[width=0.47\textwidth]{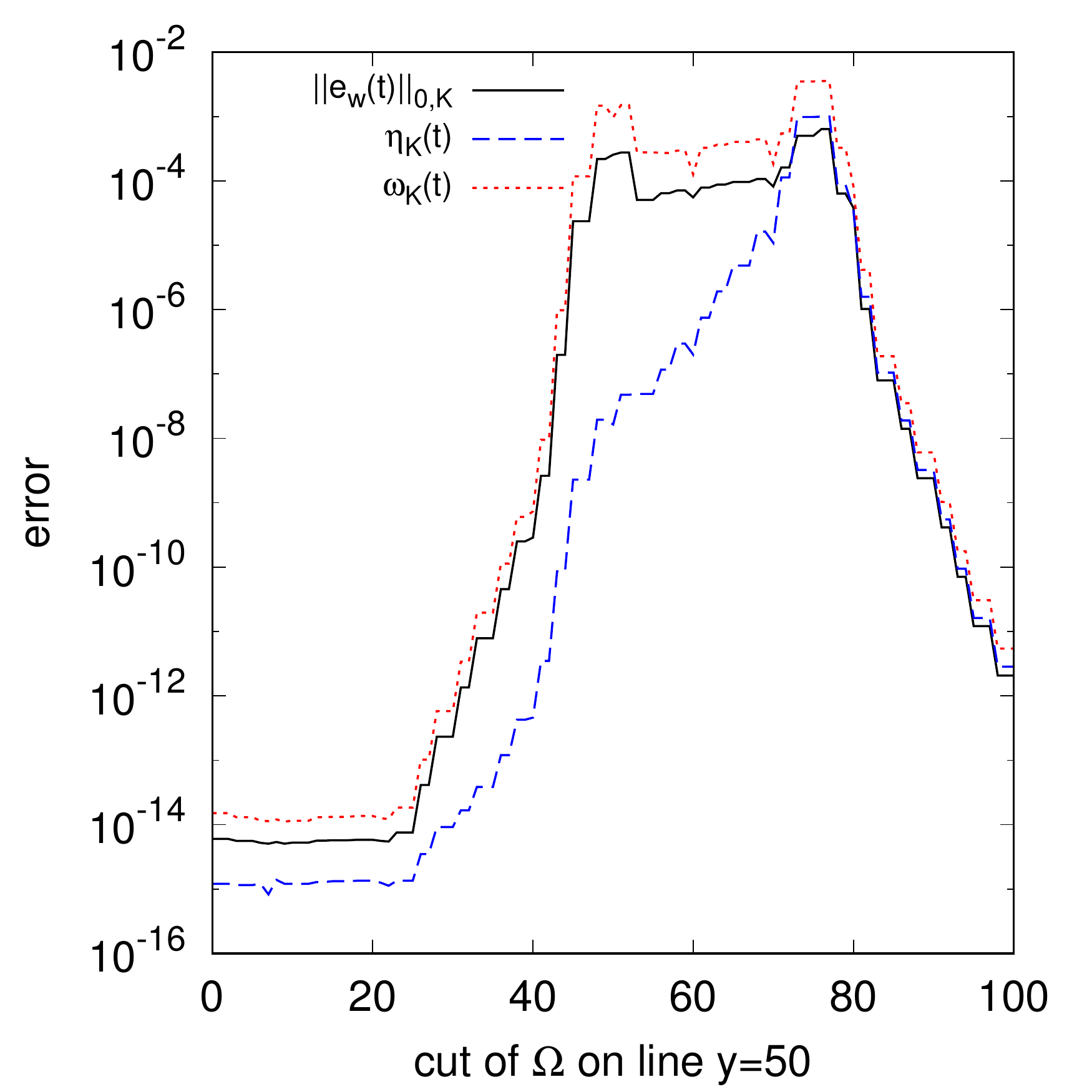}
  }
  \caption{Left: plot of approximate solution (left $y$-axis) and source term
    (right $y$-axis) over the line $y=50$ at time $t=100$. Right: plot of the
    error and estimators over the line $y=50$ at time $t=100$.}
  \label{FHN-trav-errplots}
\end{figure}

Based on the above observations, when applying mesh adaptation in Section
\ref{sec:numer-react} we will makes use of the estimator
(\ref{ode-recov-est-glob}) to control the interpolation error for the variable
$w$. Additionally, since the estimator $\eta^{S,W}$ only reports significant
error when the source term varies rapidly, and hence when $u$ varies rapidly, it
will be dropped from the computation. Note, however, that this assumption is
dependent on the ionic model used, including the use of external source terms.


\subsection{Adaptive algorithm}
\label{subsec:bdf2-alg}

\subsubsection{Scalar problem}
\label{subsubsec:algor-scal}

Let $TOL_S,\,TOL_T$ be positive constants, denoting the tolerance for space and
time error respectively. As in \cite{lozpicpra09}, the goal will be to adapt
both the mesh and time step in order to satisfy the following inequalities on
each interval
\begin{align}
  0.1875\,\,TOL_S\leq \frac{\eta_n^S}{\vertiii{u_h}_n}\leq
  0.75\,\,TOL_S,\label{eqn-TOL-S}\\
  0.5\,\,TOL_T\leq
  \frac{\eta_n^T}{\vertiii{u_h}_n}\leq 1.5\,\,TOL_T,\label{eqn-TOL-T}
\end{align}
where
$\vertiii{v}_n=\left(\int_{t_{n-1}}^{t_n}\noro{v(t)}{\Omega}^2\,\mathrm{d}t\right)^{1/2}$.
The choice of constants appearing in the upper and lower bounds in
(\ref{eqn-TOL-S}) may seem unusual, particularly the upper bound $0.75$ which is
less than $1$. However, we found that with the adaptation algorithm we describe
below, the error after adaptation was generally significantly lower than the
target error.

The mesh is adapted using a non-Euclidean metric, derived using the techniques
from \cite{micper06} that we briefly describe. The metric consists of a positive
definite matrix $\mathcal{M}_K$ corresponding to each element. The Jacobian
$J_K$ relates to the metric by the relation
$\mathcal{M}_K=\mathcal{R}_K^T\Lambda_K^{-2}\mathcal{R}_K$. Therefore, given the
prescribed global error tolerance $\widetilde{TOL},$ the idea is to determine a
new optimally defined element $\widetilde{K}$ by choosing new directions
$\tilde{r}_{1,K},\,\tilde{r}_{2,K}$ and aspect ratio
$\tilde{s}_K=\frac{\tilde{\lambda}_{1,K}}{\tilde{\lambda}_{2,K}}$ which minimize
$\eta_{\widetilde{K}}$ and scaling the area
$\tilde{\lambda}_{1,K}\tilde{\lambda}_{2,K}$ so that the error satisfies
$\eta_{\widetilde{K}}=\frac{\widetilde{TOL}}{\sqrt{N_T}}.$ Since the mesh
adaptation software we use in this paper, MEF++, requires the metric to be
defined on vertices, the metric needs to be averaged. We use the simple
averaging for each vertex $p$:
\[
\mathcal{M}_p=\frac{1}{N_p}\sum_{K\in \Delta_p}\mathcal{M}_K,
\]
where $\Delta_p$ is the patch of elements containing $p$ as a vertex, and $N_p$
is the number of elements of $\Delta_p$. Since the BDF2 method involves the
variables $u_h^{n-2},\,u_h^{n-1},\,u_h^n$, the mesh should be adapted to the
solution on the entire interval $[t_{n-2},t_n].$ Therefore, following an idea
from \cite{rio12}, for each sub-interval $[t_{s-1},t_s]$ we construct the metric
$\mathcal{M}_s$ under the condition $\frac{\eta^S}{\vertiii{u_h}_n}=TOL_S.$ Then
we take the average $\mathcal{M}^n=\frac{1}{3}\sum_{s=0}^2\mathcal{M}_{n-s}$.
We have made the common assumption that the edge residual $r_K$ dominates the
space residual $R_K$, and so have dropped the latter from the calculation, see
\cite{carver99},\,\cite{boupicalalos09},\,\cite{burpic03}.

The time-step adaptation is implemented with a standard method involving
multiplicative factors. After computing the solution $u_h^n$ the estimator
$\eta^T$ is computed. If (\ref{eqn-TOL-T}) is not satisfied, a multiplicative
factor $m>0$ determines a new step $\tau_{new}=m\tau.$ For our purposes, we
chose $m=0.67$ if the error was too large, and $m=1.5$ if the error was too
small. Note that the estimator is only defined for $n=3$ onwards. We therefore
cannot completely control the accumulation of error in the first two time
steps. We will assume that if the error does not satisfy the upper bound
(\ref{eqn-TOL-T}) for the third step, then it also does not satisfy the bound
for the first two steps. We therefore terminate the algorithm and restart using
a smaller time step.

In order to avoid asking for too much control of either the space or time error,
we fix the ratio $\frac{TOL_T}{TOL_S}=\mu.$ In \cite{lozpicpra09}, this ratio is
fixed at $\mu=1$. The choice of $\mu$ used in this paper will be discussed
further in Section \ref{sec:numer-react}, as it is observed that a suitable choice
depends on the particular problem being solved.

\begin{algorithm}[h]
  \caption{Space-time solution-adaptation loop.}
  \label{loop-rdiff}
  \begin{enumerate}
    \item \textit{Initialization: } Find a suitable starting mesh.
      \begin{enumerate}
      \item Repeat the following for a fixed number of iterations (5 or 10):
        \begin{enumerate}
        \item \label{itm:solve1}
          Interpolate $u_h^0=\pi_h(u_0)$ on the current mesh and
          solve for $u_h^1,\,u_h^2,\,u_h^3$.
        \item Adapt the mesh using $u_h^0,\,u_h^1,\,u_h^2,\,u_h^3$.
        \end{enumerate}
        \item Interpolate $u_h^0$ on the current mesh and solve for $u_h^1,\,u_h^2,\,u_h^3,\,u_h^4$.
      \end{enumerate}
    \item \textit{Space Loop: }\label{itm:space}
      \begin{enumerate}
        \item Compute $\eta_n^S.$
        \item If (\ref{eqn-TOL-S}) is satisfied, move on to step \ref{itm:time}.
        \item If (\ref{eqn-TOL-S}) is not satisfied, adapt the mesh using
          $u_h^{n-2},\,u_h^{n-1},\,u_h^n$.
        \item Interpolate $u_h^{n-2},\,u_h^{n-1}$ on the new mesh, and solve for
          $u_h^n,\,u_h^{n+1}.$  Go to step \ref{itm:space}.
      \end{enumerate}
    \item \textit{Time Loop: }\label{itm:time}
      \begin{enumerate}
        \item Compute $\eta_n^T.$
        \item If (\ref{eqn-TOL-T}) is satisfied, or if the ratio
          $\frac{\tau_n}{\tau_{n-1}}$ is too large or too small, compute
          $u_h^{n+1}$ and go to step
          \ref{itm:space}.
        \item If (\ref{eqn-TOL-T}) is not satisfied, adjust the time step,
          recompute $u_h^n$ and go to step (\ref{itm:time}).
      \end{enumerate}
  \end{enumerate}
\end{algorithm}


\subsubsection{Monodomain problem}
\label{subsubsec:algor-mono}

We outline the extension of the adaptive algorithm from the scalar problem to
the full monodomain problem. For mesh adaptation, choose positive constants
$TOL_S^U,\,TOL_S^W$. The goal is then to control the relative error:
\begin{align}
  0.25\,\,TOL_S^U\leq \frac{\eta_n^{S,U}}{\vertiii{u_h}_n}\leq
  TOL_S^U,\label{eqn-TOL-SU}\\
  0.25\,\,TOL_S^W\leq \frac{\omega_n^{S,W}}{\vertiii{w_h}_n}\leq
  0.65\,\,TOL_S^W.\label{eqn-TOL-SW}
\end{align}
As for the scalar case, we remark the constants appearing the upper and lower
bounds are highly implementation and problem dependent, and often chosen for
performance reasons. The mesh is adapted if either of the upper bounds in
(\ref{eqn-TOL-SU}) or (\ref{eqn-TOL-SW}) are violated, or if both of the lower
bounds are violated.  The mesh is adapted using a similar method as for the
scalar problem. However, instead of constructing separate metrics for each time
step $t_{n-2},\,t_{n-1},\,t_n$ as was done previously, here we construct for
each variable a metric $\mathcal{M}(u)$ and $\mathcal{M}(w)$ using the averages
$\frac{1}{3}(u_h^{n-2}+u_h^{n-1}+u_h^{n})$ and
$\frac{1}{3}(w_h^{n-2}+w_h^{n-1}+w_h^{n})$ respectively. Then we adapt the mesh
using the metric intersection $\mathcal{M}=\mathcal{M}(u)\cap\mathcal{M}(w)$.

The method for the time-step adaptation requires a small modification for the
monodomain system. Choose positive constants $TOL_T^U,TOL_T^W$. The time step is
decreased by a factor of $2/3$ if
\begin{align}
  \frac{\eta_n^{T,U}}{\vertiii{u_h}_n}>
  1.5 TOL_T^U,\quad\text{ or }\quad\quad
  \frac{\eta_n^{T,W}}{\vertiii{w_h}_n}>
  1.5 TOL_T^W,\label{eqn-TOL-T-up}
\end{align}
and the time step is increased by a factor of $3/2$ if
\begin{align}
  \frac{\eta_n^{T,U}}{\vertiii{u_h}_n}<
  0.5 TOL_T^U,\quad\text{ and }\quad\quad
  \frac{\eta_n^{T,W}}{\vertiii{w_h}_n}<
  0.5 TOL_T^W.\label{eqn-TOL-T-low}
\end{align}


\section{Numerical results}
\label{sec:numer-react}

All numerical computations in this section are performed with the MEF++ software
developed by GIREF \cite{gir}. The first two test cases are for a scalar
reaction-diffusion problem. The first verifies the equivalence of the error and
estimator and establishes the proper order of convergence of the estimators when
performing the adaptive algorithm. The second considers the issue of efficiency
of the adaptive method. Test cases 3 and 4 are for the monodomain problem. Test
case 3 verifies the reliability of the estimators for the error in space, and
assess the mesh adaptation algorithm. This includes a comparison of efficiency
when solving with mesh adaptation vs. solving on a uniform mesh. The last test
case illustrates the space-time adaptation method applied to a problem with
realistic time scales for a heart beat.


\subsection{Test case 1}
\label{subsec:dep5}


Let $\Omega=(0,1)\times(0,1),$ $T=0.04,$ and take $u$ to be the solution to
\begin{equation}
  \label{eqn:dep5}
  \left\{
    \begin{array}{ll}
      \frac{\partial u}{\partial t}-\Delta u+10^4u(u-1)(u-0.25)=0,&\quad\text{in }(0,T)\times\Omega, \\
      \nabla u\cdot n=0,&\quad \text{in }(0,T)\times\partial\Omega,\\
      u(0)=e^{-100(x^2+y^2)},&\quad \text{in }\Omega.
    \end{array}
  \right.
\end{equation}
The solution is a traveling wave with a circular profile, moving upwards and
right from the bottom left corner until $u\approx1$ in the entire domain (see
Figure \ref{fig-de5-contour-mesh}). We plot the space and time error estimators
as a function of time (Figure \ref{fig-dep5-err-uniform}) on a relatively fine
uniform mesh with $12800$ elements ($h=0.0125$), and with constant time step
$\tau=0.0004.$ The error estimator in space follows more or less the area of the
wave-front, increasing until the front hits the top and right boundaries, and
decreasing as the wave exits the domain. The time estimator behaves similarly,
but has two peaks, corresponding to the wave hitting the boundary at $t=0.027$
and just as it exits at $t=0.038$.

\begin{figure}[ht]
  \centering
    \includegraphics[width=0.45\textwidth]{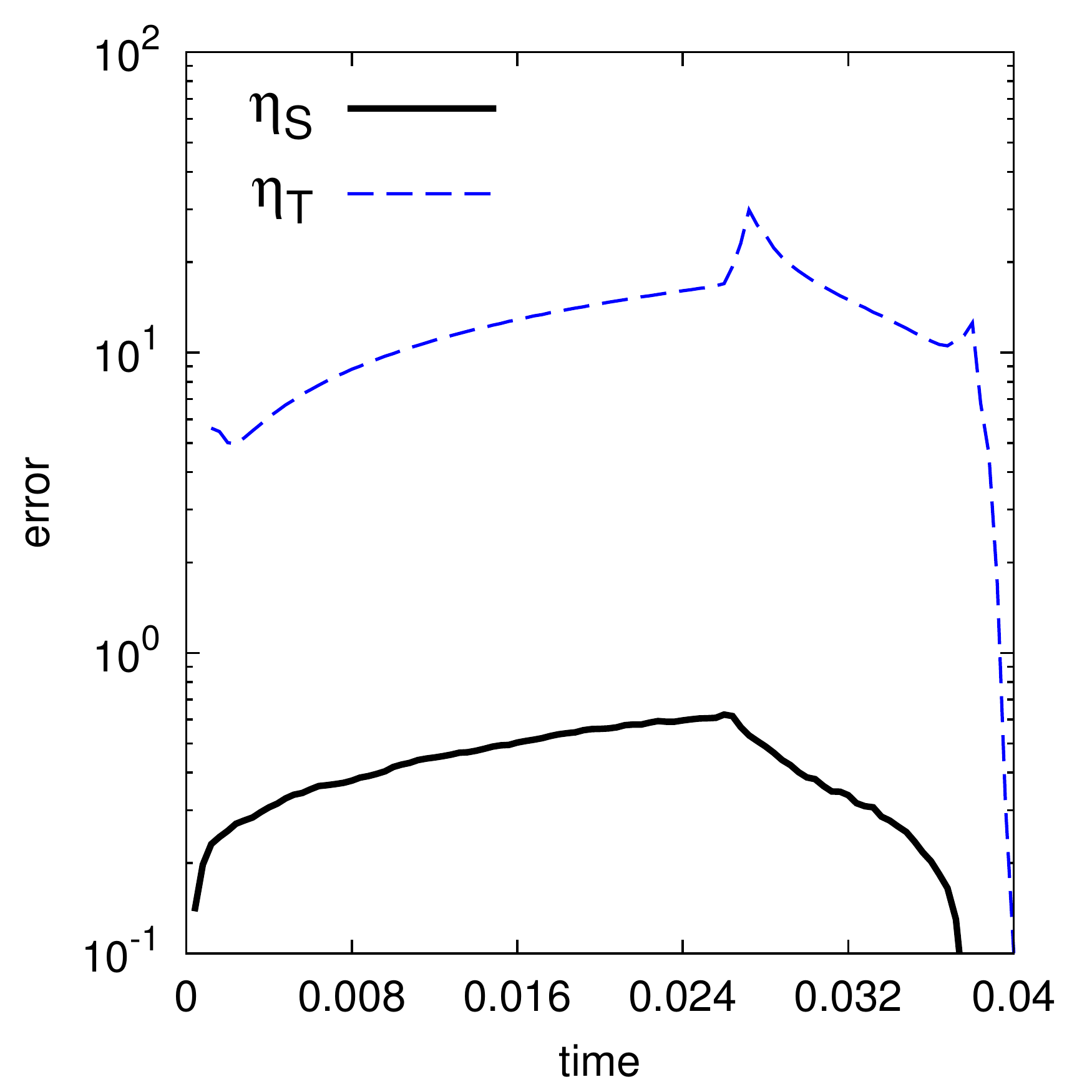}
  \caption{Test case 1: plot of the error estimators in time for the solution on a
    uniform mesh.}
  \label{fig-dep5-err-uniform}
\end{figure}


\subsubsection{Effectivity index}
\label{subsubsec:eff-ind}

As in \cite{lozpicpra09}, we define the effectivity indices
\begin{align*}
  ei
  =\frac{\left((\eta^S)^2+(\eta^T)^2\right)^{1/2}}{\vertiii{e_h}},
  \quad\quad
  ei^S=\frac{\eta^S}{\vertiii{e_h}},
  \quad\quad
  ei^T=\frac{\eta^T}{\vertiii{e_h}},
\end{align*}
where
\[
  \vertiii{v}=\left(\int_{0}^{T}\noro{v(t)}{\Omega}^2\,\mathrm{d}t\right)^{1/2},\quad
  \eta^S=\left(\sum_{n\geq 1}\left(\eta_n^S\right)^2\right)^{1/2},\quad
  \eta^T=\left(\sum_{n\geq 3}\left(\eta_n^T\right)^2\right)^{1/2}.
\]
In the absence of an exact solution, we approximate the effectivity index with
the use of a reference solution. Normally, a reference solution is computed on a
sufficiently fine uniform mesh with a very small time step. However, we found
that for this example, computing a reference solution with a uniform mesh is
impractical, as we illustrate below. For fixed time step $\tau=0.0002,$ denote
by $u_{\Delta t}$ the semidiscrete in time solution to (\ref{eqn:dep5}), and
$u_{H\Delta t}$ the fully discrete solution computed on a uniform mesh with
$H=0.025.$ Table \ref{tab-refcalc-unif} illustrates the approximation of the
error $|(u_{H\Delta t}-u_{\Delta t})(0.01)|_{1,\Omega}$ by the value
$e_H^{h,ref}=|(u_{H\Delta t}-u_{h\Delta t})(0.01)|_{1,\Omega}$, where
$u_{h\Delta t}$ is the fully discrete solution computed on various uniform
meshes, and solutions using space adaptation. If we assume that the error at
time $t=0.01$ is $O(H)$ and $O(h)$, respectively, then at worst
\[
  \frac{\left|e_H^{h,ref}-|(u_{H\Delta t}-u_{\Delta
        t})(0.01)|_{1,\Omega}\right|}{|(u_{H\Delta t}-u_{\Delta
      t})(0.01)|_{1,\Omega}}
  \leq\frac{|(u_{h\Delta t}-u_{\Delta
      t})(0.01)|_{1,\Omega}}{|(u_{H\Delta t}-u_{\Delta t})(0.01)|_{1,\Omega}}
\]
is $O\left(\frac{h}{H}\right).$ Unless $h<<H,$ one cannot trust the accuracy of
the reference solution. With $H=0.025,$ in this case a mesh with $6400$
elements, this restriction on $h$ calls for a uniform mesh with millions of
elements, and unrealistic demands for memory and CPU usage. On the other hand,
with adapted meshes the approximate error approaches the value $0.953$ in
reasonable CPU time. In what follows, the reference solution will be computed
using space adaptation with $TOL_S=0.015625$ and with a finer time step
$\tau_{ref}=2.5\times 10^{-5}$. The resulting meshes range from $43000$ elements
at time $t=0$ to $180000$ elements at $t=0.01$.

\begin{table}[ht]
  \footnotesize
  \centering
  \begin{tabular}{c|*{3}{c}|*{6}{c}}
    \multicolumn{1}{c|}{}&\multicolumn{3}{c|}{h (uniform meshes)}&\multicolumn{5}{c}{$TOL_S$
                                                        (adapted meshes)}\\
    \multicolumn{1}{c|}{}&0.0125 &0.00625 &0.003125 &0.25  &0.125 &0.0625 &0.03125 &0.015625\\
    \hline 
    $e_H^{h,ref}$       &0.7380  &0.8992   &0.9396    &0.9487 &0.9541 &0.9556 &0.9530 &0.9531\\
    max. elements  &25600  &102400  &409600   &1113  &3431  &10450 &42172 &160563\\
    CPU time (sec)  &168    &871     &5229     &26    &65    &202   &878   &4089\\
  \end{tabular}
  \caption{Approximation of error in space at $t=0.01$ with $H=0.025,\,\tau=0.0002$.}
  \label{tab-refcalc-unif}
\end{table}

We compute the error and effectivity indices for a few values of $h$ and $\tau$
in Table \ref{tab-eff-unif}. The effectivity index $ei$ remains at a good value
($1\leq ei\leq 10$) provided the space and time estimators remain close in
magnitude. Moreover, the index $ei^S$ remains close to $1$. The effectivity
index increases as we decrease $h$ and $\tau$. For the coarse mesh $h=0.025$ the
solution has not yet reached the asymptotic convergence, noting that the error
decreases by a factor of $6$ (first and last rows) instead of the theoretic
factor of $4$ predicted by the general theory. The evolution of the error in
time is shown in Figure \ref{fig-errtime-unif2}. Note that the space and time
error estimators grow at the same rate for each value of $h,\,\tau$, while the
exact error grows slower for the lower value of $h,\,\tau.$ As a result, the
growth of the exact error is more closely captured by the estimators for finer
mesh and time step.

\begin{table}[ht]
  \footnotesize
  \centering
  \begin{tabular}{c
    S[table-format=1.0(0)e1,round-precision = 2]
    *{8}{c}}
    $h$     &$\tau$ &$\vertiii{e_h}$ &$\eta^S$ &$\eta^T$  &$ei$ &$ei^S$ &$ei^T$\\
    \hline
    0.025   &0.0002   &0.0582          &0.0494   &0.199    &3.52 &0.849 &3.42\\
    0.025   &0.0001   &0.0560          &0.0497   &0.0528   &1.29 &0.888 &0.943\\
    0.0125  &0.0002   &0.0240          &0.0247   &0.199    &8.34 &1.03  &8.29\\
    0.0125  &0.0001   &0.0216          &0.0250   &0.0531   &2.72 &1.16  &2.46\\
    0.00625 &0.0001   &0.00966         &0.0124   &0.0533   &5.66 &1.28  &5.52\\
  \end{tabular}
  \caption{Error and effectivity indices for uniform mesh and constant time step.}
  \label{tab-eff-unif}
\end{table}


\begin{figure}[ht]
  \centering
  \subfloat[${h=0.025,\,\tau=0.0002.}$]{
    \label{fig-errtime-unif5}
    \includegraphics[width=0.45\textwidth]{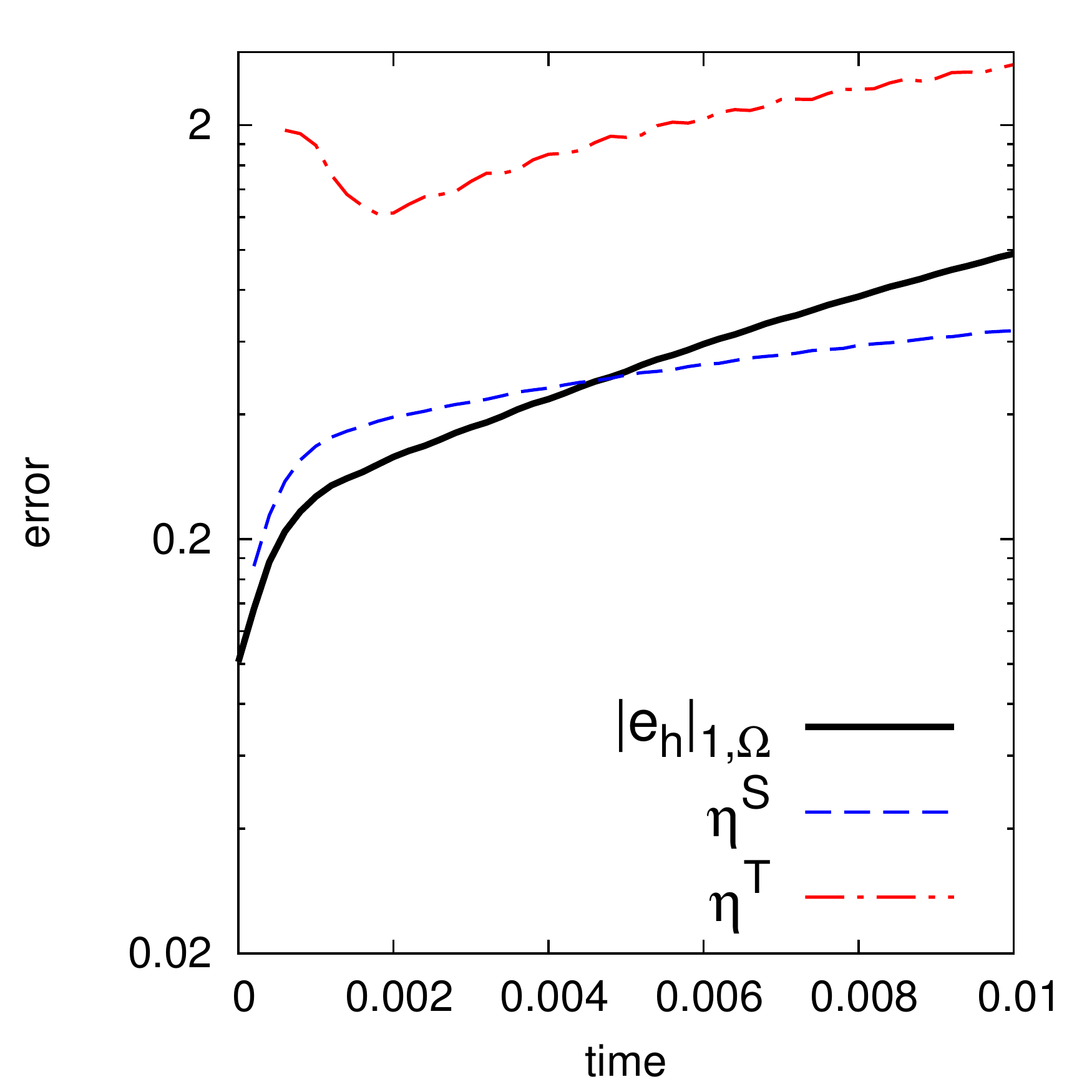}
  }
  \subfloat[$h=0.00625,\,\tau=0.0001.$]{
    \label{fig-errtime-unif14}
    \includegraphics[width=0.45\textwidth]{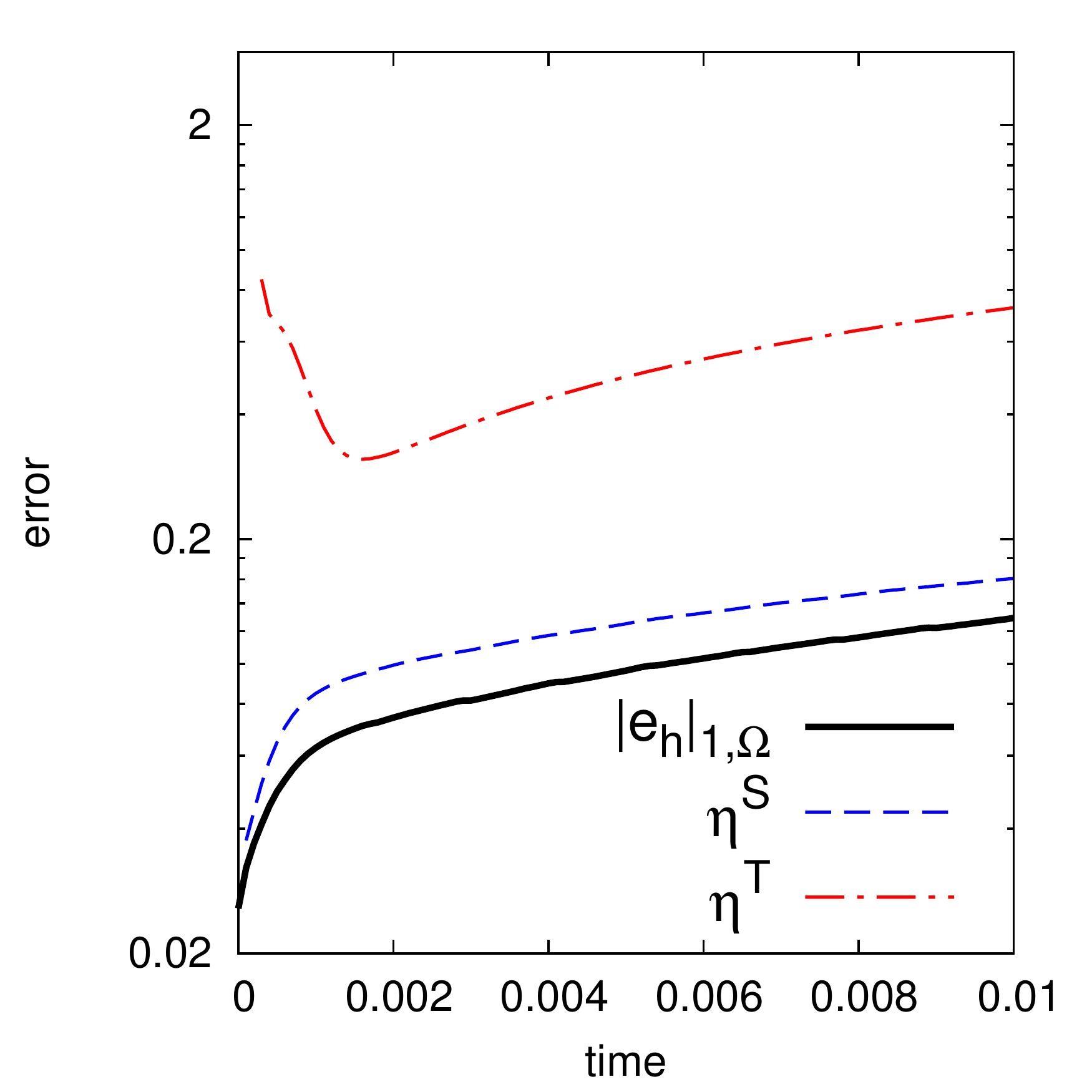}
  }

  \caption{Plot of the exact and estimated error in time for uniform mesh and constant time step.}
  \label{fig-errtime-unif2}
\end{figure}

\begin{remark} \normalfont
  If the reference solution is computed on an adapted mesh $\mathcal{T}_h$, the
  integral is approximated by interpolating the gradient of $u_{H\Delta t}$ at
  the Gauss points of $\mathcal{T}_h$. To verify the accuracy of the integral,
  we apply successive subdivisions of the quadrature formula. That is, the
  quadrature rule on each element is obtained by splitting the element into four
  copies by dividing each edge in half. If the initial quadrature rule is
  accurate order of $h^\ell$, the resulting subdivided rule is accurate of order
  $\left(\frac{h}{2}\right)^\ell$. Generally, on a fine adapted mesh, the
  difference is less than $1$\% after $3$ subdivisions.
\end{remark}


We now address the optimality of the time estimator. In Table
\ref{tab-dep5-opterr} we compute the estimators on uniform meshes and time
steps, and refine with a fixed ratio $\frac{\tau^2}{h}.$ We conclude that
asymptotically, the space error $\eta^S$ converges $O(h)$, the time estimators
$\eta^{i,T}$ for $i=1,\,3,\,4$ converge $O(\tau^2),$ while the estimator
$\eta^{2,T}$ converges $O(h\tau).$ It was observed in \cite{akrcha10} that the
estimator may be of sub-optimal order for a given ODE system when the first step
is computed with backward Euler. However, in our case we are not interested in
the best approximation of the semidiscretization in space, but the approximation
in both space and time of the PDE. We expect that the behaviour of the
estimators in these two situations may be different.

\begin{table}[ht]
  \footnotesize
  \centering
  \begin{tabular}{c
    S[table-format=1.0(1)e2,round-precision = 2]
    *6{c}}
    $h$ &$\tau$ &$\eta^S$ &$\eta^{1,T}$ &$\eta^{2,T}$ &$\eta^{3,T}$
    &$\eta^{4,T}$ &$\eta^T$\\
    \hline
    0.1     &8.00e-4 &0.168  &0.0126   &0.0881  &2.14   &3.49   &4.09\\
    0.05    &5.66e-4 &0.139  &0.00836  &0.0253  &0.875  &1.34   &1.60\\
    0.025   &4.00e-4 &0.0902 &0.00483  &0.00861 &0.436  &0.603  &0.744\\
    0.0125  &2.83e-4 &0.0492 &0.00259  &0.00308 &0.246  &0.289  &0.380\\
    0.0625  &2.00e-4 &0.0251 &0.00138  &0.00113 &0.137  &0.146  &0.201\\
    0.03125 &1.41e-4 &0.0126 &$7.11\times 10^{-4}$ &$4.06\times 10^{-4}$ &0.0732 &0.0731 &0.103\\
    \hline
    0.1     &2.00e-4 &0.162  &0.00118  &0.0309   &0.373   &0.289   &0.473\\
    0.05    &1.41e-4 &0.141  &$6.85\times 10^{-4}$ &0.00781  &0.107   &0.0938  &0.142\\
    0.025   &1.00e-4 &0.0917 &$3.58\times 10^{-4}$ &0.00239  &0.0354  &0.0396  &0.0532\\
    0.0125  &7.07e-5 &0.0497 &$1.82\times 10^{-4}$ &$8.33\times 10^{-4}$ &0.0190  &0.0189  &0.0268\\
    0.0625  &5.00e-5 &0.0252 &$9.28\times 10^{-5}$ &$2.96\times 10^{-4}$ &0.0101  &0.00941 &0.0138\\
    0.03125 &3.53e-5 &0.0126 &$4.65\times 10^{-5}$ &$1.04\times 10^{-4}$ &0.00523 &0.00468 &0.00702\\
  \end{tabular}
  \caption{Error estimators for uniform mesh and constant time step.}
  \label{tab-dep5-opterr}
\end{table}


\subsubsection{Dealing with overshoot of the time estimator}
\label{subsubsec:oscil}

We address a technical difficulty in implementing the space-time adaptation
algorithm. The time estimator is observed to strongly spike at times when the
mesh is adapted. To assess what is happening, we apply mesh adaptation with a
constant time step for various levels of $TOL_S$ and various time steps. We
found that for fixed $TOL_S,$ as the time step is decreased, the magnitude of
the overshoot increases (see top row of Figure \ref{fig-dep5-oscil}). Moreover,
from the bottom row of Figure \ref{fig-dep5-oscil} we see that the overshoot of
the estimator does not reflect the nature of the true error, and therefore,
decreasing the time step in order to attempt to control the time estimator would
not be worthwhile. The most likely explanation for the overshoot is that
interpolating the solution on the new adapted mesh introduces high frequency
transients, which are quickly damped. While the finite element solution itself
remains good, the estimator $\eta^T$ is built using finite difference schemes
for second and third-order derivatives, so the transients are magnified by small
time steps. In Figure \ref{fig-dep5-oscil}, note that the overshoot is largest
for $\eta^{3,T}$, which requires the solution at four different time steps. As
$TOL_S$ is decreased, the interpolation error decreases, so smaller time steps
may be taken. To apply the space-time adaptation algorithm, this implies that
the ratio $\frac{TOL_T}{TOL_S}$ cannot be taken too small. In practice, for this
test case, we found that the algorithm could be used with $\frac{TOL_T}{TOL_S}$
as low as $8,$ however, this is still unnecessarily restrictive. In Figure
\ref{fig-dep5-oscil}, we observe that the dominant terms are $\eta^{3,T}$ and
$\eta^{4,T}$, and that after the initial overshoot, $\eta^{3,T}$ quickly settles
back to the level of $\eta^{4,T}$. Moreover, note that in Table
\ref{tab-dep5-opterr} for uniform meshes, the column for $\eta^{3,T}$ is closely
matched by that for $\eta^{4,T}.$ In what follows, we replace $\eta^{T}$ with the
modified estimator:
\begin{align}\label{eq-etaT-mod}
  \tilde{\eta}^T
  &=\left((\eta_n^{1,T})^2+(\eta_n^{2,T})^2+(\eta_n^{4,T})^2\right)^{1/2},
\end{align}
While oscillations are also observed in $\eta^{1,T},$ they are small relative to
$\eta^{4,T}$. Therefore, including $\eta^{1,T}$ in the estimator did not result
in oscillations in (\ref{eq-etaT-mod}). Another possible solution, which was
not explored here, would be to use a more accurate interpolation operator as was
done in \cite{boupic13}.

\begin{figure}
  \centering
  \subfloat{
    \includegraphics[width=0.45\textwidth]{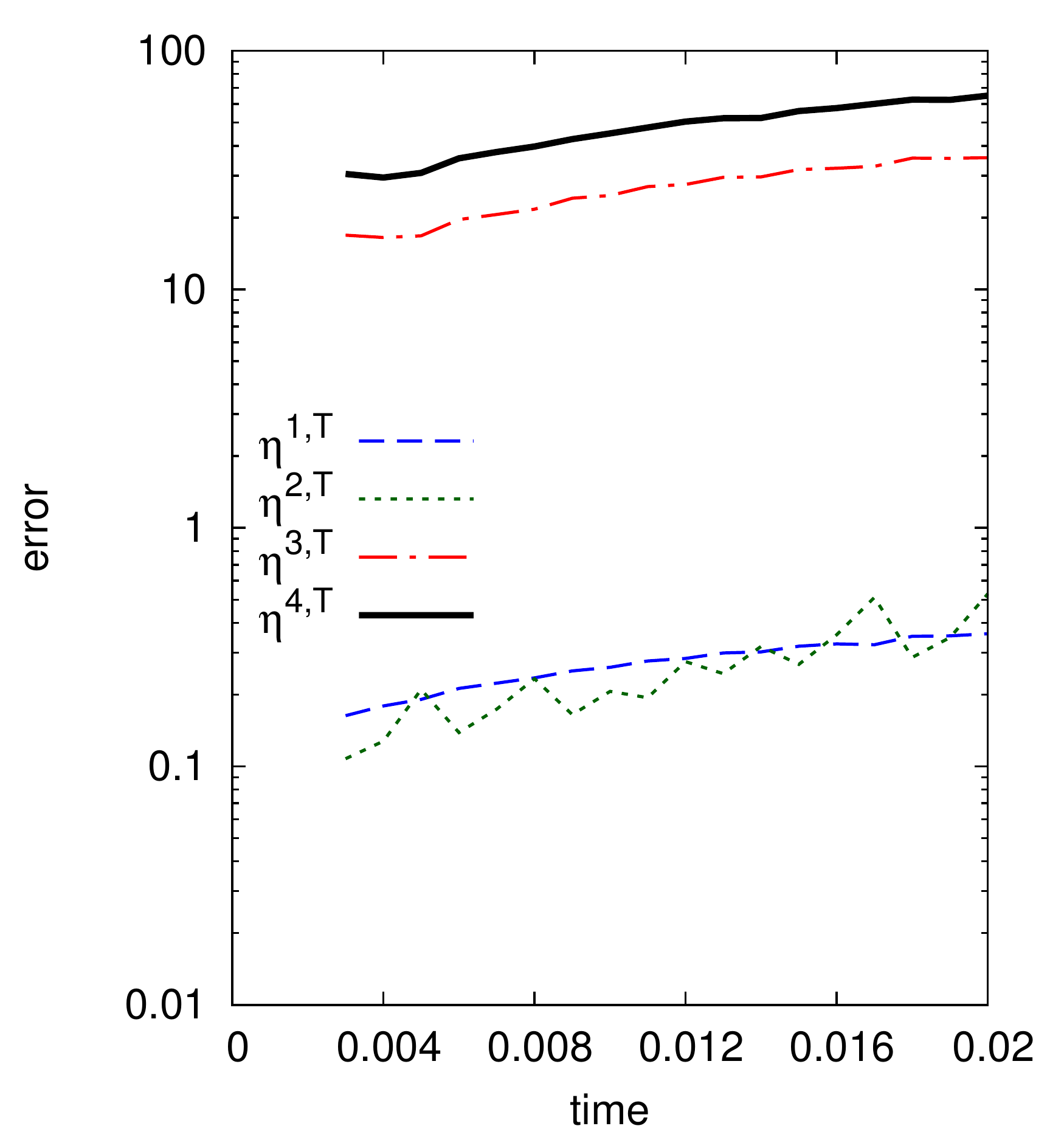}
  }
  \subfloat{
    \includegraphics[width=0.45\textwidth]{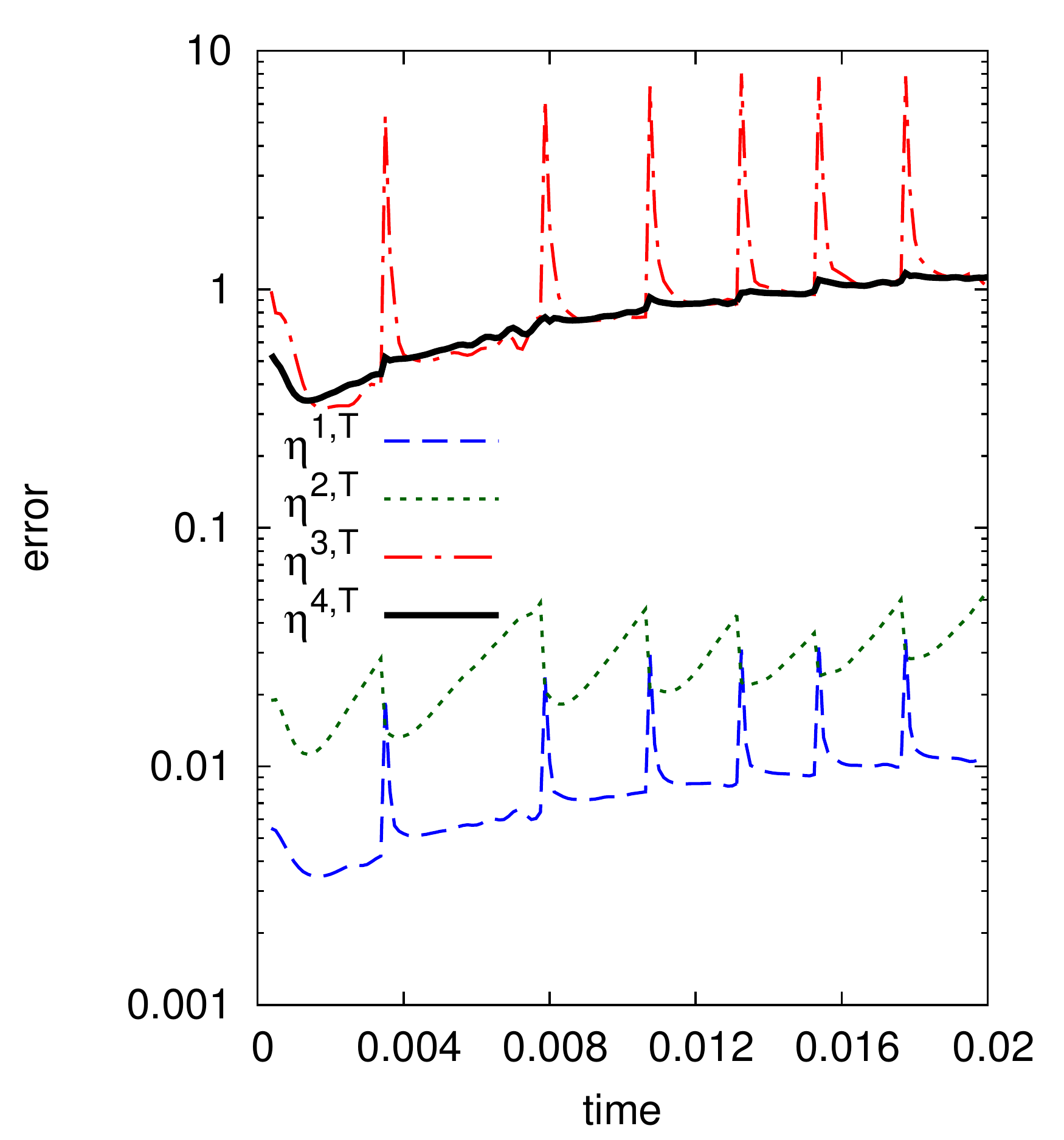}
  }
  \\
  \subfloat{
    \includegraphics[width=0.45\textwidth]{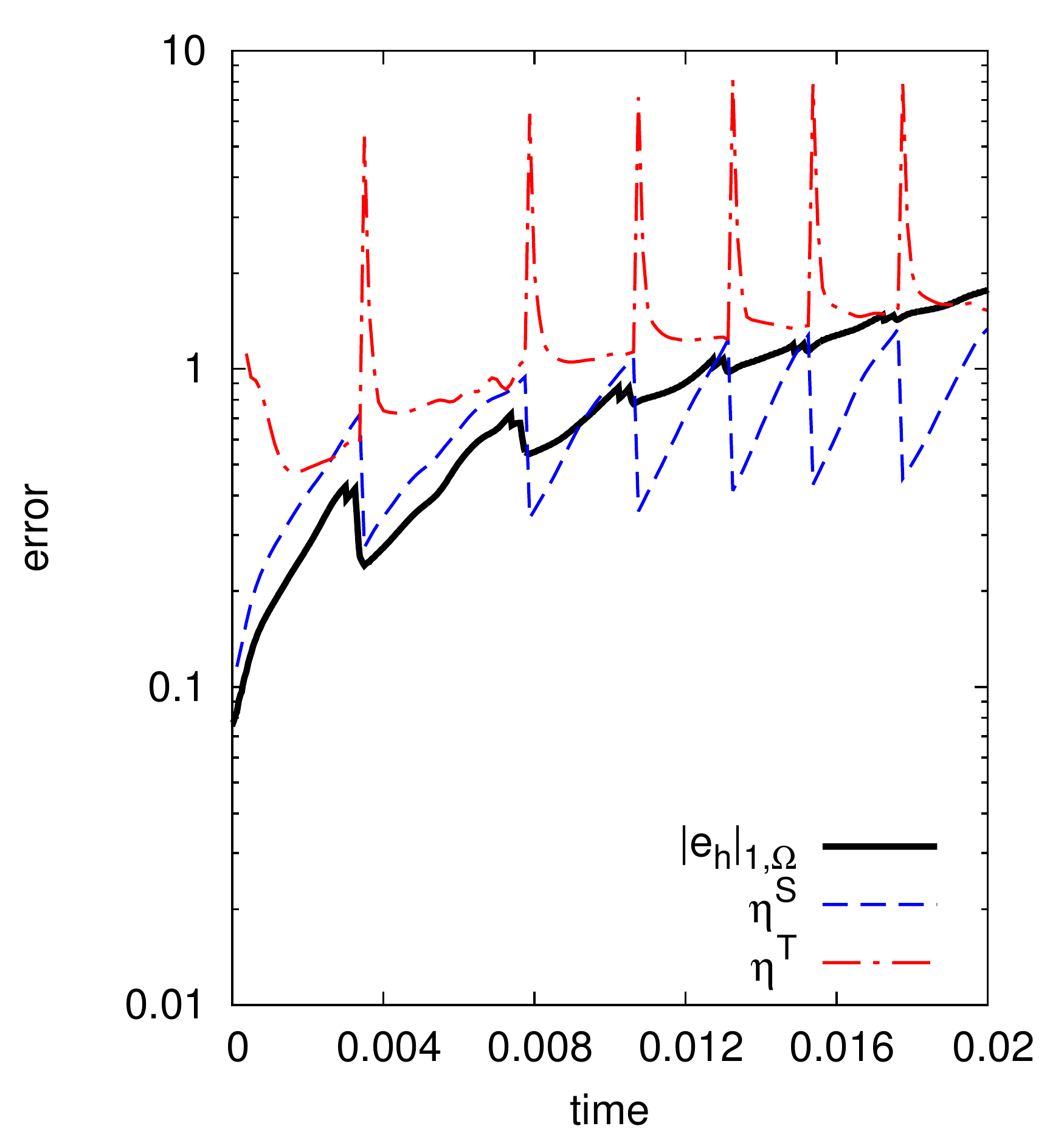}
  }
  \subfloat{
    \includegraphics[width=0.45\textwidth]{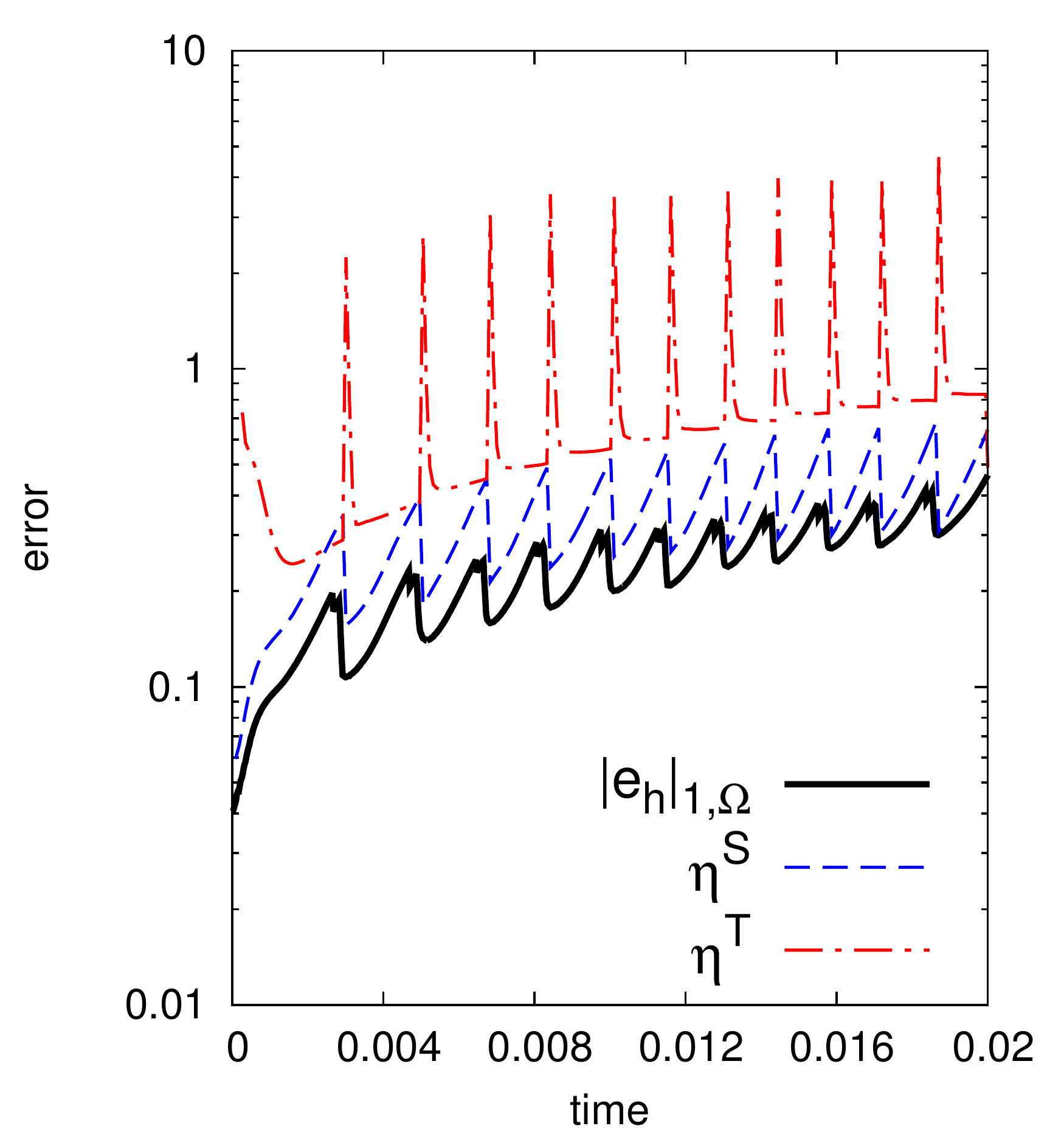}
  }
  \caption{Test case 1. Top: evolution of time estimator applying mesh
    adaptation with $TOL_S=0.5$ with constant time step $\tau=0.001$ (left) and
    $\tau=0.000125$ (right). Bottom: evolution of the estimators and exact
    error for $TOL_S=0.5,\,\tau=0.000125$ (left) and
    $TOL_S=0.52,\,\tau=8.85625\times 10^{-5}$ (right).}
  \label{fig-dep5-oscil}
\end{figure}


\subsubsection{Space-time adaptation}
\label{subsubsec:dep5-adapt}

We fix the ratio $\frac{TOL_T}{TOL_S}=0.75$ and apply space-time adaptation
(Algorithm \ref{loop-rdiff}). Note that as the wave exits the domain, the
normalizing factor $|u_h|_{1,\Omega}$ appearing in (\ref{eqn-TOL-S}) quickly
decreases to zero. To avoid pathological behaviour, we use the modified form
$\max(|u_h|_{1,\Omega},1)$. Examples of adapted meshes and the solution are
shown in Fig. \ref{fig-de5-contour-mesh}. The majority of the elements are
located near the wave-front and elongated parallel to the wave.  Note that the
mesh is somewhat coarser near the center of the wave, where the wave is nearly
linear in the direction orthogonal to the wave. Figure
\ref{fig-dep5-adapt-step-err} shows the evolution of the error estimators, the
time step, and the number of elements for a complete computation using
$TOL_S=0.125,\,TOL_T=0.09375$. The adaptation maintains the relative space and
time error near a constant value until the normalizing factor $|u_h|_{1,\Omega}$
quickly drops off as the wave exits the domain. Recall from Figure
\ref{fig-dep5-err-uniform} for a uniform mesh, the space and time estimators
slowly increase and decrease as the surface of the wave-front increases and
eventually exits. This behaviour is reflected in the adaptive algorithm with a
slow growth and decrease in the number of elements, and in the fact that the
time step is mostly constant. Moreover, the decreases in the time step and
corresponding spikes in the time estimator coincide with the wave hitting the
boundary. The oscillation in $\eta^S$ is a reflection of the mesh adaptation,
with the error dropping off suddenly after adaptation, and growing quickly as
the wave-front moves beyond the refined region. In Table
\ref{tab-est-adapt-dep5} it is shown that the estimated error terms converge
proportionally to $TOL_S,\,TOL_T$ when applying the space-time adaptation
algorithm. In Table \ref{tab-stats-dep5} we collect some additional statistics.
Columns three, four and six address the efficiency of the time-step adaptation
algorithm. From column three, we note that the number of time-step modifications
is essentially independent of the error tolerance. Recall that each time the
time step is modified, the current time step needs to be recomputed after which
condition (\ref{eqn-TOL-T}) is checked once again. The current time step may
possibly need to be recomputed several times before the error condition is
satisfied. Column four collects the total number of times that a time step needs
to be recomputed, and we conclude that in this case, no additional
recomputations are required to satisfy the error condition. Then from column 6,
it is shown that the total percentage of CPU time spent computing the time
estimator is low (no more than $0.3$\%). From column five, we observe that the
total number of time steps has a moderate growth as a function of
$TOL_T^{-1}$. More precisely, as we decrease $TOL_T$ by $2$, the number of time
steps increases by approximately a factor of $\sqrt{2}$ (as was observed in
\cite{lozpicpra09} for the Crank-Nicolson method). This is consistent with the
fact that the BDF2 method is second-order accurate. Lastly, note that the number
of mesh adaptations grows sublinearly as $TOL_S$ is decreased.

\begin{figure}
  \centering
  \subfloat{
    \label{fig-dep5-cont-start}
    \includegraphics[width=0.32\textwidth]{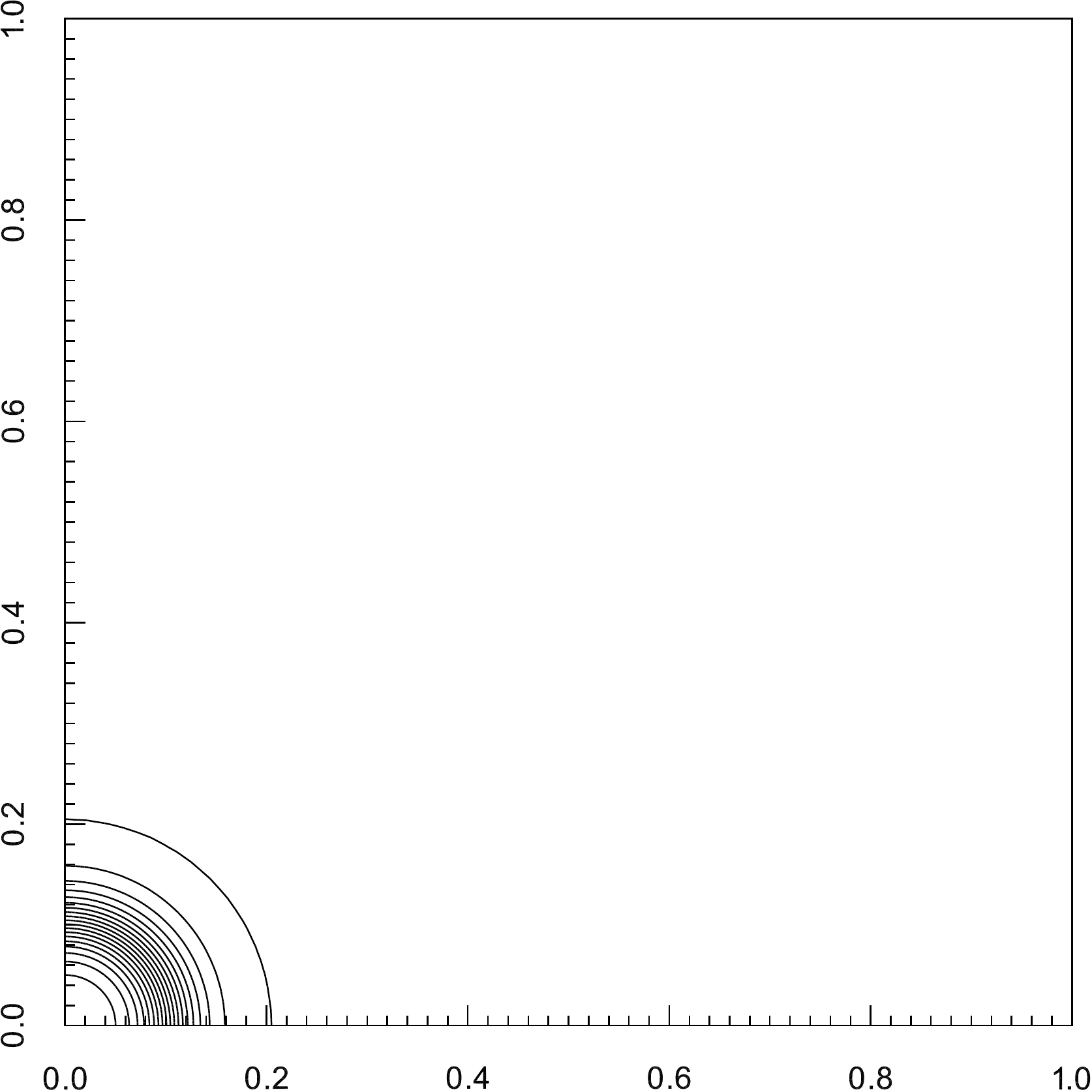}
  }
  \subfloat{
    \label{fig-dep5-mesh-start}
    \includegraphics[width=0.32\textwidth]{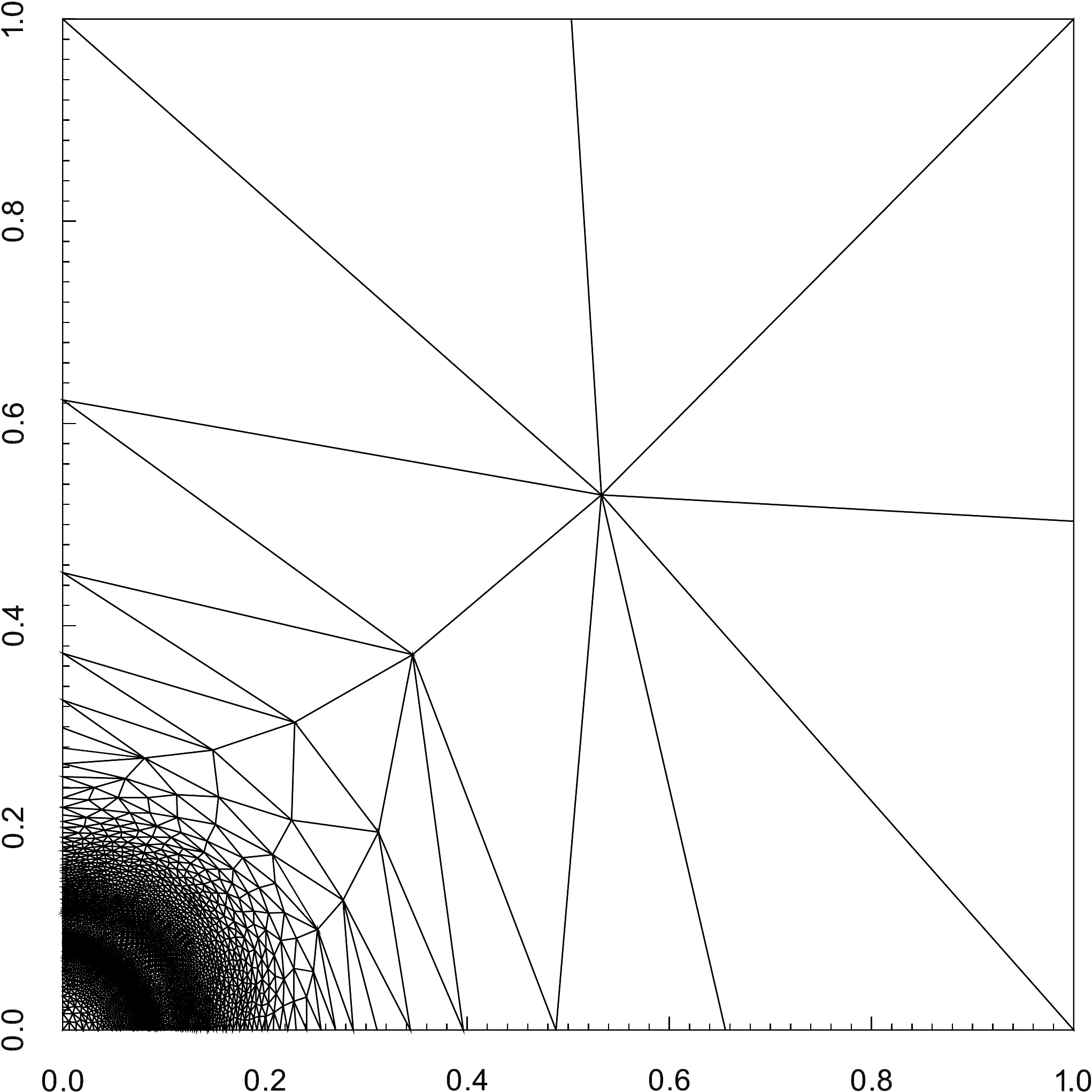}
  }
  \\
  \subfloat{
    \label{fig-dep5-cont-bound}
    \includegraphics[width=0.32\textwidth]{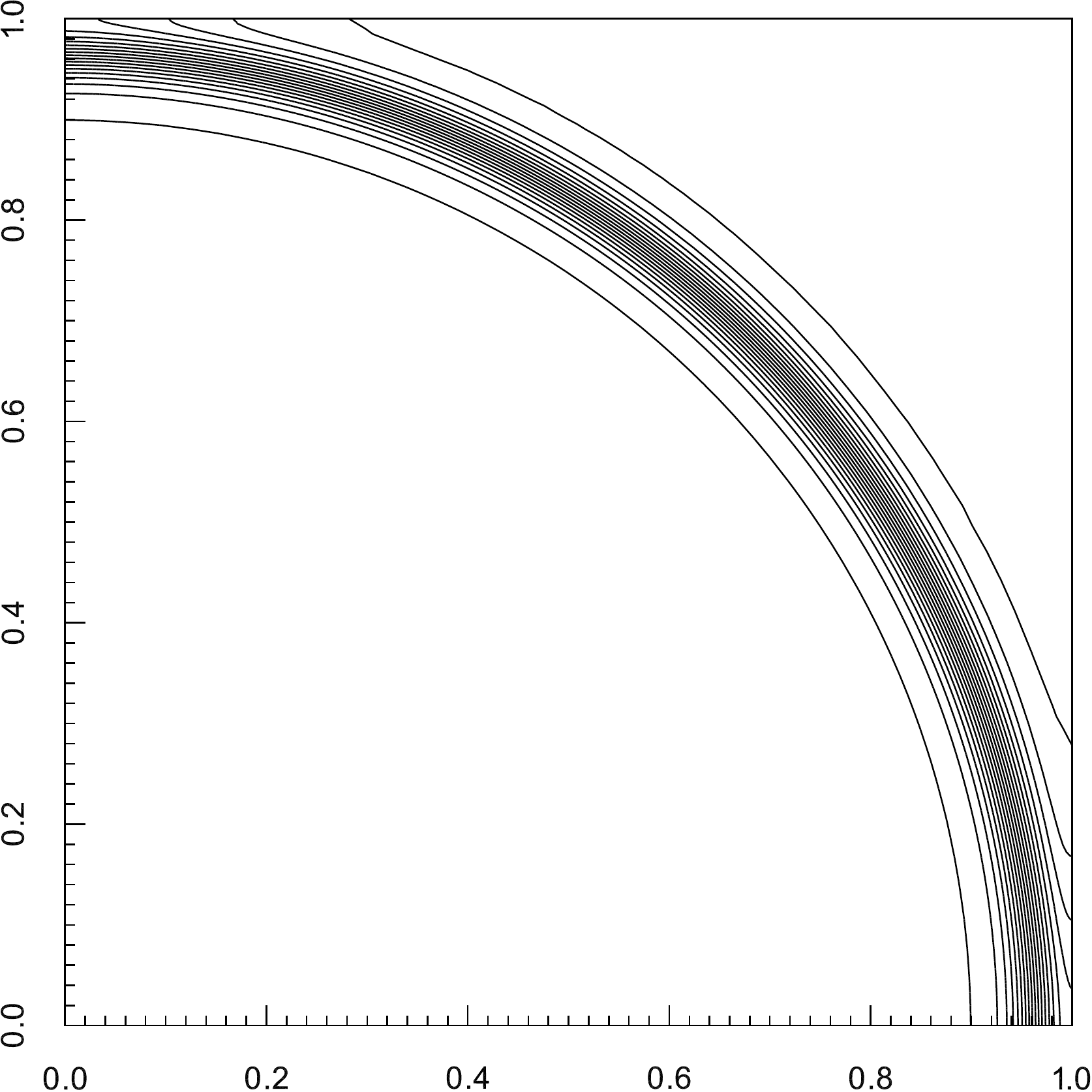}
  }
  \subfloat{
    \label{fig-dep5-mesh-bound}
    \includegraphics[width=0.32\textwidth]{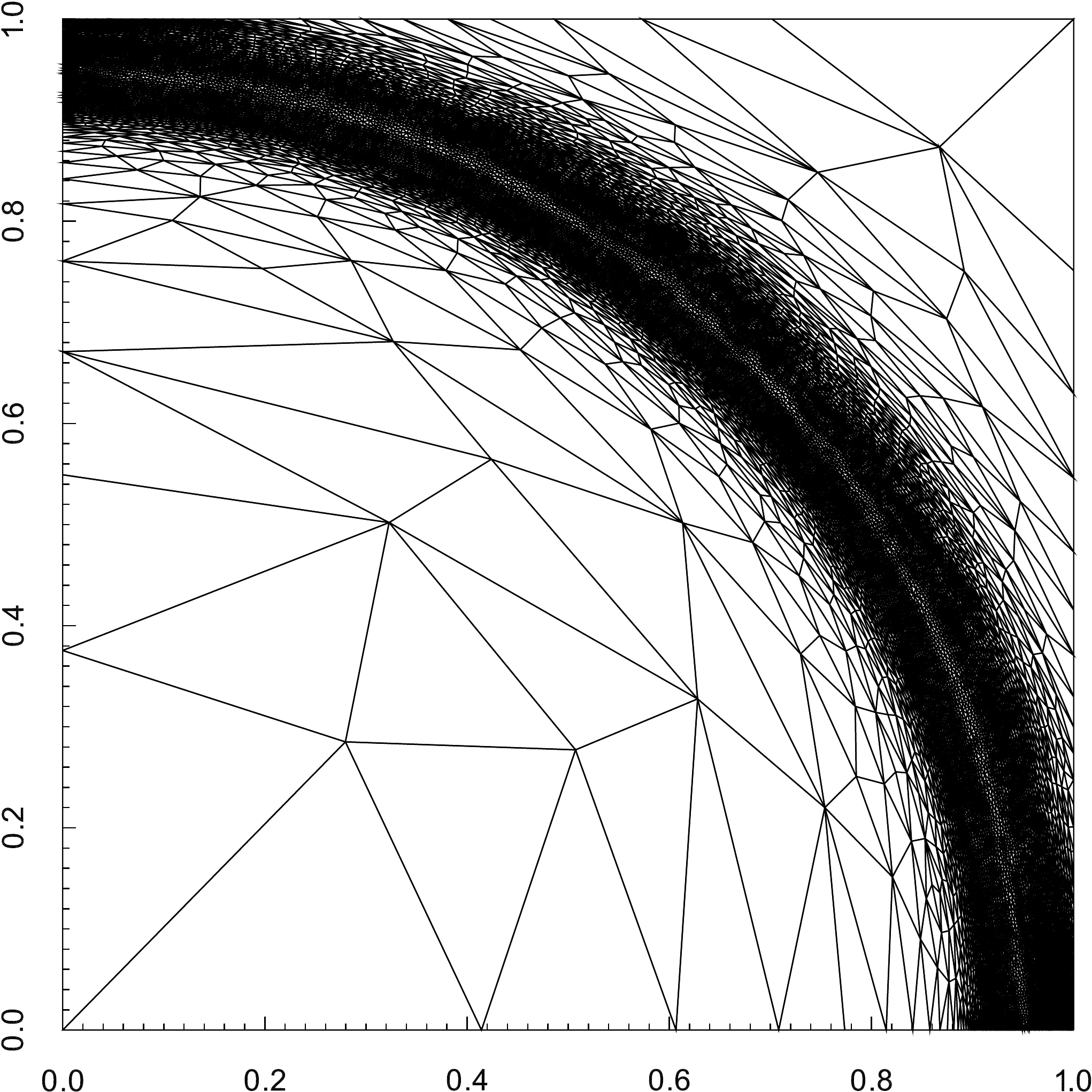}
  }
  \subfloat{
    \label{fig-dep5-mesh-bound-zoom}
    \includegraphics[width=0.32\textwidth]{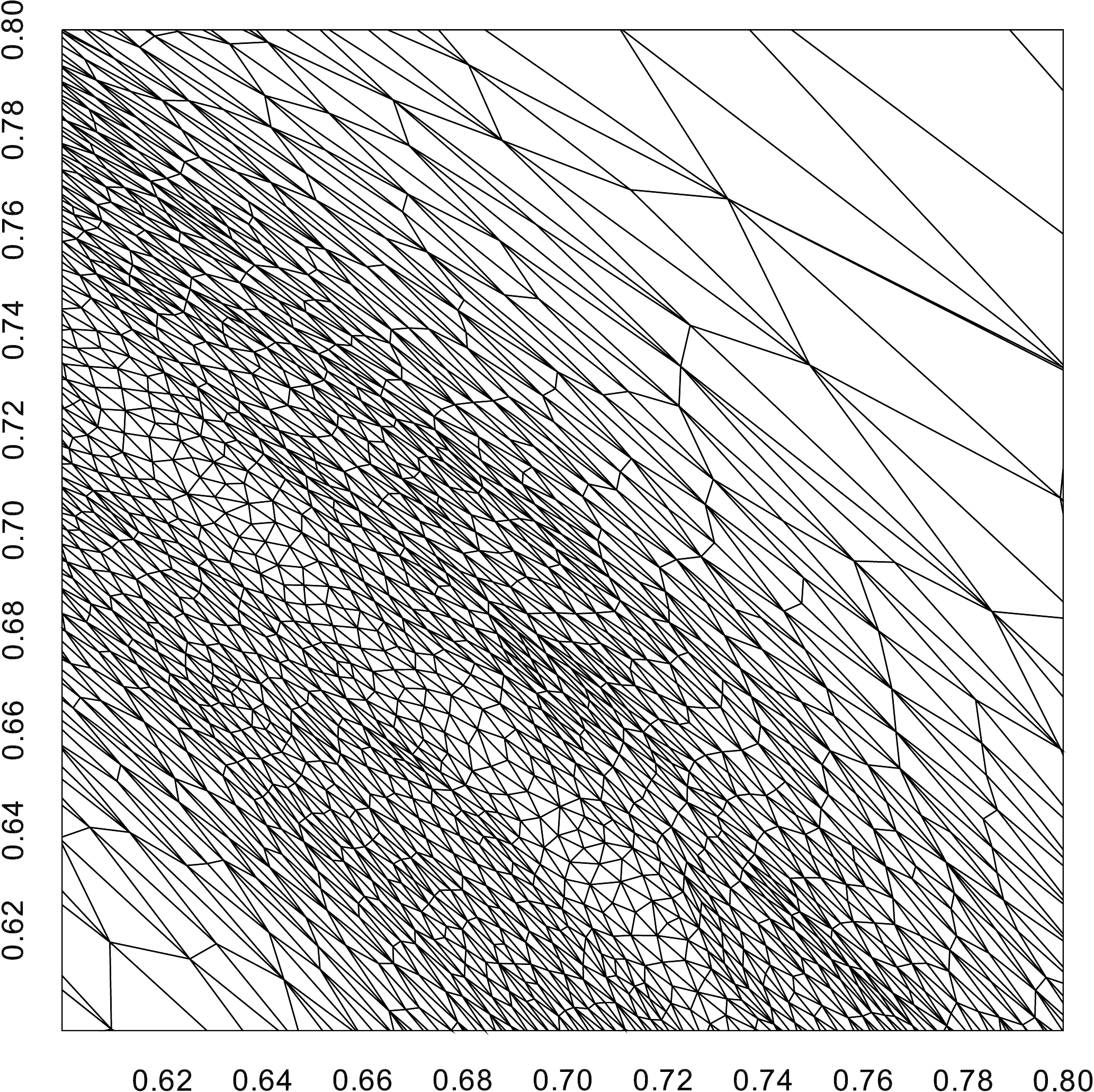}
  }
  \\
  \subfloat{
    \label{fig-dep5-cont-exit}
    \includegraphics[width=0.32\textwidth]{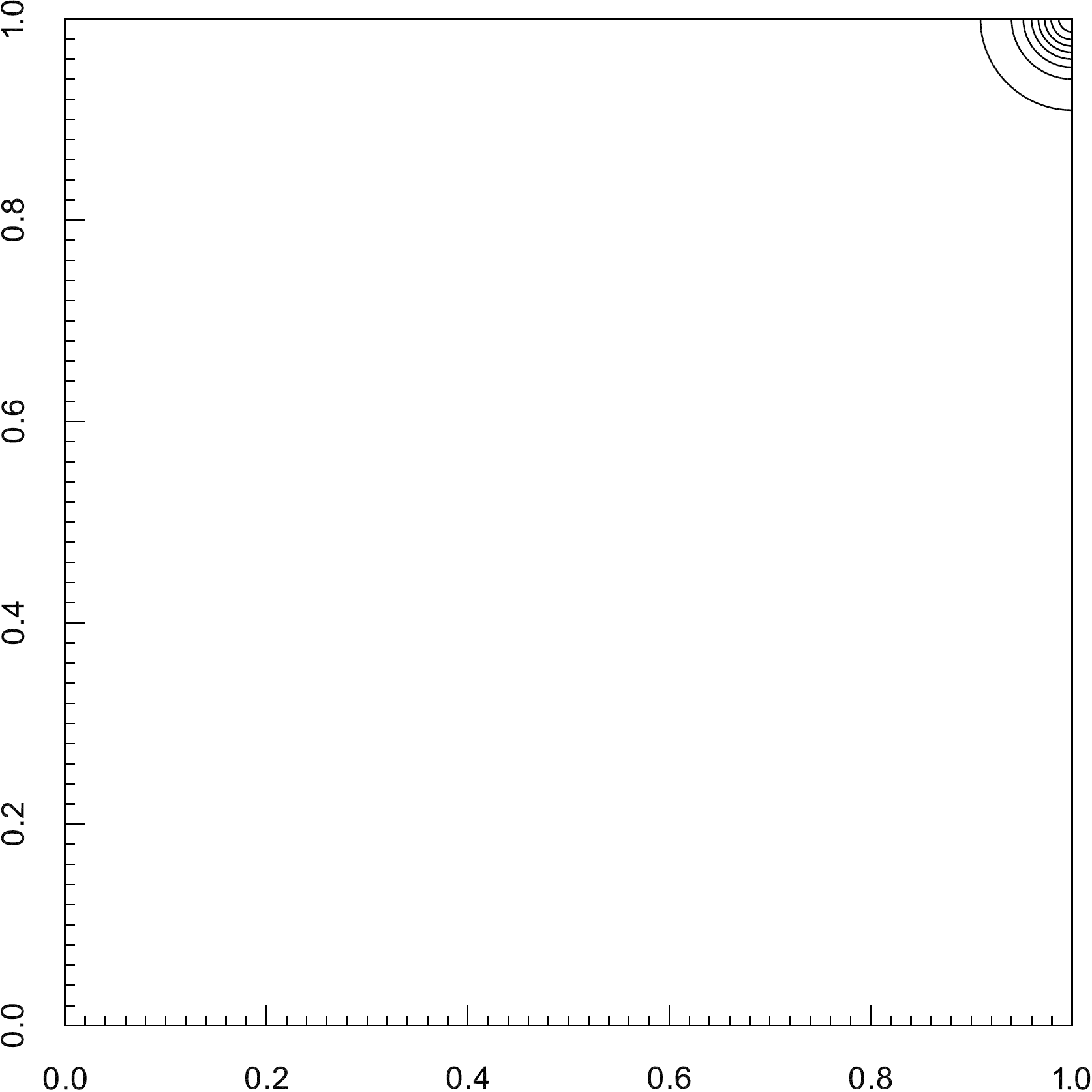}
  }
  \subfloat{
    \label{fig-dep5-mesh-exit}
    \includegraphics[width=0.32\textwidth]{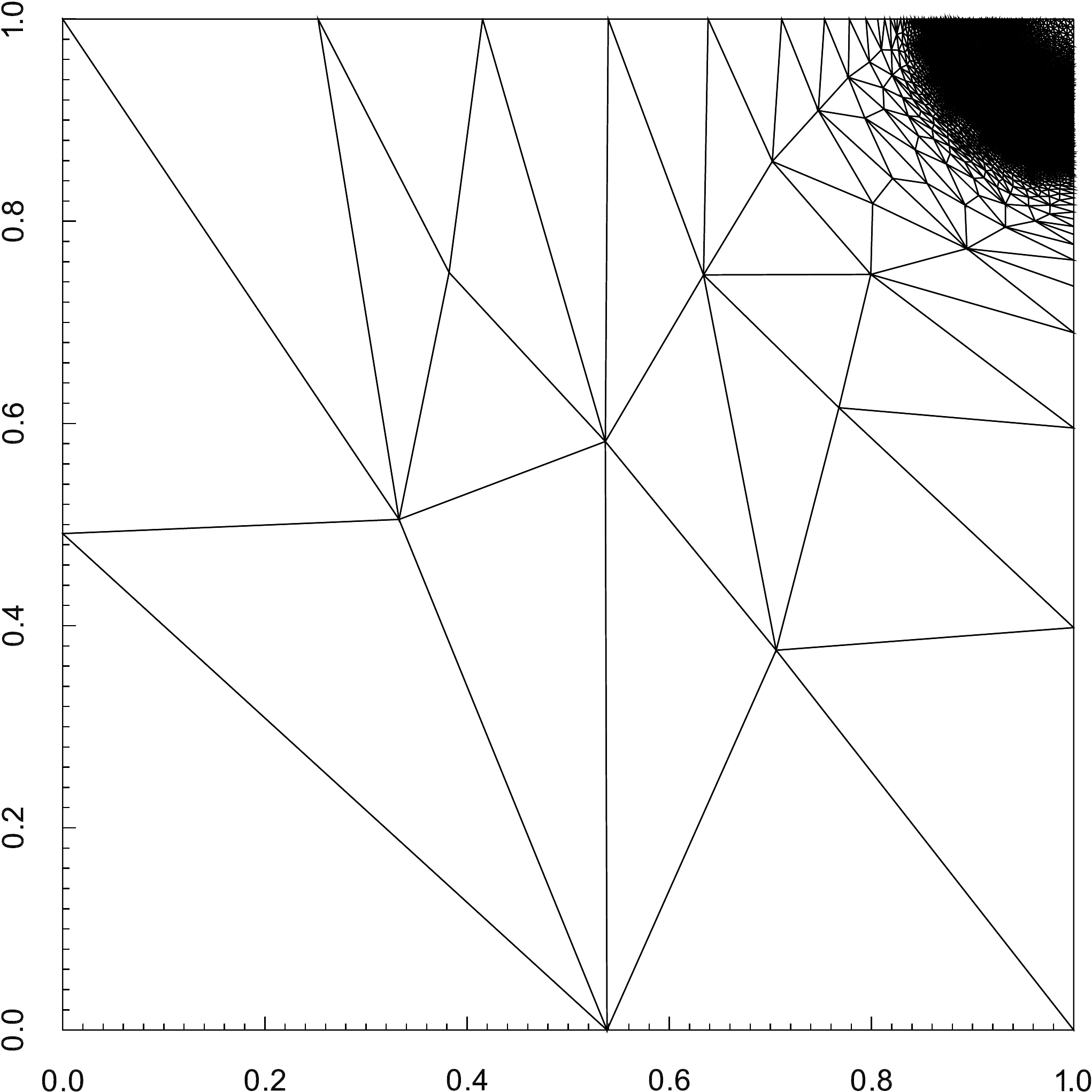}
  }
  \caption{Examples of contours and adapted meshes with
    $TOL_S=0.125,\,TOL_T=0.09375$ at $t=0.00141933$ (top), $t=0.0265008$
    (center) with zoom to the wave-front (center right), and at $t=0.0382107$
    (bottom).}
  \label{fig-de5-contour-mesh}
\end{figure}

\begin{figure}
  \centering
  \subfloat{
    \label{fig-dep5-err-adapted}
    \includegraphics[width=0.45\textwidth]{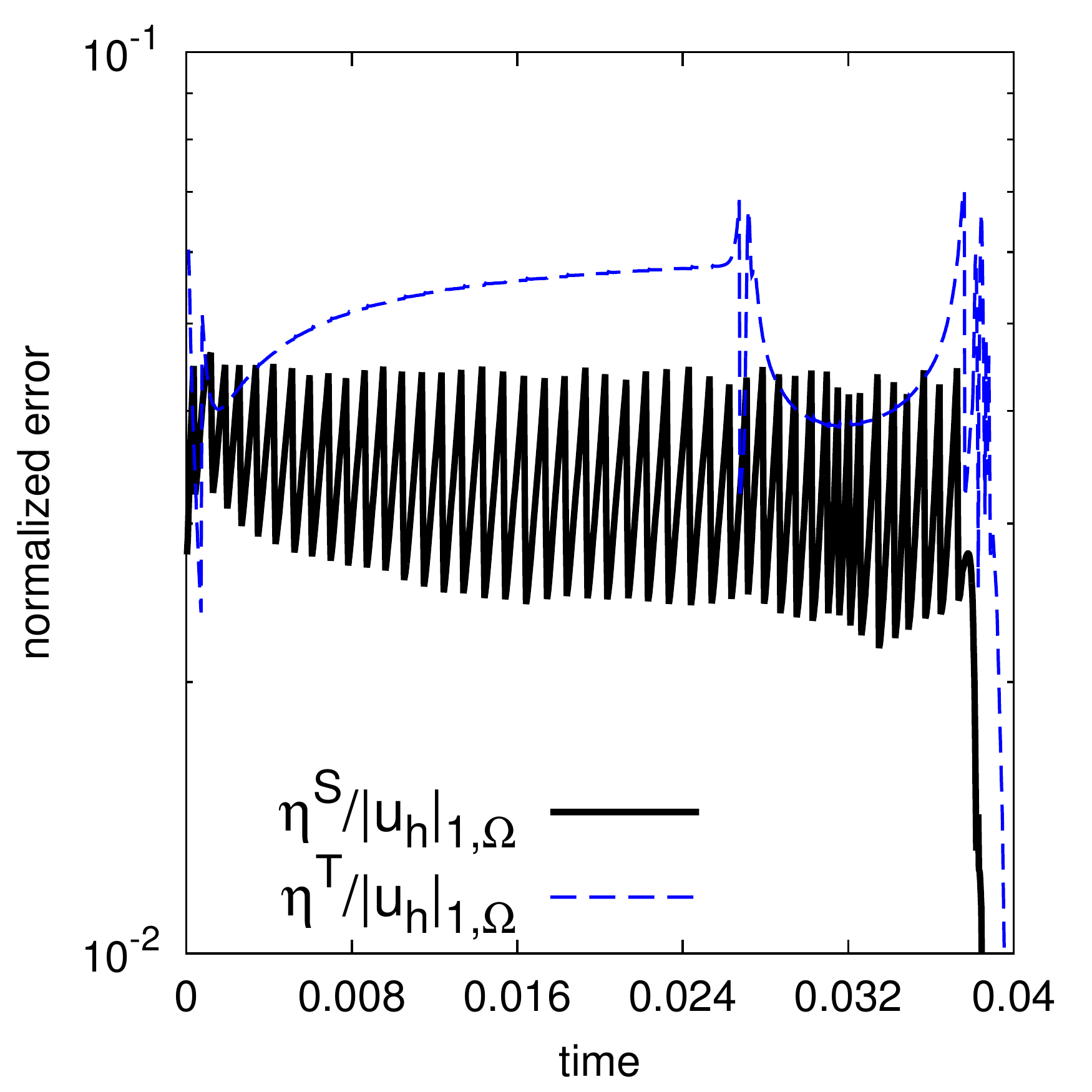}
  }
  \subfloat{
    \label{fig-dep5-tstep}
    \includegraphics[width=0.45\textwidth]{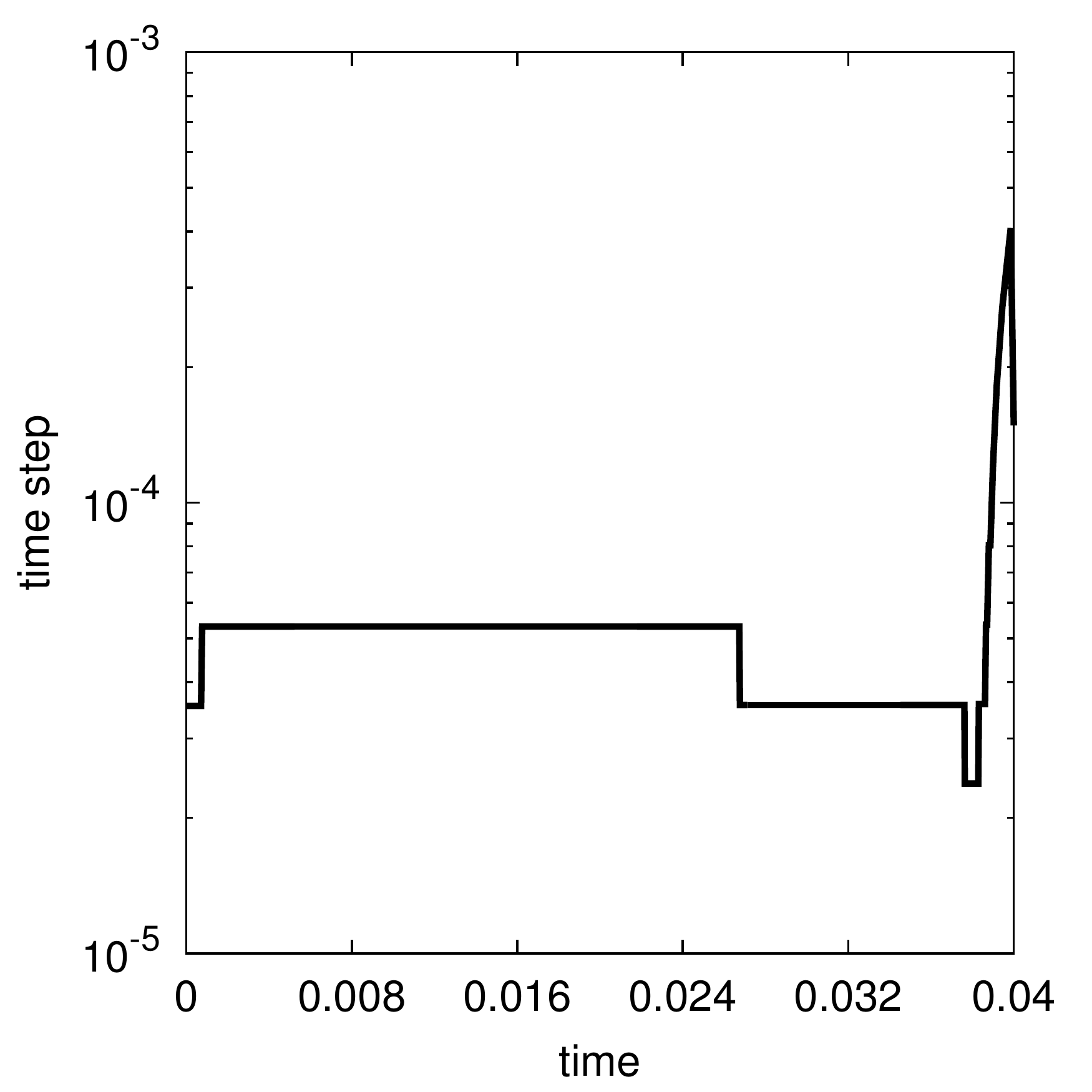}
  }
  \\
  \subfloat{
    \label{fig-dep5-elements}
    \includegraphics[width=0.45\textwidth]{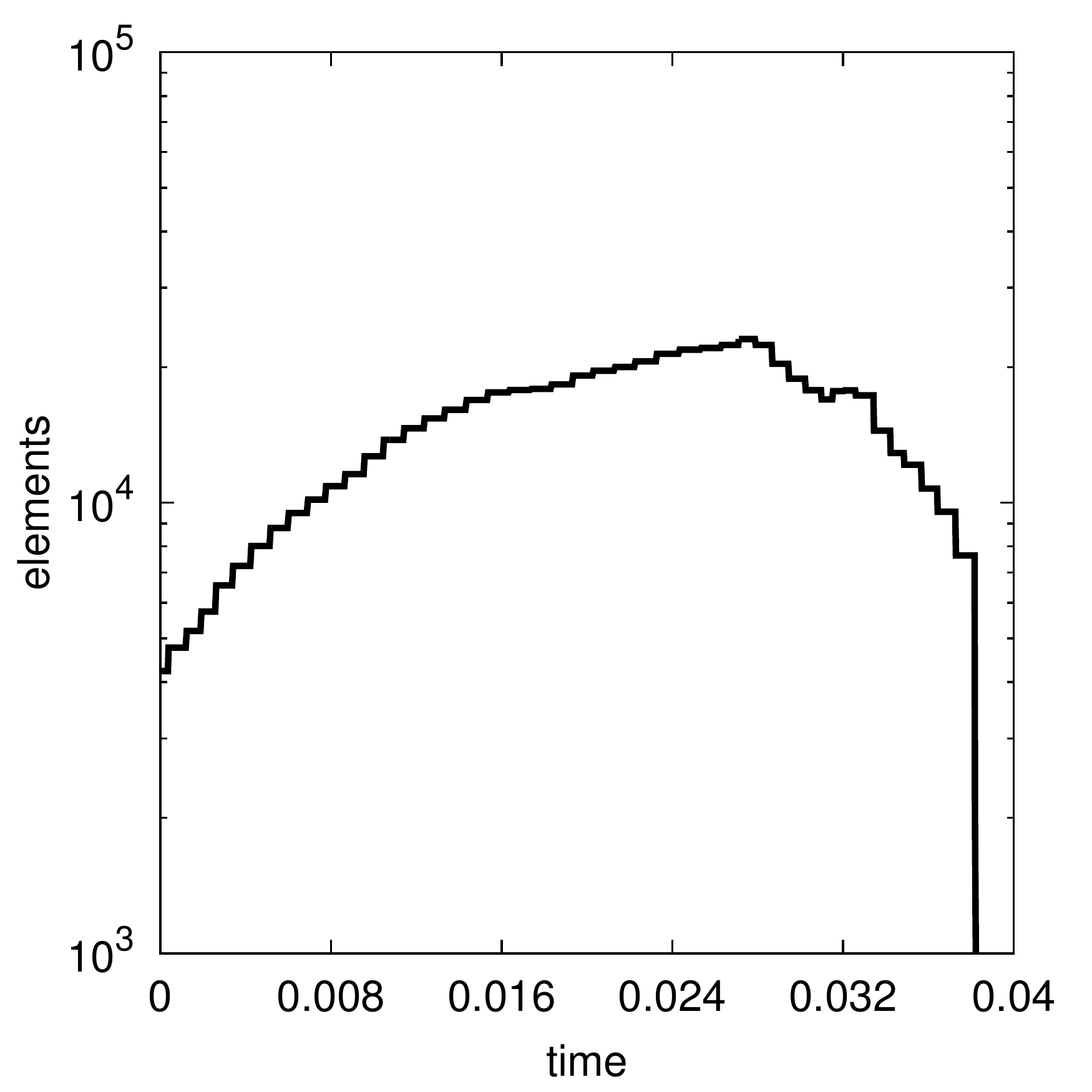}
  }
  \caption{Test case 1: space-time adaptive algorithm with
    $TOL_S=0.125,\,TOL_T=0.09375$: the relative space and time error estimators (top left), the
    evolution of the time step (top right), and number of elements (bottom).}
  \label{fig-dep5-adapt-step-err}
\end{figure}

\begin{table}[ht]
  \footnotesize
  \centering
  \begin{tabular}{cc*7{S[table-format=1.0(1)e2,round-precision = 2]}}
    $TOL_S$ &$TOL_T$ &$\eta^S$ &{$\tilde{\eta}^T$} &$\eta^{1,T}$ &$\eta^{2,T}$ &$\eta^{4,T}$\\
    \hline
    0.5     &0.375    &0.137597632233182   &0.236009005671297  &0.00252858662614283   &0.00883502157841513  &0.235830087809592 \\ 
    0.25    &0.1875   &0.0697810424661404  &0.120835828871556  &0.0014580128436934    &0.00271879468296343  &0.120796442845007 \\ 
    0.125   &0.09375  &0.0364521117192907  &0.0606894376243882 &0.000752260513321911  &0.00107121210547335  &0.0606753235685232 \\
    0.0625  &0.046875 &0.0192086236144145  &0.0304959648785695 &0.000366749702842699  &0.000443369787470237 &0.0304905363767379 \\
    0.03125 &0.0234375&0.00977412972805375 &0.0152282603025142 &0.000175911383826608  &0.00016483673467284  &0.0152263501005458 \\
  \end{tabular}
  \caption{Estimated error after space-time adaptation for Test case 1 with
    $TOL_T=0.75\cdot TOL_S.$}
  \label{tab-est-adapt-dep5}
\end{table}

\begin{table}[ht]
  \footnotesize
  \centering
  \begin{tabular}{*{8}{c}}
    $TOL_S$ &remeshings &step adapts &step restarts &total steps &CPU time est (\%)\\
    \hline
    0.5     &21 &11 &11 &305  &0.22\\
    0.25    &38 &11 &11 &432  &0.18 \\
    0.125   &48 &9  &9  &611  &0.19\\
    0.0625  &62 &10 &10 &862  &0.24\\
    0.03125 &73 &11 &11 &1219 &0.28\\
  \end{tabular}
  \caption{Mesh and time-step statistics for space-time adaptation for Test case
    1 with $TOL_T=0.75\cdot TOL_S$.}
  \label{tab-stats-dep5}
\end{table}


\subsection{Test case 2}
\label{subsec:block4}

\begin{sloppypar}

The goal of the following test case is to determine the gain in efficiency using
space-time adaptation. In particular, we are interested in situations where the
time error varies more rapidly and in magnitude than in the previous test case.
We choose the same domain, initial data and reaction term as in (\ref{eqn:dep5})
and take $T=0.01$. We replace the diffusion term with $-\text{div}(A\nabla u)$
where
\begin{align*}
  A(x,y)
  &=1 + 19\left(1-\tanh^2\left(20\left(\left(x^2+y^2\right)^{1/2}-0.3\right)\right)\right).
\end{align*}
The wave slows down and spreads outwards as it enters the region with higher
diffusion where $\left(x^2+y^2\right)^{1/2}\approx 0.3$, and then picks up speed
as the wave exits this region. The theory in this case follows as in Section
\ref{sec:apost} with appropriate changes to the definition of the edge and
element residuals. To simplify the computation of the estimator, on each element
$K$ we replace $A$ by its value $A_K$ at the barycenter of $K$. Then we define
the new edge residual $r_K^A(u_h)=A_K[\nabla u_h].$ This definition is somewhat
different from what appears in the literature, for instance in \cite{pic03a} and
\cite{micper10}, but was easier to implement.
\end{sloppypar}

As for test case 1, we compute a reference solution using mesh adaptation with
$TOL_S=0.015625$, with meshes ranging from $55000$ to $250000$ elements, and a
relatively small constant time step $\tau_{ref}=1.25\times 10^{-5}$. A solution
is then computed on a uniform mesh with $h=0.025$, with $6400$ elements, and
time step $\tau=2\times 10^{-4}$ and the estimated error and exact error are
approximated from the reference solution, plotted in Figure
\ref{fig-block4-const-err} (left). We see that the estimated error $\eta^S$ is
generally close to the exact error. The time estimator $\tilde{\eta}^T$
decreases as the wave enters the area of higher diffusion and slows down. The
estimator then increases as the wave exits this region and speeds up
again. Additionally, we compute the estimated and exact error when applying mesh
adaptation with $TOL_S=0.25$, with meshes of about $1000$ elements, and constant
time step $2\times 10^{-4}$, plotted in Figure \ref{fig-block4-const-err}
(right). The observed trend for $\tilde{\eta}^T$ is similar to that for uniform
meshes. The estimated error $\eta^S$ is again close to the exact error, both
dropping whenever the mesh is adapted. As noted for test case 1, the space error
grows more quickly for adapted meshes compared to uniform meshes since the wave
moves beyond the refined region of the mesh. Finally, Figure
\ref{fig-block4-adapt-err} illustrates the evolution of the error and time step
applying the space-time adaptation algorithm with $TOL_S=0.25,\,TOL_T=0.375,$
also resulting in meshes of about $1000$ elements.

\begin{figure}
  \centering
  \subfloat{
    \label{fig:block4-err-uniform}
    \includegraphics[width=0.45\textwidth]{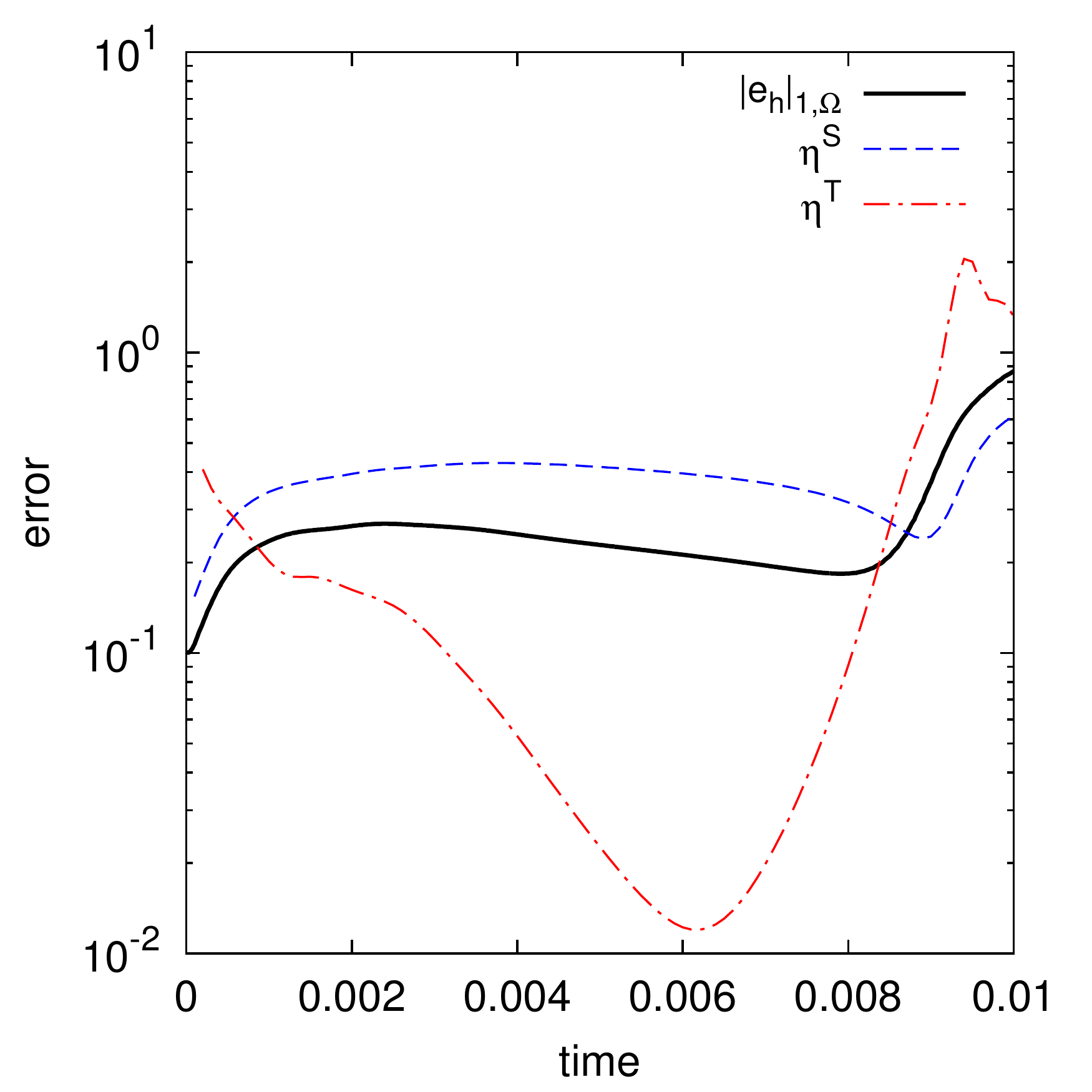}
  }
  \subfloat{
    \label{fig-block4-err-adspace}
    \includegraphics[width=0.45\textwidth]{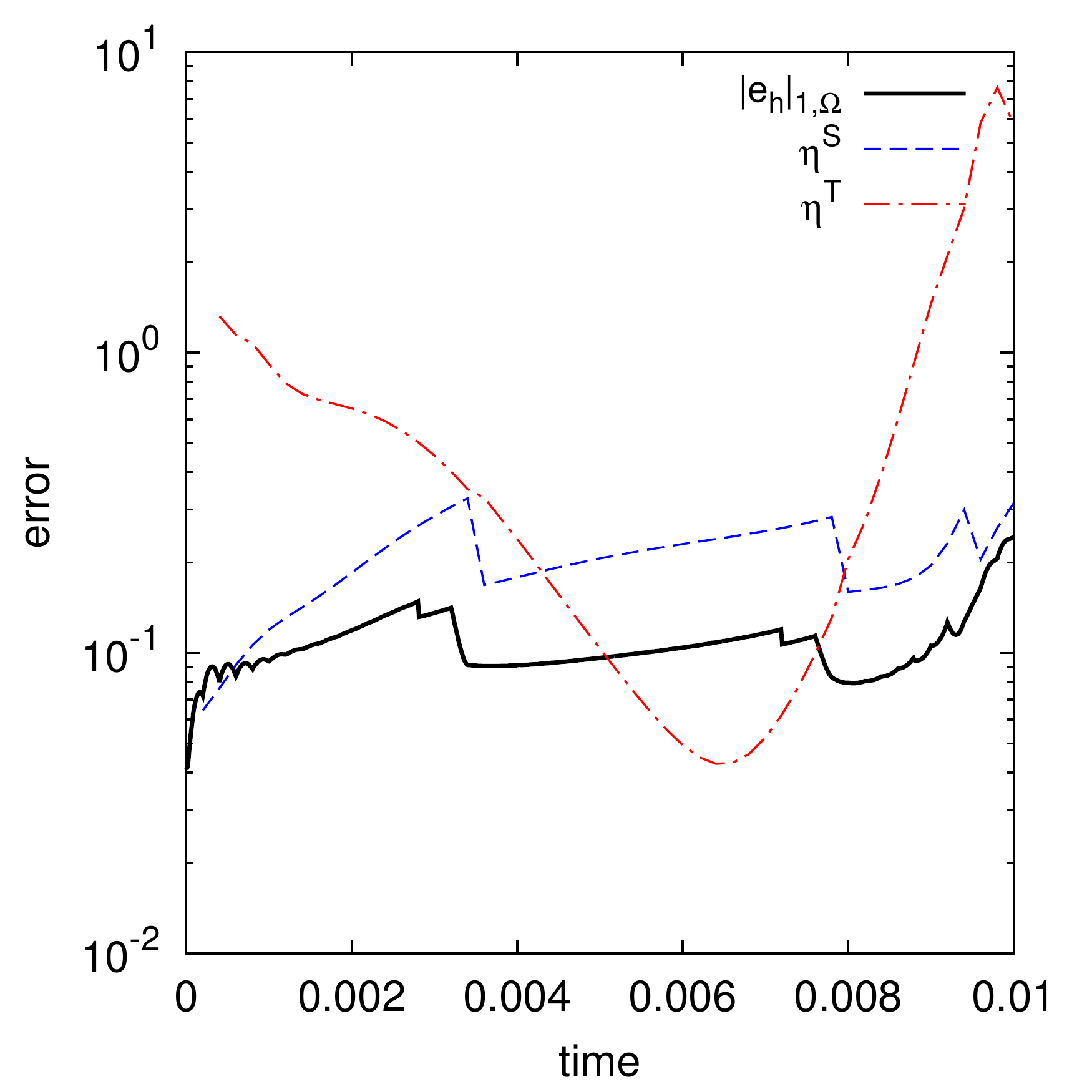}
  }
  \caption{Plot of error for Test case 2. Plot of error in time on a uniform
    mesh with $h=0.025, \tau=2\times 10^{-4}$ (left) and using space-only adaptation with
    $TOL_S=0.25, \tau=2\times 10^{-4}$ (right).}
  \label{fig-block4-const-err}
\end{figure}

\begin{figure}
  \centering
  \subfloat{
    \label{fig-block4-err-adspacetime}
    \includegraphics[width=0.45\textwidth]{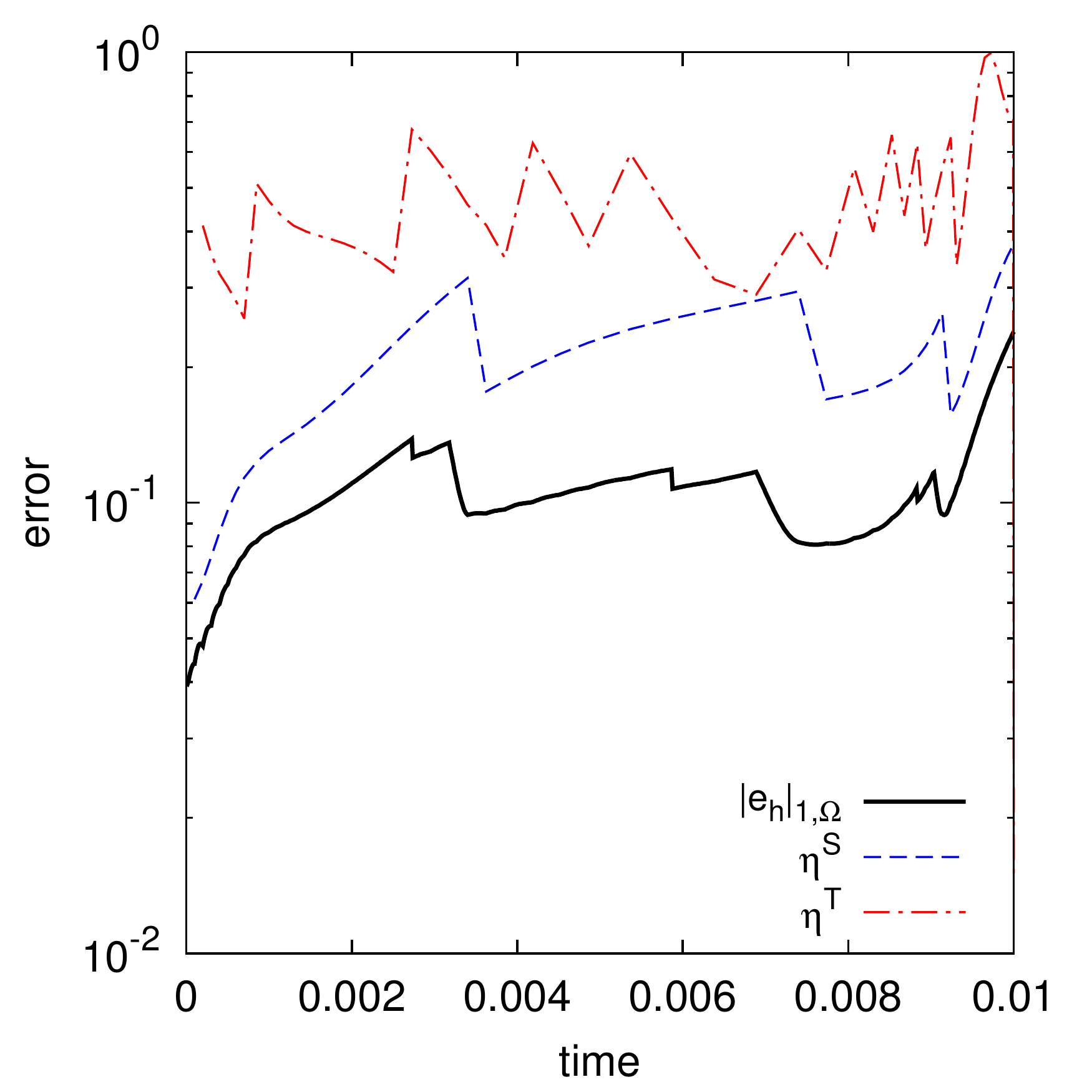}
  }
  \subfloat{
    \label{fig-block4-tstep}
    \includegraphics[width=0.45\textwidth]{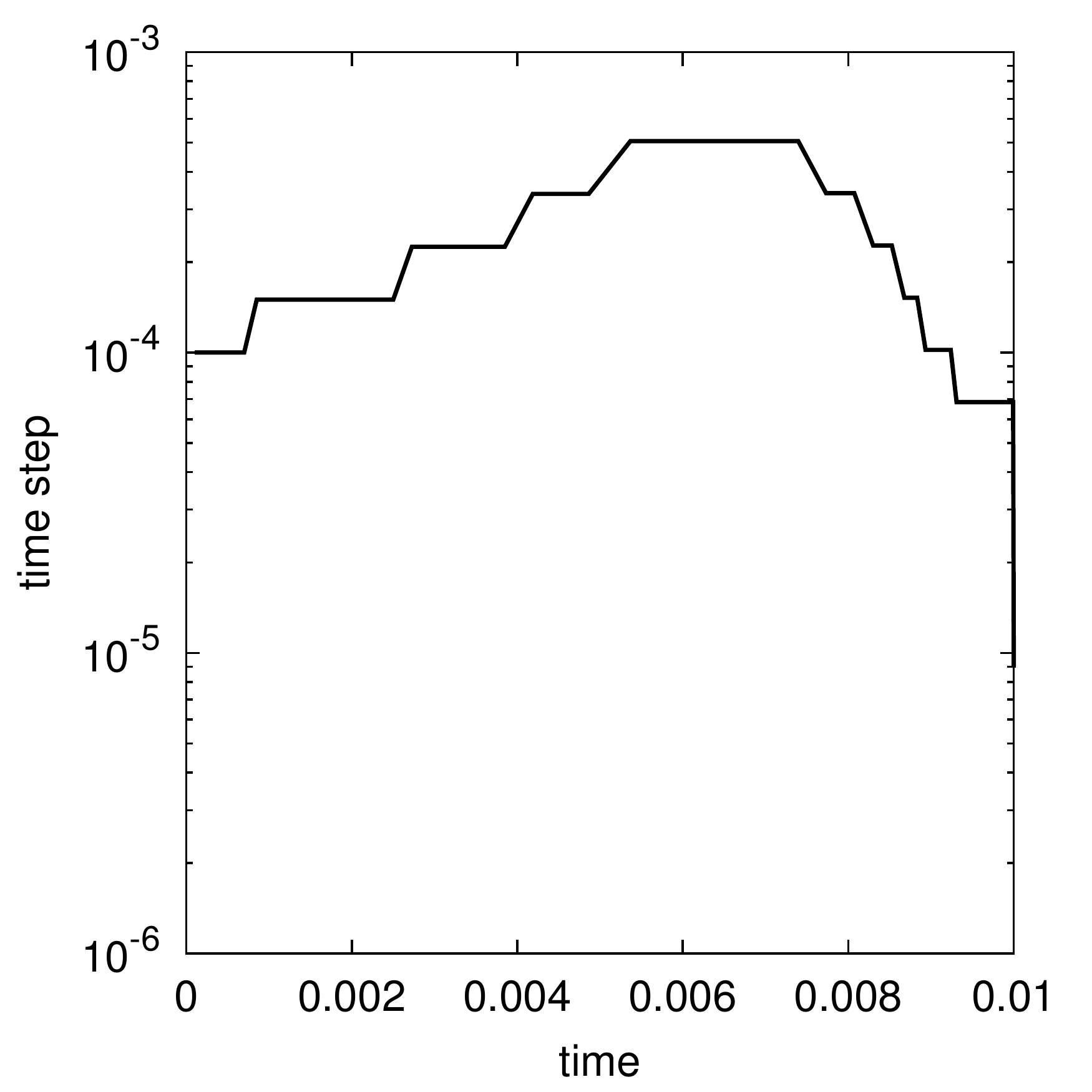}
  }
  \caption{Test case 2: plot of the error (left) and time step (right) using
    space-time adaptation with $TOL_S=0.25$ and $TOL_T=0.375$.}
  \label{fig-block4-adapt-err}
\end{figure}

The efficiency of the adaptive method is assessed by determining the level of
error that can be obtained in a given total CPU time. We compare computations
with uniform meshes with constant time step, adapted meshes with constant time
step, and applying the full space-time adaptation algorithm. For a given spatial
resolution (i.e. a fixed uniform mesh, or fixed error tolerance $TOL_S$), we
compute the solution with enough time steps in order to assess the general CPU
time vs. error trend, which is recorded in Figure
\ref{fig:block4-errvscpu_v2}. For instance, the curve for the uniform meshes
with $h=0.0125$ represents computations using six different time steps, ranging
from $4\times 10^{-4}$ on the left to $1.25\times 10^{-5}$ on the right. As the
time step is decreased, the CPU time increases as expected, while the error
level decreases and eventually stabilizes when the time step reaches about
$1\times 10^{-4}$. We infer that for $h=0.0125$ a time step between
$1\times 10^{-4}$ and $2\times 10^{-4}$ provides a good compromise between
achieving the lowest error and minimizing the CPU time. We have applied this
criteria to record some best-case results for each level of $h$ in Table
\ref{tab:block-errcpu}. The curves in Figure \ref{fig:block4-errvscpu_v2} for
the adapted meshes with constant time step are obtained in the same way for each
tolerance $TOL_S$. The curves for the space-time adapted solutions for a given
$TOL_S$ are obtained by varying the tolerance $TOL_T$. For instance, the curve
for $TOL_S=0.125$ corresponds to six computations with $TOL_T$ ranging from
$0.375$ on the left to $0.01171875$ on the right. We have similarly recorded
some best-case results for the adapted methods in Table \ref{tab:block-errcpu}.

We observe from both Figure \ref{fig:block4-errvscpu_v2} and Table
\ref{tab:block-errcpu} that applying either adaptive method is generally more
efficient than using uniform meshes. For instance, from the table, if we compare
the adaptive results with $TOL_S=0.25$ to the computation with a uniform mesh
with $h=0.0125$, $\tau=1\times 10^{-4}$, we see that the using a uniform mesh
takes about $8$ times the total CPU time to achieve a comparable level of error
to the adaptive methods. It should be noted that when solving with the adaptive
methods, between $30-50$\% of the total CPU time is spent in the adaptive phase
of the solution-adaptation loop. This includes computing the error estimator,
constructing the metric and adapting the mesh. It is likely that the CPU time
could further be reduced by optimizing the adaptive algorithm.

There does not seem to be a significant difference in efficiency when employing
mesh adaptation with either an adapted or a constant time step. However, it
should be noted that one of the advantages of using the space-time adaptive
method is that there is no need to guess the appropriate time step to be
used. For the test cases considered in this paper, we found that choosing
$TOL_T$ close to $TOL_S$ consistently gives a good result.

\begin{figure}[ht]
  \centering
  \subfloat{
    \includegraphics[width=0.6\textwidth]{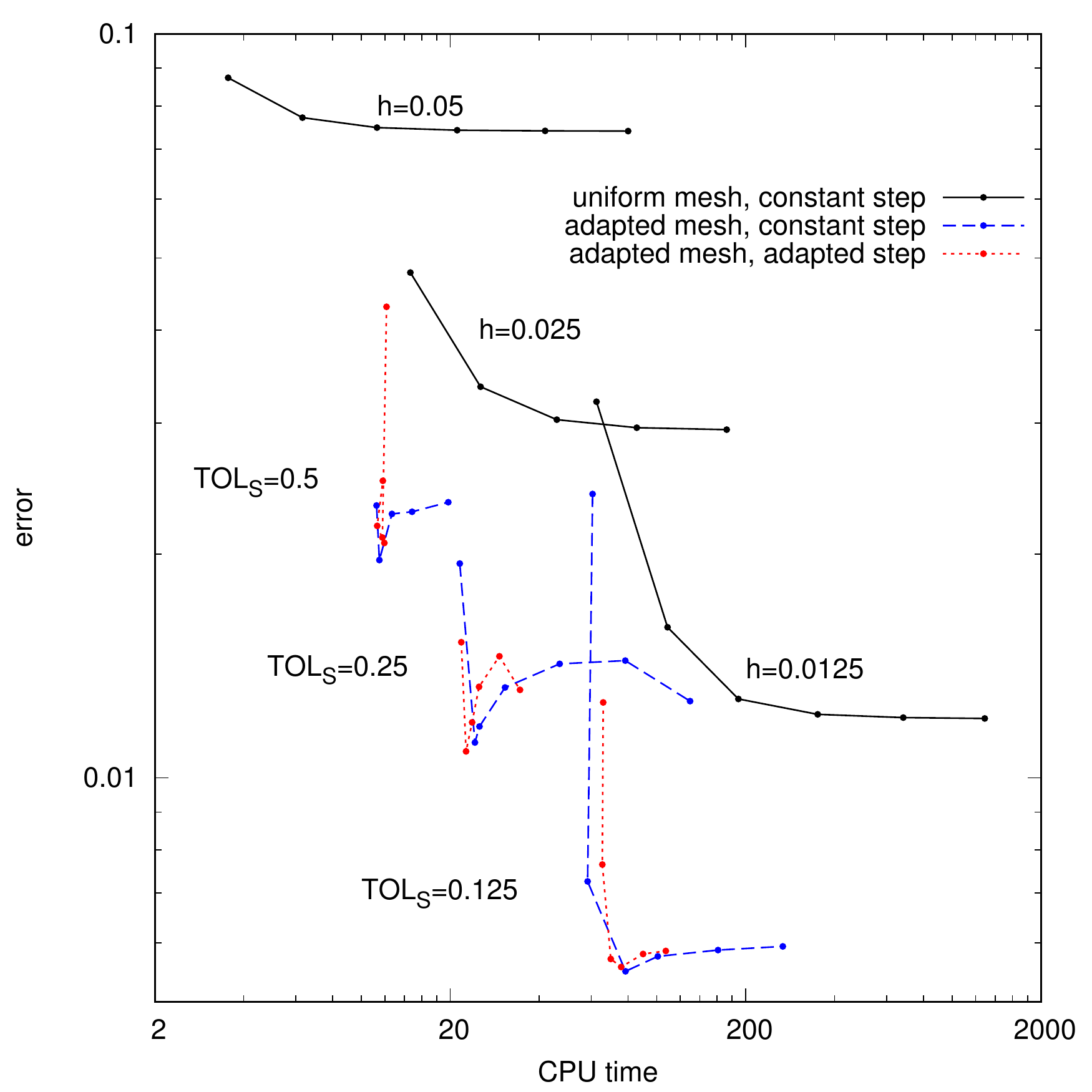}
  }
  \caption{Test case 2: CPU time vs. error.}
  \label{fig:block4-errvscpu_v2}
\end{figure}

\begin{table}
  \footnotesize
  \centering
  
  \caption*{Uniform meshes and constant time step.}
  \vspace{-7pt}
  \begin{tabular}{c
    S[table-format=1.0(1)e2,round-precision = 2]
    S[scientific-notation=false,round-precision=4]
    S[scientific-notation=false,round-precision=2]
    S[scientific-notation=false,round-precision=1]}
    $h$    &{$\tau$} &{$|e_h|_{1,\Omega}$} &{$ei^{S}$} &{CPU(sec)} \\
    \hline
    0.05   &4.00e-4  &0.0873513969141388   &0.847212   &3.541      \\
    0.025  &2.00e-4  &0.0335717357015931   &1.13001    &25.3278    \\
    0.0125 &2.00e-4  &0.0159385569219136   &1.17915    &108.773    \\
    0.0125 &1.00e-4  &0.0127684461004017   &1.47236    &189.002    \\
  \end{tabular}

  \vspace{7pt}

  \caption*{Adapted meshes and constant time step.}
  \vspace{-7pt}
  \begin{tabular}{c
    S[table-format=1.0(1)e2,round-precision = 2]
    S[scientific-notation=false,round-precision=4]
    S[scientific-notation=false,round-precision=2]
    S[scientific-notation=false,round-precision=1]}
    $TOL_S$ &{$\tau$} &{$|e_h|_{1,\Omega}$}  &{$ei^{S}$} &{CPU(sec)} \\
    \hline
    0.5     &0.000200 &0.019627991815999     &1.96293    &11.5074    \\
    0.25    &0.000200 &0.0111595204800256    &1.90974    &24.2181    \\
    0.125   &0.000100 &0.00549700541232919   &1.98303    &78.3333    \\
  \end{tabular}

  \vspace{7pt}

  \caption*{Adapted meshes and adapted time step.}
  \vspace{-7pt}
  \begin{tabular}{cc
    S[scientific-notation=false,round-precision=4]
    S[scientific-notation=false,round-precision=2]
    S[scientific-notation=false,round-precision=1]}
    $TOL_S$ &{$TOL_T$}  &{$|e_h|_{1,\Omega}$} &{$ei^{S}$} &{CPU(sec)} \\
    \hline
    0.5     &0.375      &0.0218271382653028   &1.89914    &11.3081  \\
    0.25    &0.375      &0.0108539190820204   &2.05535    &22.6158  \\
    0.125   &0.046875   &0.00556897584502264  &2.00226    &75.8097  \\
  \end{tabular}

  \vspace{3pt}

  \caption{Test case 2: error, effectivity index and total CPU time.}
  \label{tab:block-errcpu}

\end{table}


\subsection{Test case 3}
\label{subsec:FHN1}

Take the domain $\Omega=(0,100)\times(0,100),\,T=350$. Solve the FitzHugh-Nagumo
model with $\epsilon=0.01$, $\kappa=0.16875$, $a=0.25$, and initial condition
\begin{align*}
  u_0
  =0.5-\frac{1}{\pi}\arctan\left(\sqrt{x^2+y^2}-200\right),
  \quad\quad
  w_0=0.
\end{align*}
At time $t=0$ the wave is activated in the lower left corner of the domain, and
a circular wave action potential results moving away from the origin.


\subsubsection{Performance of mesh adaptation}
\label{subsubsec:mono-adalg}

In Figure \ref{fig:FHN1-sol-mesh} we show an example solution and mesh using the
adaptive method. The mesh elements are heavily concentrated in the
depolarization and repolarization regions of the transmembrane potential, while
there are few elements in the region ahead of the wave-front, where both
variables are nearly constant. Additionally, the regions corresponding to the
plateau and recovery are reasonably well refined, capturing the slow variation
of the recovery variable. Figure \ref{fig:FHN1-zoom-cont} shows a zoom on the
wave front, illustrating that the mesh fits the variation of both variables.

\begin{figure}
  \centering
  \subfloat{
    \includegraphics[width=0.42\textwidth]{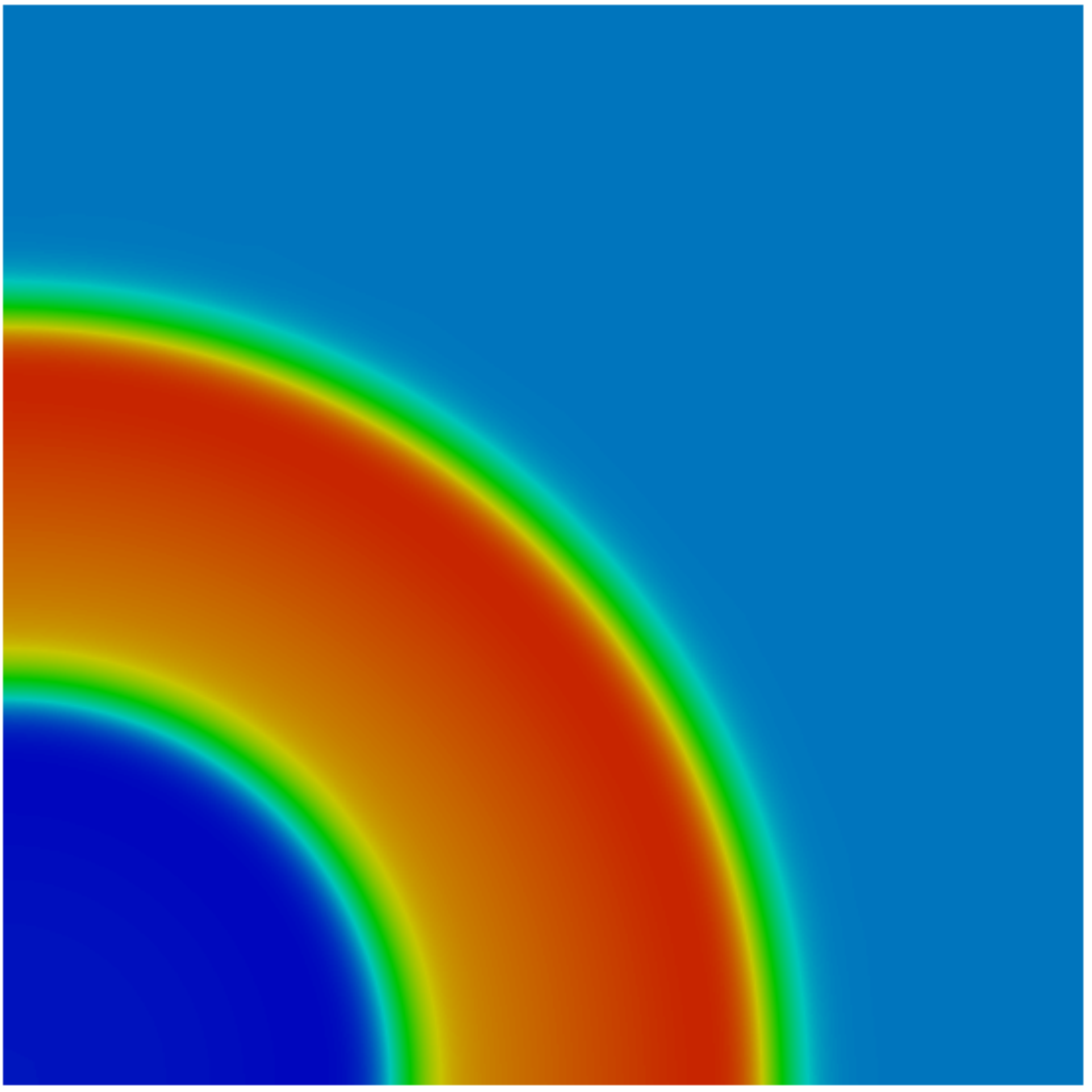}
  }
  \subfloat{
    \includegraphics[width=0.1\textwidth]{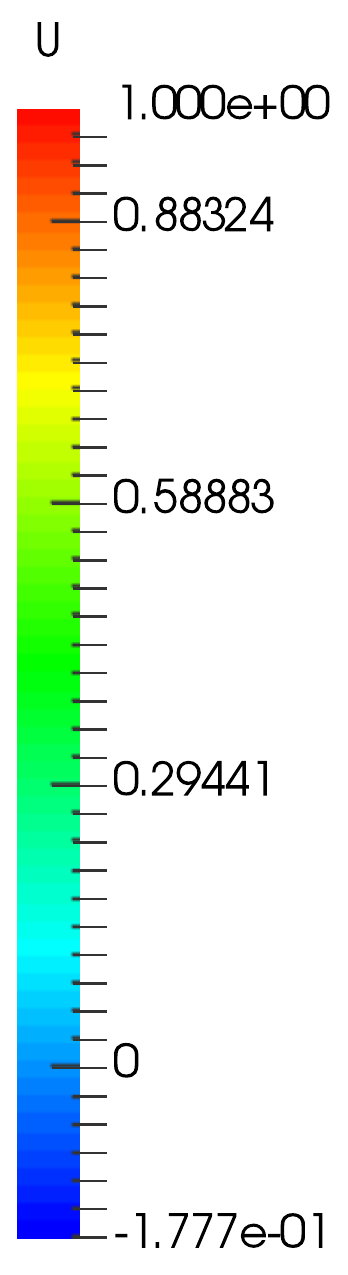}
  }
  \subfloat{
    \includegraphics[width=0.42\textwidth]{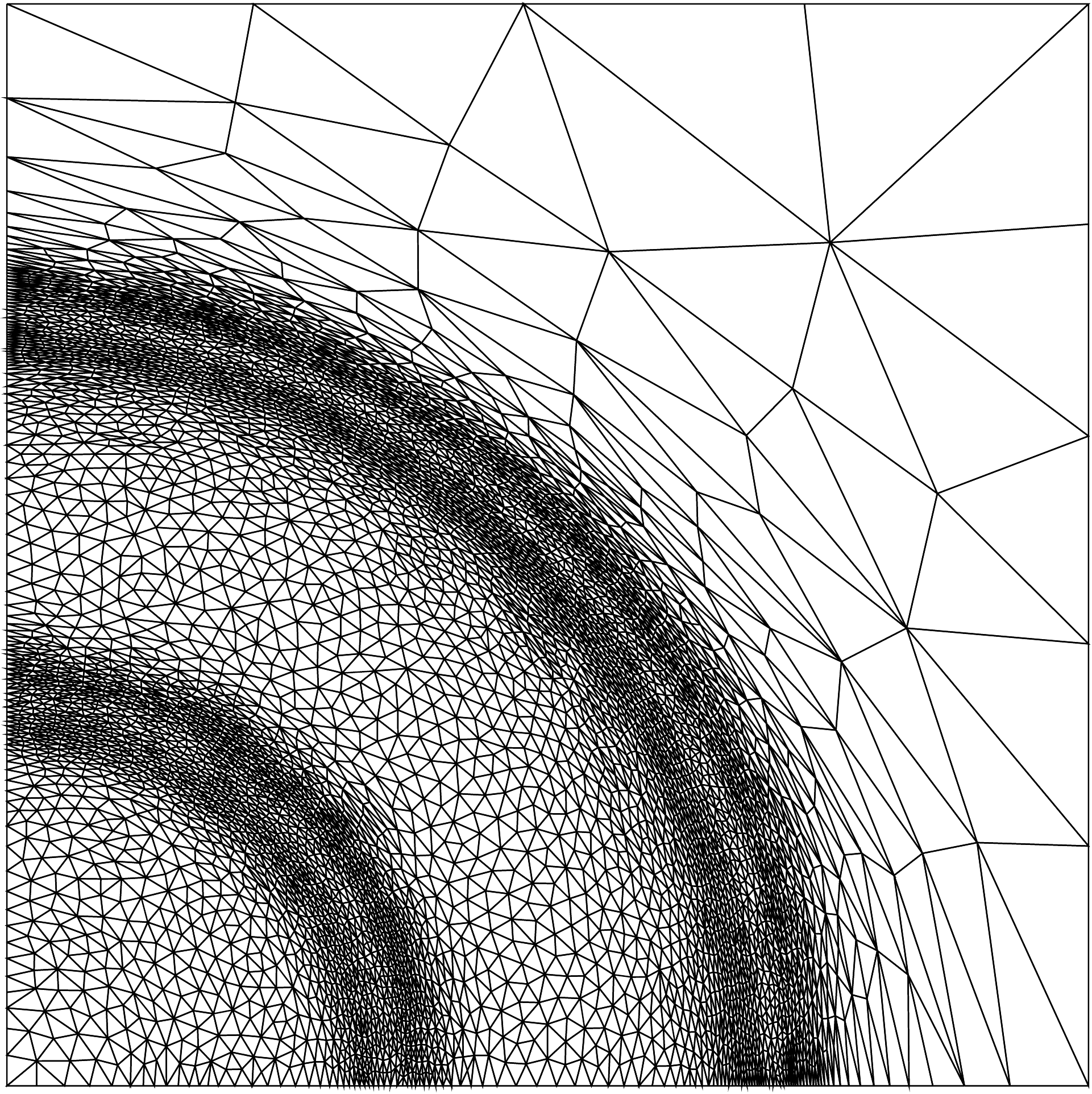}
  }
  \caption{Test case 3: transmembrane potential $u$ (left) and corresponding
    adapted mesh (right) when applying the adaptive method with
    $TOL_S^U=0.125,\,TOL_S^W=0.0125$ and $\tau=0.5$ at $t=175$.}
  \label{fig:FHN1-sol-mesh}
\end{figure}

\begin{figure}
  \centering
  \subfloat{
    \includegraphics[width=0.32\textwidth]{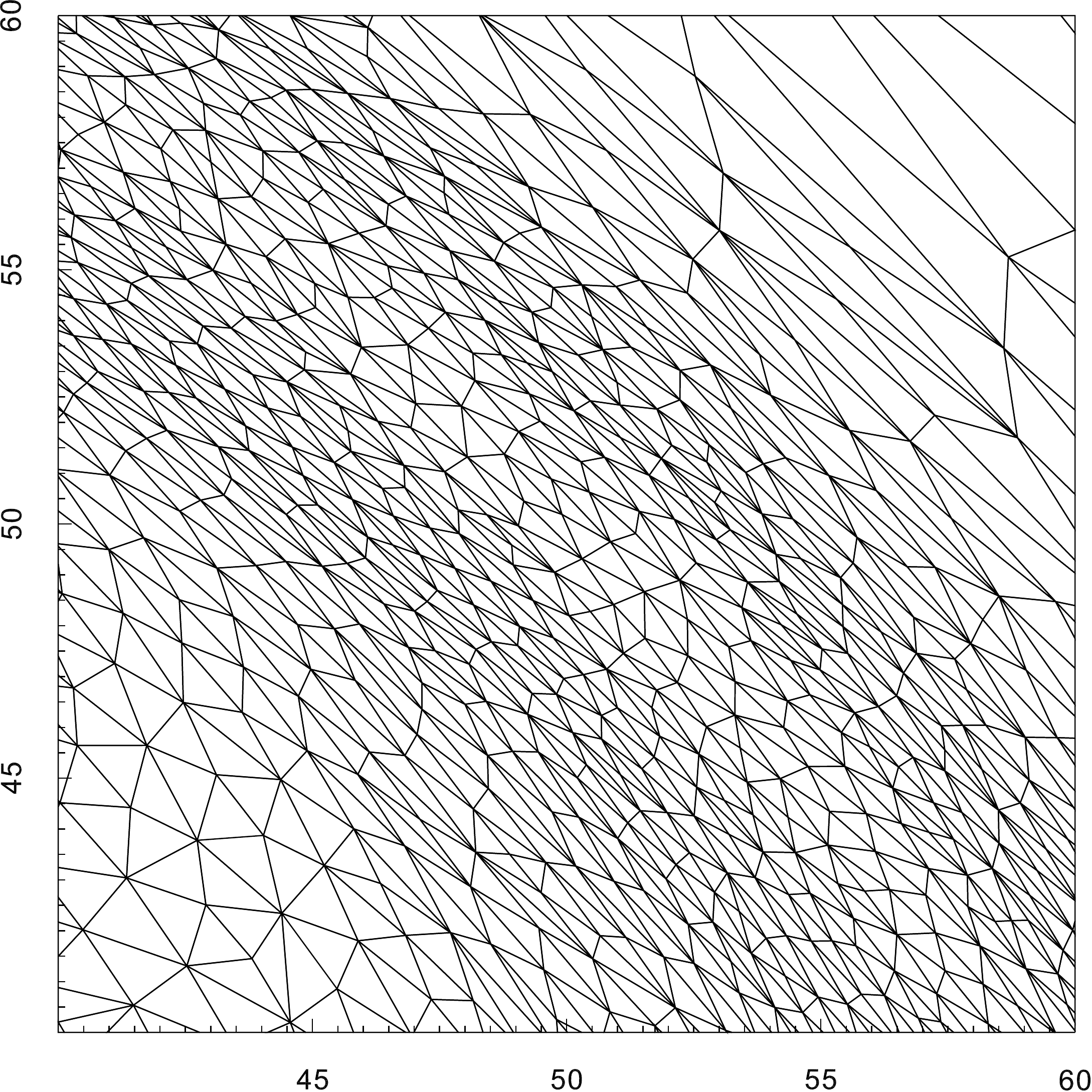}
  }
  \subfloat{
    \includegraphics[width=0.32\textwidth]{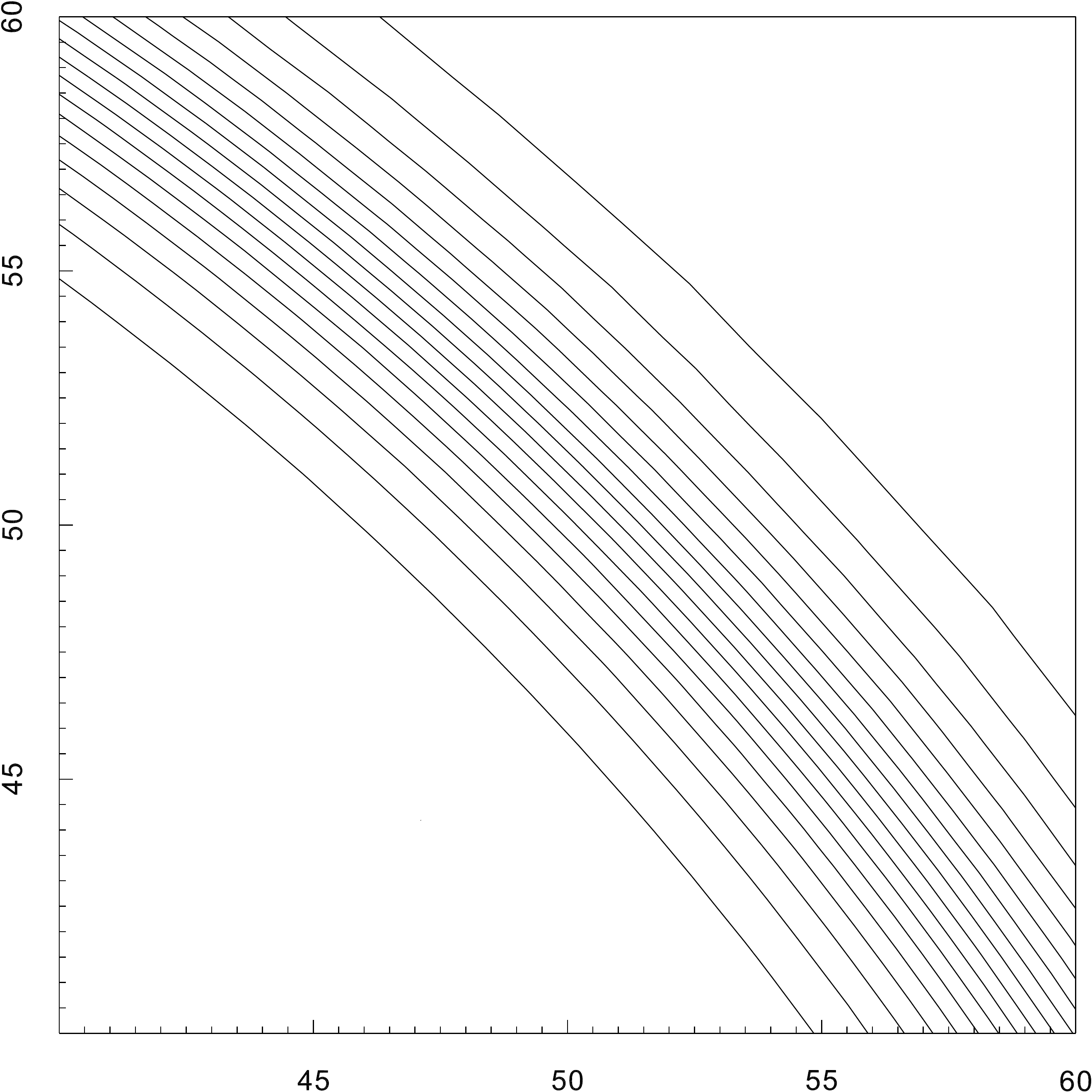}
  }
  \subfloat{
    \includegraphics[width=0.32\textwidth]{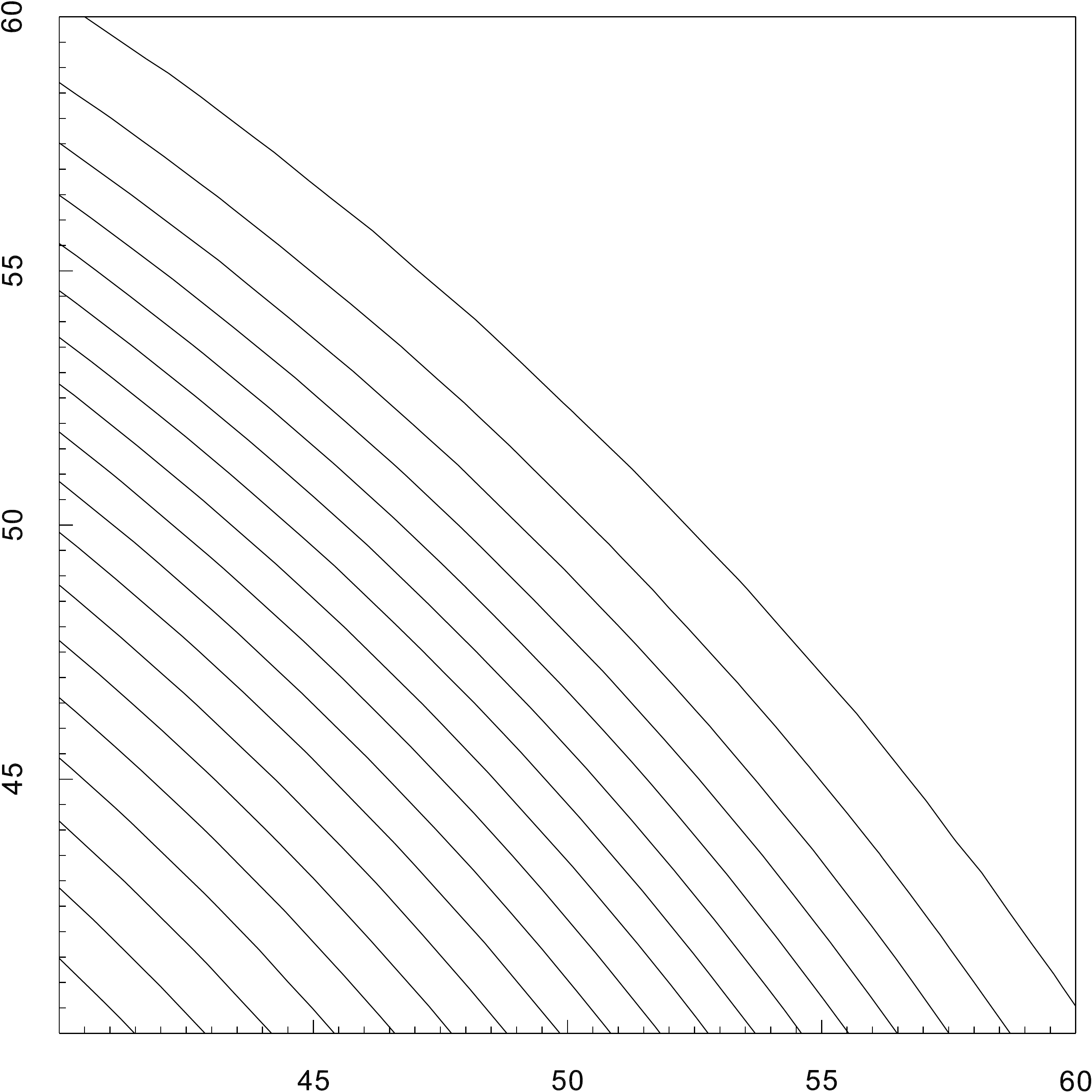}
  }
  \caption{Test case 3: zoom of the mesh at wave-front at time $t=175$ (left),
    and contours of $u$ (centre) and $w$ (right).}
  \label{fig:FHN1-zoom-cont}
\end{figure}


\subsubsection{Computation of the effectivity index}
\label{subsubsec:mono-effect-ind}

We estimate the following effectivity indices
\begin{align*}
  ei^{S,U}=\frac{\eta^{S.U}}{\vertiii{e_u}},
  \quad &ei^{S,W}=\frac{\omega^{S.W}}{\|e_w\|_{L^2(0,T;H)}}.
\end{align*}
Since we do not have an exact solution, to assess the robustness of the
estimators we compute a reference solution. The reference solution for the
scalar problem was computed using a fine adapted meshes. Here, it was decided
that a more efficient method to compute a reference solution was using higher
order $P_2$ elements on a fine uniform mesh. In what follows the reference
solution is computed on mesh with $h=0.625$, with $205441$ degrees of freedom,
and a constant time step $\tau=0.0625$. Figure \ref{fig:pointplot} provides some
heuristic evidence that the solutions are converging to the reference solution.

Some computations of the effectivity indices are shown in Table
\ref{tab:FHN1-err-eff-CPU} illustrating the reliability of the estimators for
this test case. Note that as we take finer uniform meshes, the effectivity
index $ei^{S,U}$ approaches the value reached for the adapted meshes, while
$ei^{S,W}$ also increases and becomes closer to $1$.

\begin{figure}[ht]
  \centering
  \subfloat{
    \includegraphics[width=0.45\textwidth]{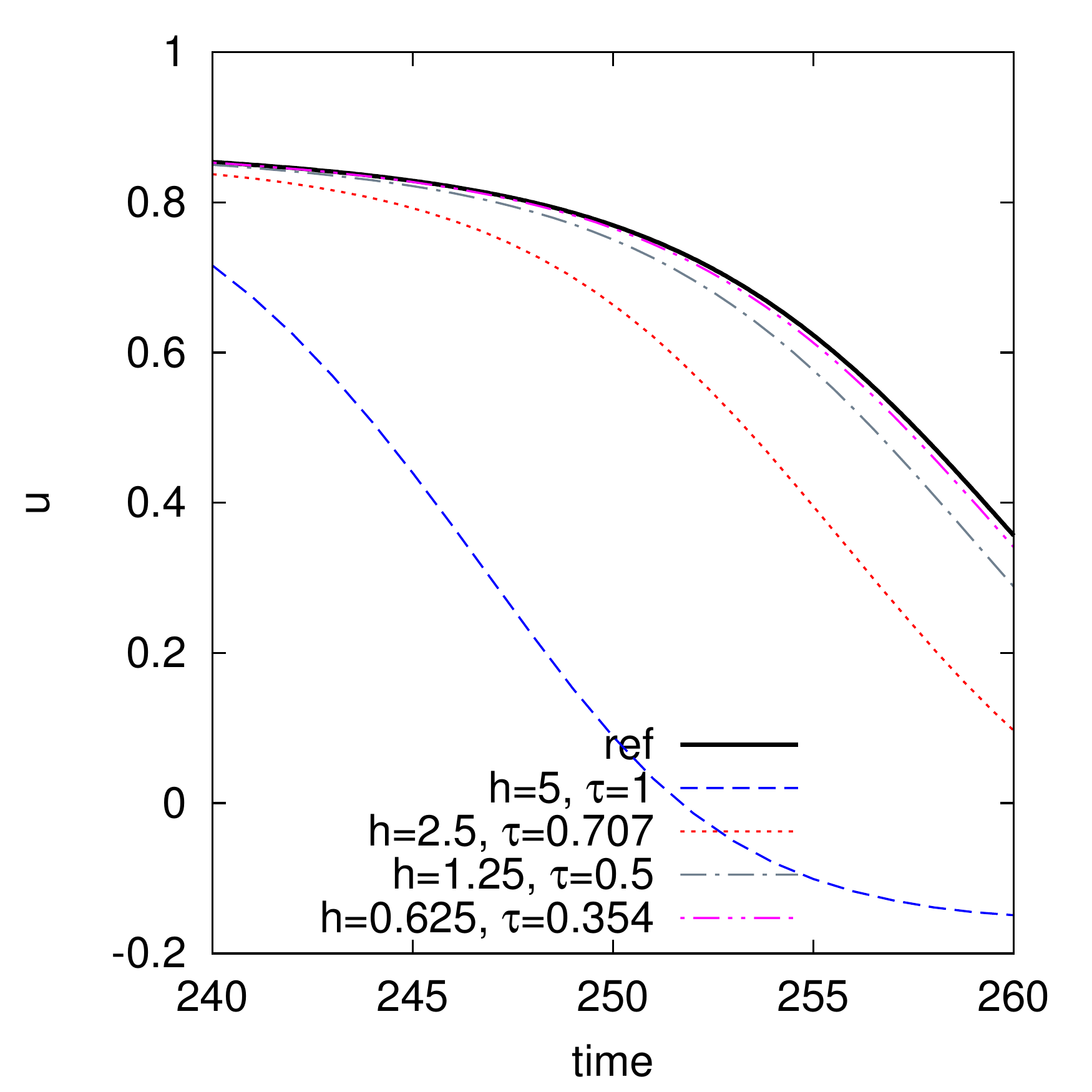}
  }
  \subfloat{
    \includegraphics[width=0.45\textwidth]{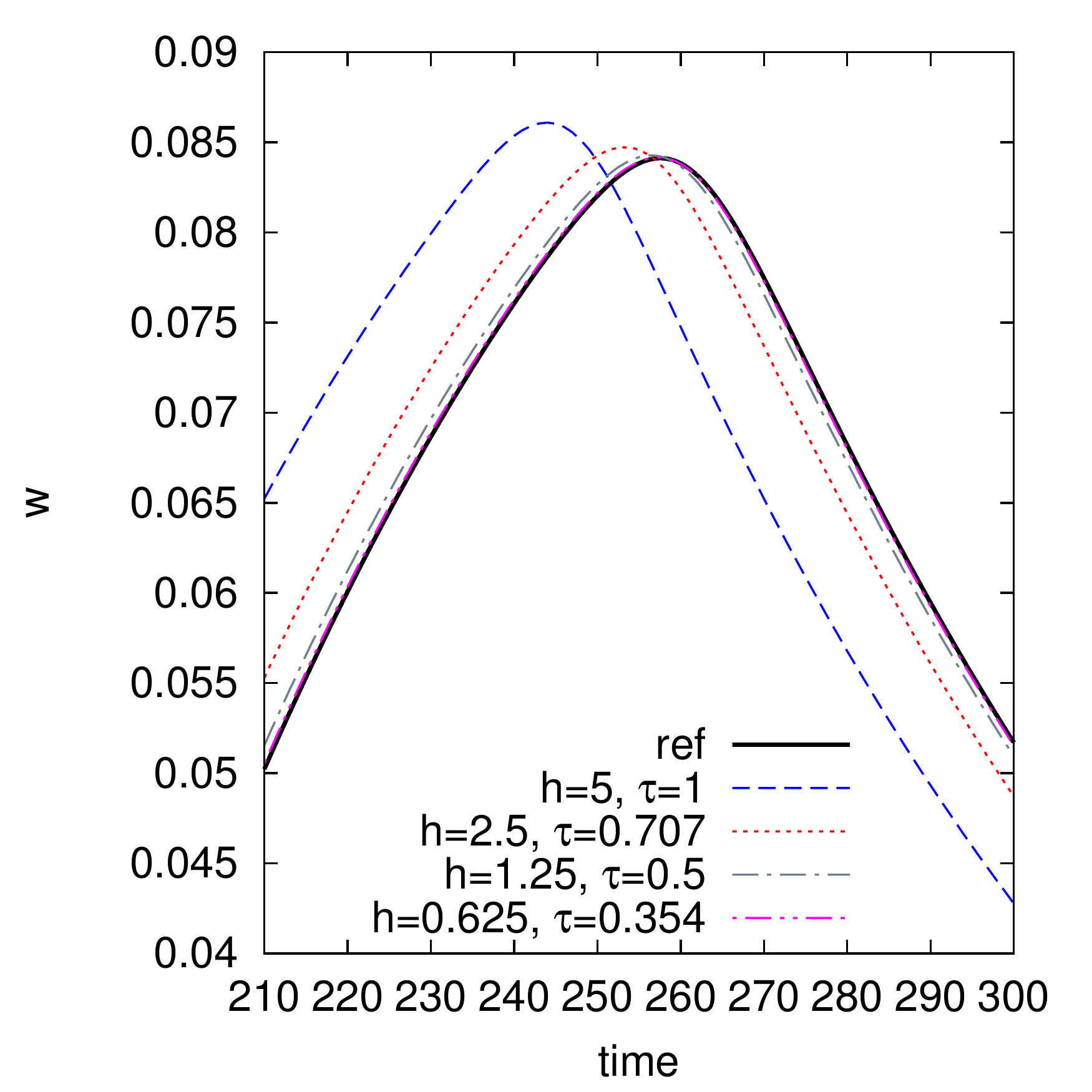}
  }
  \caption{Test case 3: comparison of reference solution with solutions on
    uniform meshes, plotted at point $(x,y)=(0.5,0.5)$ of $u$ (left) and $w$
    (right).}
  \label{fig:pointplot}
\end{figure}

\subsubsection{Efficiency of the adaptive method}
\label{subsubsec:effic-FHN1}

We assess the efficiency of the adaptive method by comparing the CPU time
vs. error level when using uniform and adapted meshes. The same approach is
taken as in Test case 2. That is, for a fixed uniform mesh, or a fixed space
error tolerance $TOL_S$, we compute the approximate solution using a variety of
fixed time steps in order to assess the general CPU time vs. error trend. The
results are reported in Figure \ref{fig:err-vs-cpu-FHN1}, with each curve
representing several computations for a fixed uniform mesh, or computations with
adapted meshes using a fixed error tolerance. For instance, the curve
representing $h=1.25$ consists of five computations with the time step ranging
from $4$ at the left-most point and $0.25$ at the right-most point. Table
\ref{tab:FHN1-err-eff-CPU} is obtained by choosing one or more points from each
curve from Figure \ref{fig:err-vs-cpu-FHN1}. These points are generally chosen
when the error first reaches the lowest point in the curve.

Before assessing the efficiency, we discuss a trend observed in the uniform mesh
computations. Note that the error level of the curves for $h=2.5$ and $h=1.25$
decrease up to a point as the time step is decreased, followed by a significant
increase in error. For $h=1.25$, the error in $w$ increases almost by a factor
of $4$ from the lowest point. Recall that we are computing a traveling wave
solution, and for large time step and coarse mesh the speed of the wave for the
computed solution is generally wrong. Since the wave speed is generally not
known, a strategy for computing an accurate solution on a given mesh might be to
refine the time step until the wave speed does not change
significantly. However, for this test case, the wave front for the approximate
solution gets ahead of or lags behind the actual wave front, and the exact error
varies significantly depending on the time step due to this displacement. The
value $\tau=2$ for this mesh roughly corresponds to the point where the computed
wave switches from being behind to ahead of the actual wave front. In general,
it is unrealistic to expect to match the wave speed except on very fine
meshes. This close matching of the wave speed for certain time steps partly
explains the apparent superconvergence of the error observed in the first three
rows of Table \ref{tab:FHN1-err-eff-CPU}. On the other hand, the computations
for the adapted meshes do not exhibit this trend in the wave speed. Moreover, in
Table \ref{tab:FHN1-err-eff-CPU} we see that if we decrease the tolerance
$TOL_S^U$ by half, we can expect the exact error in $u$ to also decrease by
half.

We compare the efficiency with Figure \ref{fig:err-vs-cpu-FHN1} and Table
\ref{tab:FHN1-err-eff-CPU}. First we note that computations with the coarser
uniform mesh are not very accurate. In order to match the accuracy for the
coarser adapted mesh computations with $TOL_S^U=0.5,$ $TOL_S^W=1$, which use at
most about $2600$ vertices, we need to use the finer uniform mesh with $h=1.25$,
with $12961$ vertices. With this choice of error tolerance, one can expect to
achieve the same level of error as with the finer uniform mesh in about $1/3$ of
the total CPU time. This same observation holds for the error in $u$ comparing
the uniform mesh computations with $h=0.625$ and the adapted meshes with
$TOL_S^U=0.25,$ $TOL_S^W=0.5$. For the error in $w$, we did not achieve the same
level of error for the adapted meshes as for the uniform meshes. However, note
that we only divided $TOL_S^W$ in half from the first level to the next two, and
as a result, only lowered the error in $w$ by half by the last row. At the same
time, the error for the last two uniform meshes decreased by a factor of four,
which is the expected asymptotic rate. For this test case, it was not found to
be worthwhile to attempt to set $TOL_S^W$ to be too low since it tended to
result in meshes that are almost uniform. As a result, we chose to refine both
error tolerances at the same rate. This choice leads to improved efficiency for
controlling the error in $u$ at the the expense of having a slightly larger
asymptotic error in $w$.

There is some overhead introduced by adapting the mesh. For this test case, the
additional operations involved for the adapted mesh, including computing the
estimator, adapting the mesh, and reinterpolation of the solution on the adapted
mesh, combine to about $15-35$\% of the total computation time, depending on how
often the adaptation occurs. Every time the mesh is adapted, the linear system
needs to be assembled again with a new sparsity pattern. Another observation is
that, in general, significantly more time is spent solving the linear system for
the adapted meshes compared to uniform meshes with a given number of DOFs.

\begin{figure}[ht]
  \centering
  \subfloat{
    \includegraphics[width=0.45\textwidth]{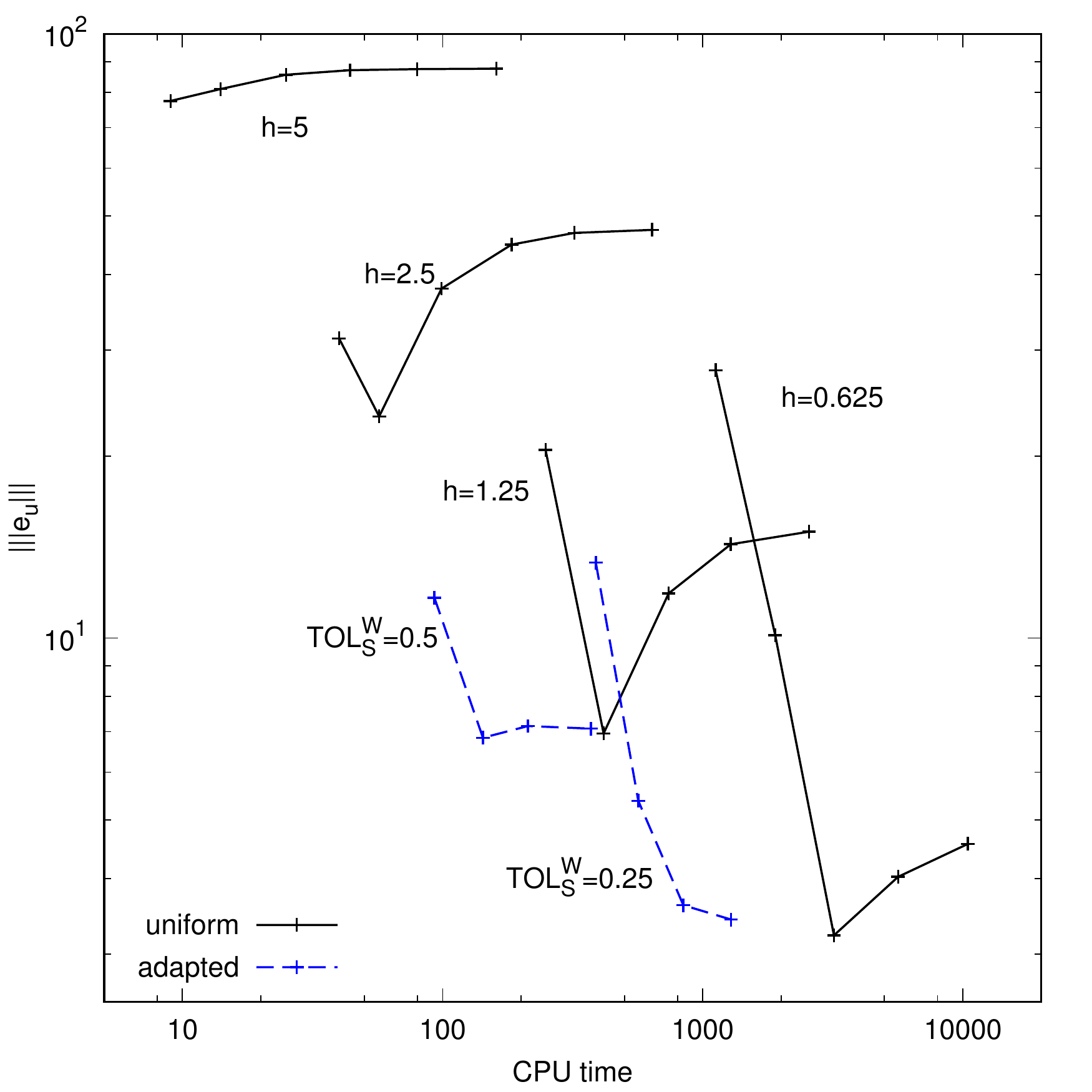}
  }
  \subfloat{
    \includegraphics[width=0.45\textwidth]{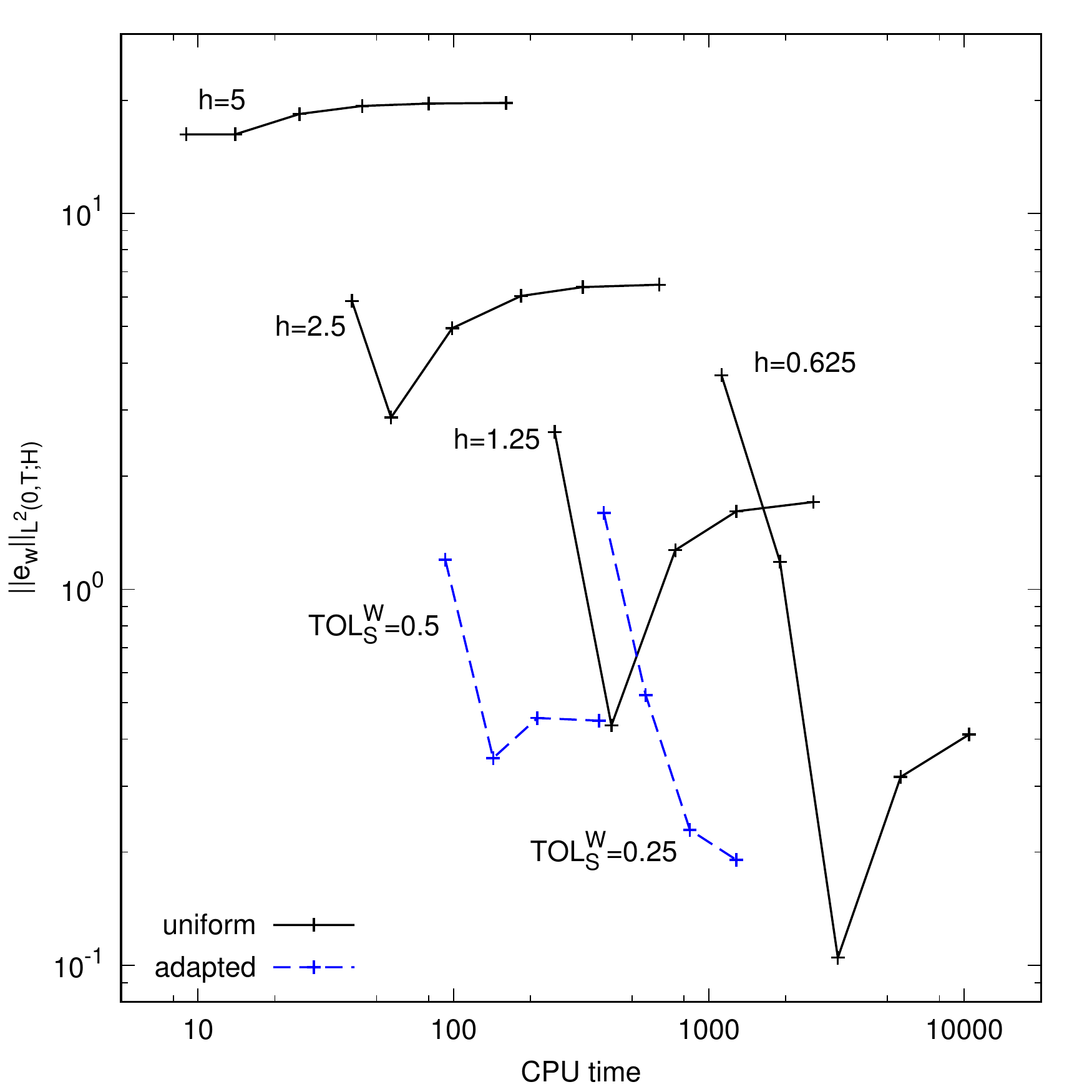}
  }
  \caption{Test case 3: error vs. CPU time for uniform and adapted mesh computations.}
  \label{fig:err-vs-cpu-FHN1}
\end{figure}

\begin{table}[ht]
  \footnotesize
  \centering

  \caption*{Uniform meshes.}
  \vspace{-7pt}
  \begin{tabular}{*{10}{c}}
    $h$ &$\tau$ &$\vertiii{e_u}$ &$\|e_w\|_{L^2(0,T;H)}$ &$ei^{S,U}$ &$ei^{S,W}$ &CPU(s) &DOFs\\
    \hline
    5     &8     &77.4 &16.2  &0.451 &0.174 &9   &841  \\
    2.5   &4     &23.2 &2.86  &1.20  &0.309 &57  &3281 \\
    1.25  &2     &6.99 &0.435 &2.38  &0.546 &416 &12961\\
    0.625 &1     &3.22 &0.104 &2.71  &0.578 &3199&51251\\
  \end{tabular}

  \vspace{7pt}

  \caption*{Adapted meshes.}
  \vspace{-7pt}
  \begin{tabular}{*{15}{c}}
    $TOL_S^U$ &$TOL_S^W$ &$\tau$ &$\vertiii{e_u}$ &$\|e_w\|_{L^2(0,T;H)}$ &$ei^{S,U}$ &$ei^{S,W}$
    &CPU(s) &DOFs \\
    \hline
    0.5   &1   &1    &6.84 &0.355 &2.89 &1.52  &143  &607  - 2633\\
    0.25  &0.5 &0.5  &3.61 &0.229 &2.86 &0.920 &843  &1880 - 7950\\
    0.25  &0.5 &0.25 &3.42 &0.190 &3.06 &1.11  &1284 &1947 - 7600\\
  \end{tabular}

  \vspace{7pt}

  \caption{Test case 3: error, effectivity index and total CPU time.}
  \label{tab:FHN1-err-eff-CPU}

\end{table}


\subsection{Test case 4}
\label{subsec:MS_tanh100_long}

Take the same domain and initial condition as Test case 3 with $T=500$, but with
the regularized Mitchell-Schaeffer model with regularization parameter
$\kappa=100$.  The parameters are taken from \cite{RB13}:
$\tau_{in}=0.315,\,\tau_{out}=5.556,\,\tau_{open}=94.942,\,\tau_{close}=168.5,\,u_{gate}=0.13$,
diffusion coefficient is $3.949$. The solution is a more realistic
representation of a heartbeat in terms of the duration of the phases of the
action potential. The action potential is activated in the lower left corner of
the domain and the wave radiates away from the origin until about time $t=60$
when the entire domain is depolarized. A long plateau follows in which the gate
$w$ closes and the inward and outward currents are roughly balanced.  At about
$t=290$ the outward current begins to dominate and the region is completely
repolarized by about $t=360$. As $u$ drops below $u_{gate},$ the gate slowly
opens. See Figure \ref{fig:MS_tanh100_long-phases} for the solution profile at
different phases.

\begin{figure}
  \centering
  \subfloat{
    \includegraphics[width=0.23\textwidth]{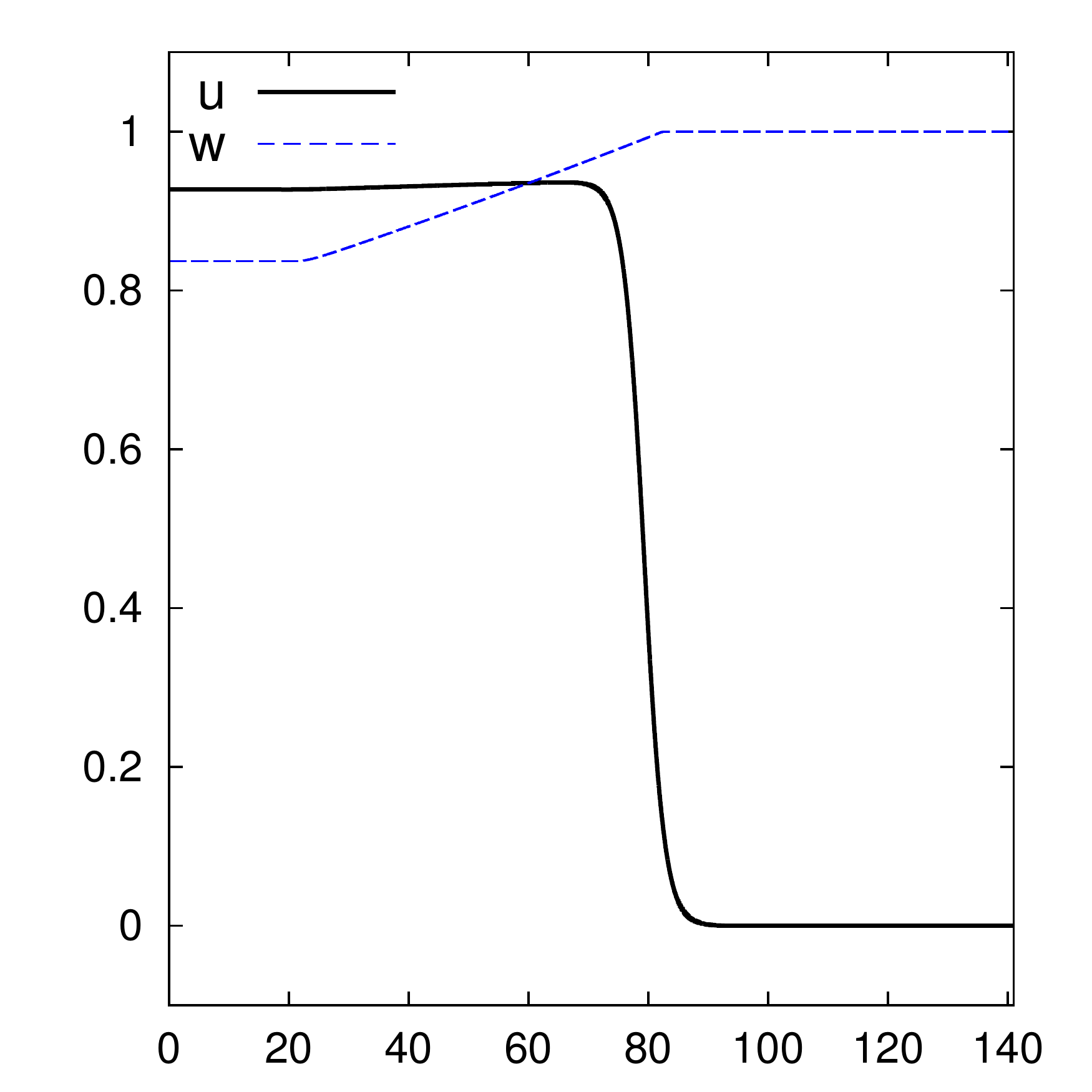}
  }
  \subfloat{
    \includegraphics[width=0.23\textwidth]{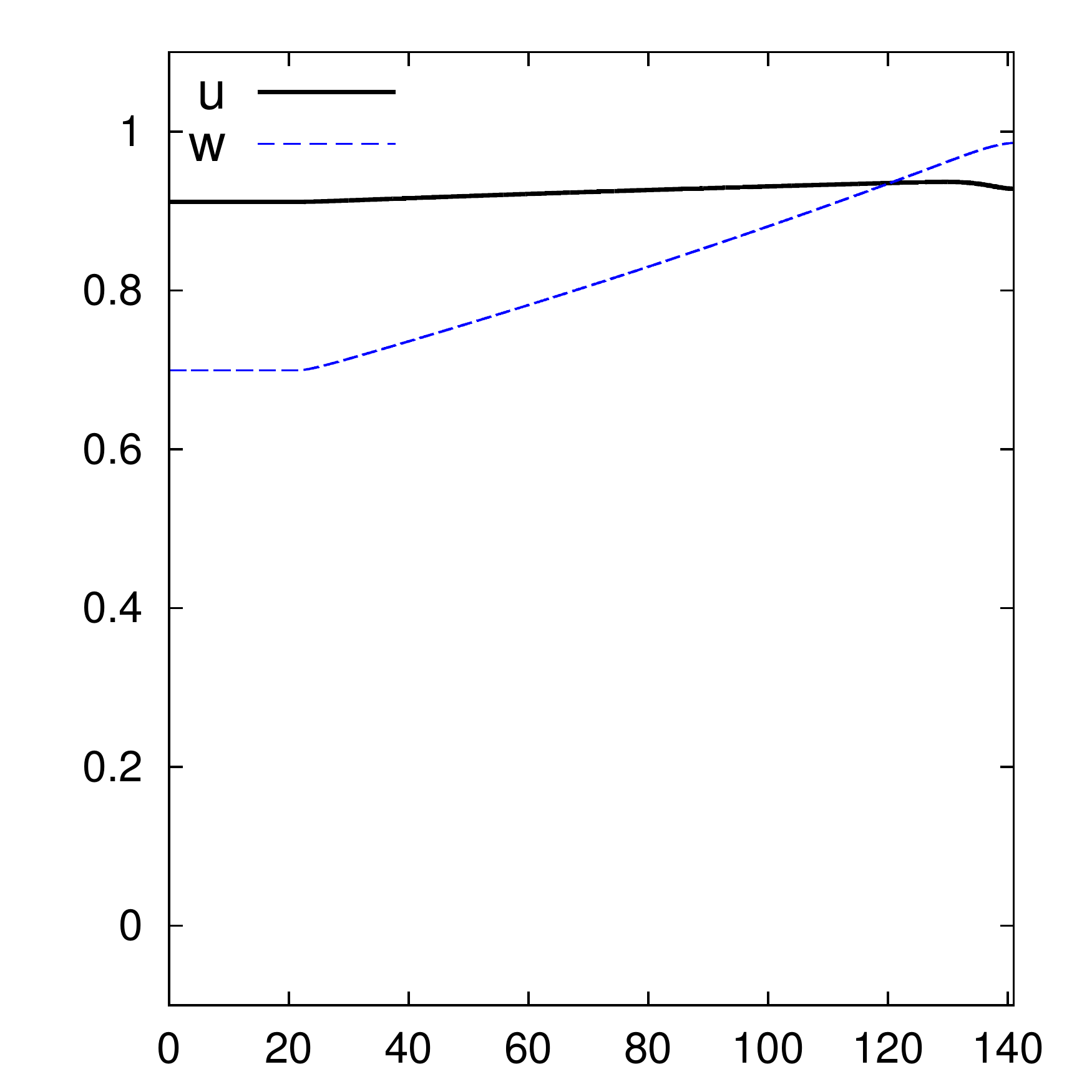}
  }
  \subfloat{
    \includegraphics[width=0.23\textwidth]{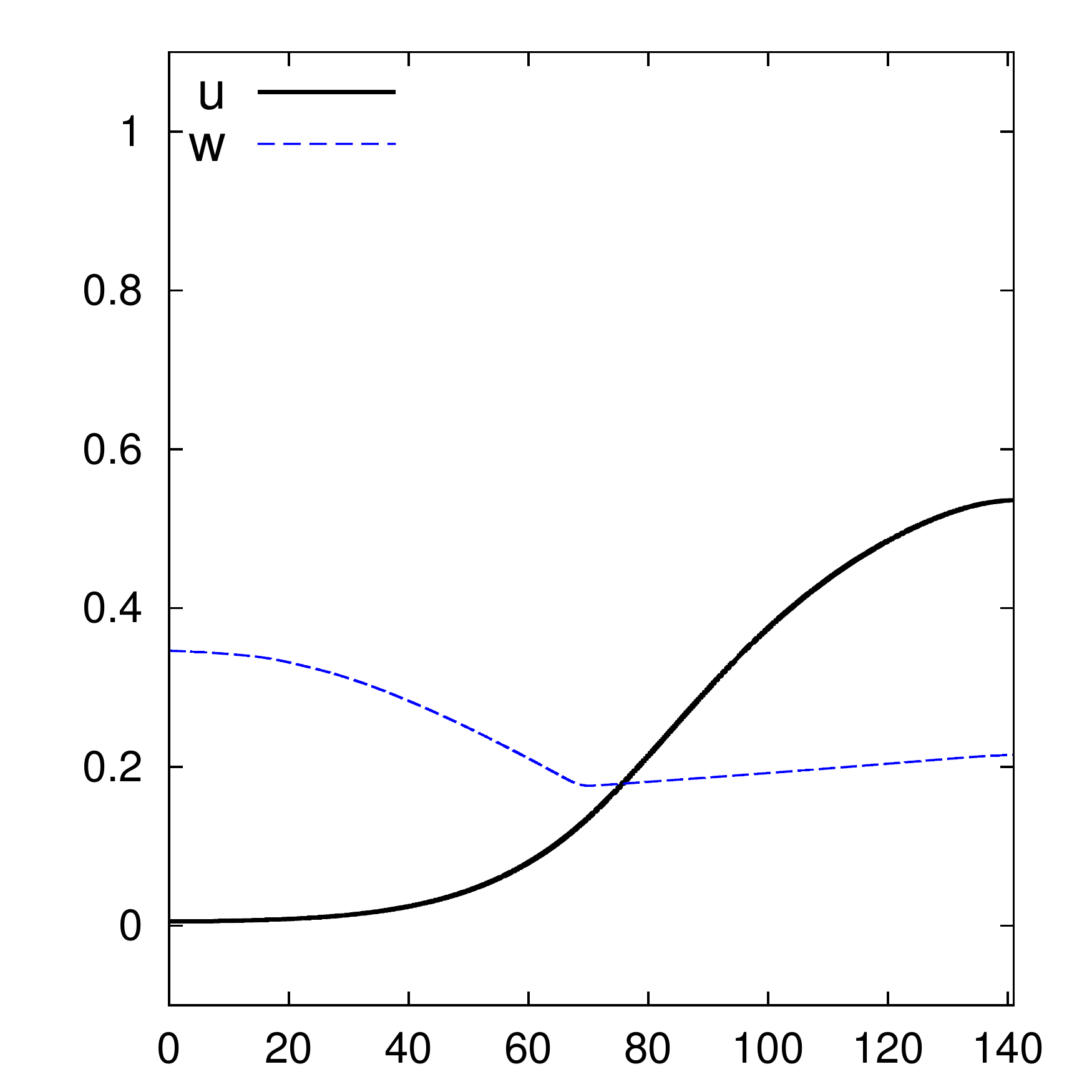}
  }
  \subfloat{
    \includegraphics[width=0.23\textwidth]{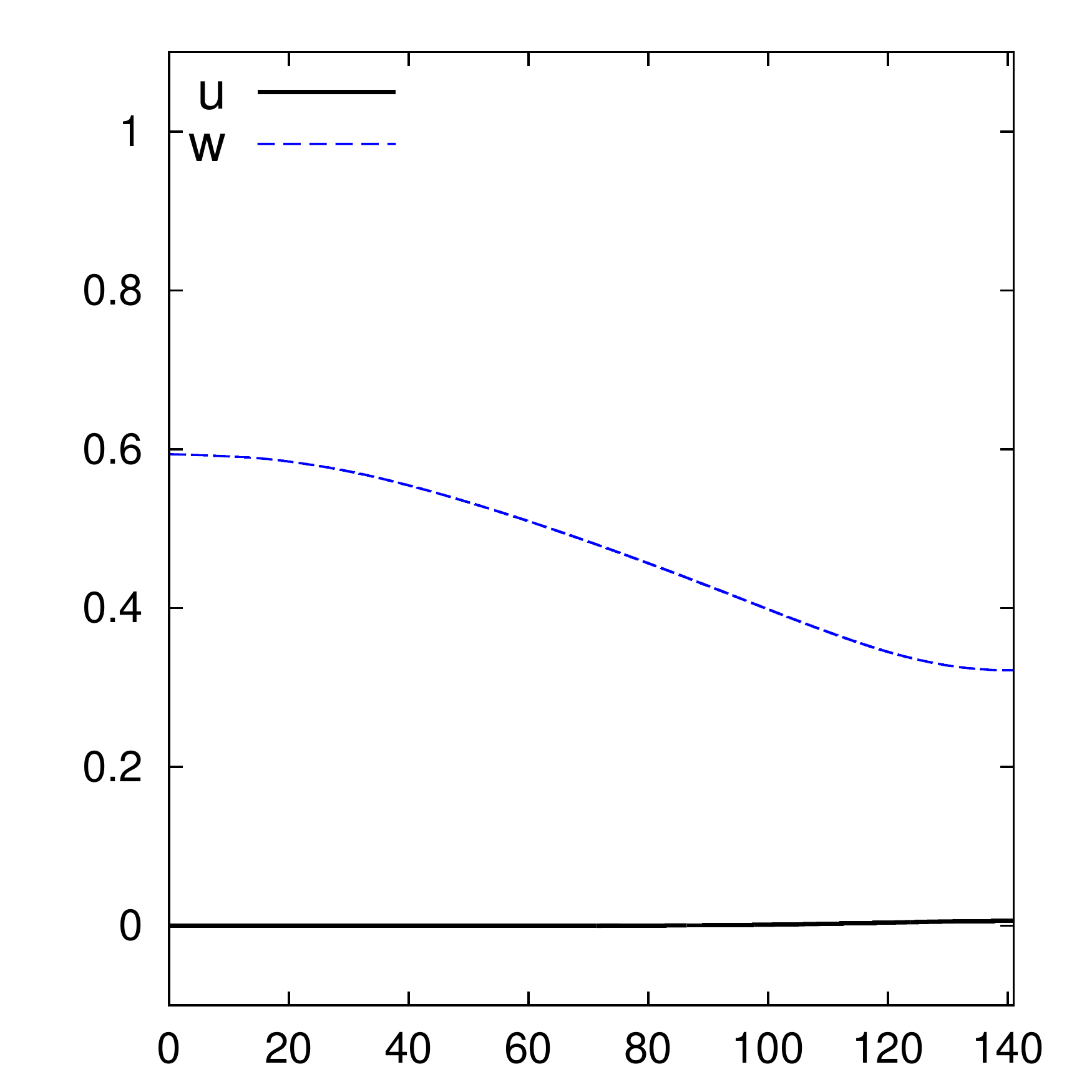}
  }
  \caption{Test case 4: plot of $u,\,w$ over the line $x=y$ as a function of arc
    length. From left to right: depolarization phase at $t=30$, plateau at
    $t=60$, repolarization phase at $t=316$, and the
    recovery phase at $t=360$.}
  \label{fig:MS_tanh100_long-phases}
\end{figure}


\subsubsection{On the residual for the recovery variable}
\label{subsubsec:mono-alterr}

For the monodomain problem, the variables $u$ and $w$ do not decouple, so the
arguments relating to problem (\ref{ode-lin}) do not directly apply. However,
the equation for $w$ in the Mitchell-Schaeffer model switches between two
equations of this form: with $\mu=\tau_{open}^{-1}$, $f=\tau_{open}^{-1}$ for
$u<u_{gate}$ and $\mu=\tau_{close}^{-1}$, $f=0$ for $u\geq u_{gate}$. During the
action potential, this switching occurs when $u$ varies quickly during the
depolarization and repolarization phases. These phases occupy a relatively short
duration. During the plateau and recovery phases, $w$ varies almost
independently of $u$. From Figure \ref{fig:MS-errline-recov} we observe that the
element residual $R_{K,n}^W$ is only significant when $u\approx u_{gate}$, which
occurs on the wave-front, and as a result, the estimator $\eta^{S,W}$ is close
to $0$ elsewhere. On the other hand, the estimator $\omega^{S,W}$ detects a
significant error contribution when $(x^2+y^2)^{1/2}\approx 20$. This can be
understood from Figure \ref{fig:MS_tanh100_long-phases} during the
depolarization and plateau phases. One expects a relatively large contribution
to the interpolation error, as $w$ transitions from a constant value. In this
region the interpolation error is not seen by the residual estimator. This is in
agreement with the discussion in Section \ref{subsec:recov-mod}.

\begin{figure}[ht]
  \centering
  \subfloat{
    \includegraphics[width=0.45\textwidth]{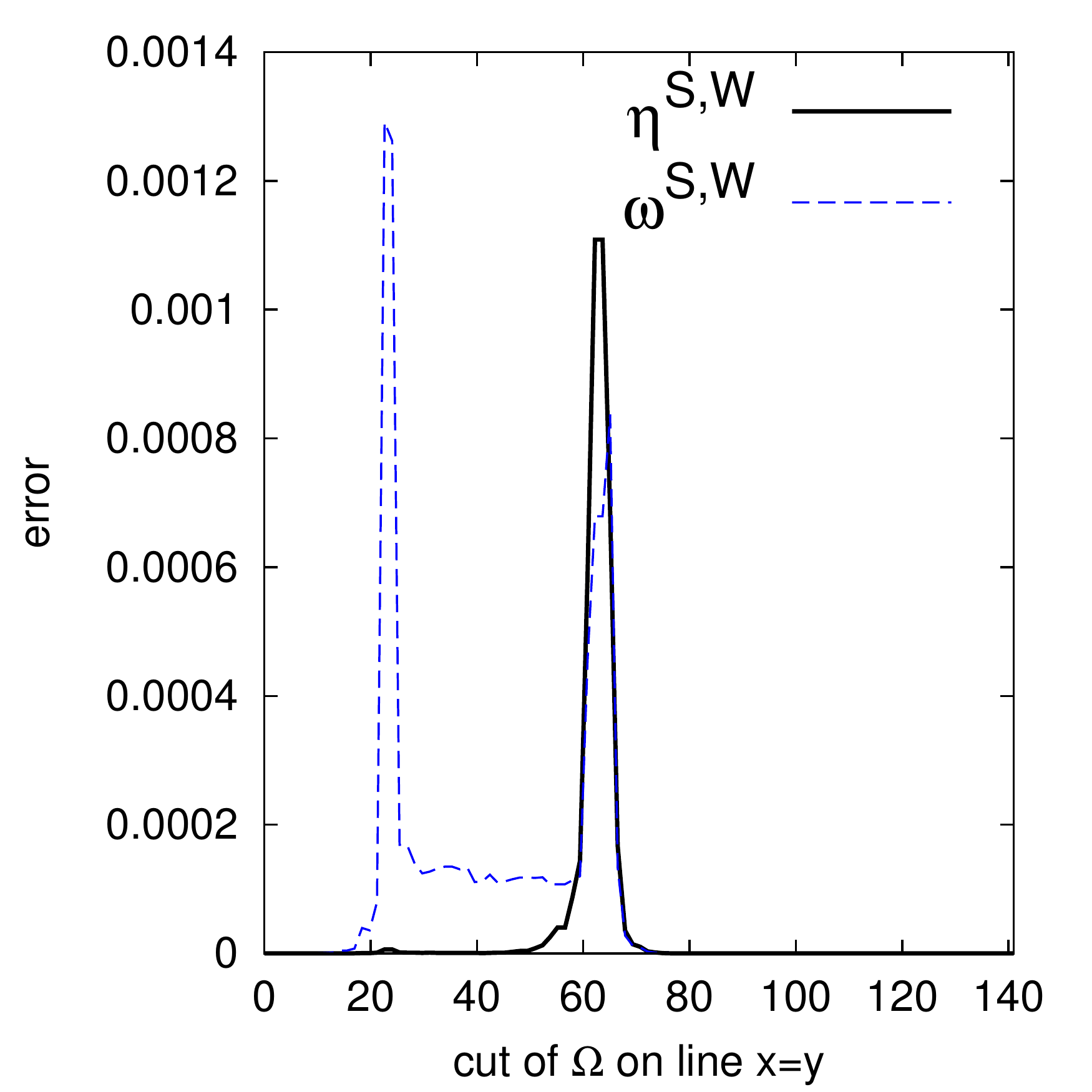}
  }
  \caption{Test case 4: plot of the estimators for $w$ on the line $x=y$ at time
    $t=20$. Solution computed on a uniform mesh with $6400$ elements and a
    constant time step $0.25$.}
  \label{fig:MS-errline-recov}
\end{figure}


\subsubsection{Space-time adaptation}
\label{subsubsec:mono-sptime}


As for the scalar problem, we do not include the terms involving the third order
finite difference reconstructions, which exhibit an overshoot each time the mesh
is adapted due to the interpolation error. We perform a complete space-time
adaptive solution with the following parameters: $TOL_S^U=0.0625$,
$TOL_S^W=0.0625$, $TOL_T^U=0.035875$ and $TOL_T^W=0.0075$. From Figure
\ref{fig:MS_tanh100_long-step-elements-err} we see a variation in the time step
of two orders of magnitude, with the smallest steps ranging from $0.05$ to $0.1$
in the depolarization and repolarization phases, and the largest steps ranging
from about $1$ to $4$ during the plateau and recovery phases. The control of the
error in space is illustrated in Figure
\ref{fig:MS_tanh100_long-step-elements-err} (right), and we see that most of the
adaptation occurs during the depolarization and repolarization phases. The
decision to adapt during the depolarization phase is given by condition
(\ref{eqn-TOL-SU}), $u$ being the fast variable, while it is given by
(\ref{eqn-TOL-SW}) during the repolarization phase. Examples of adapted meshes
for these phases are shown in Figure \ref{fig:MS_tanh100_long-mesh}. The mesh
for the depolarization phase is quite similar to those for the FitzHugh-Nagumo
model in Figure \ref{fig:FHN1-sol-mesh}, and with a small layer of refinement
near $\sqrt{x^2+y^2}=20$, fitting the observation made for Figure
\ref{fig:MS-errline-recov}. We remark that if the mesh were to be adapted only
to the variable $u$, as is done for instance in \cite{sougorpigfarberpit10}, the
mesh would only be refined in the wave front, and the slow variation in $w$
would be not be captured properly. This in turn would eventually spoil the
quality of approximation for the subsequent phases. The mesh for the
repolarization phase is generally more diffuse, with the refinement of the
action potential downstroke requiring fewer elements than the wave front,
resulting in a mesh with $7000$ elements. This is likely due to the slow
variation of $w$.

\begin{figure}[ht]
  \centering
  \subfloat{
    \includegraphics[width=0.32\textwidth]{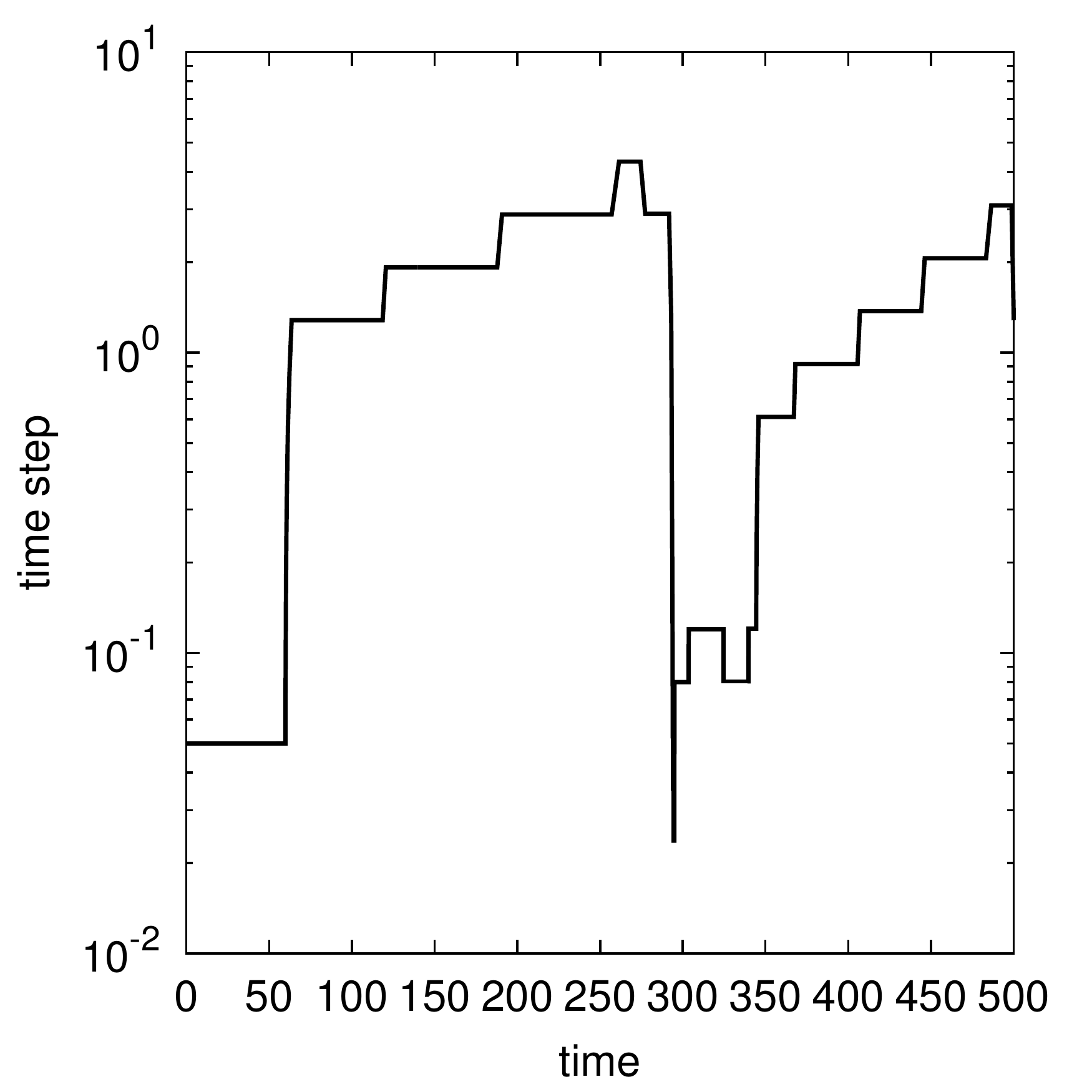}
  }
  \subfloat{
    \includegraphics[width=0.32\textwidth]{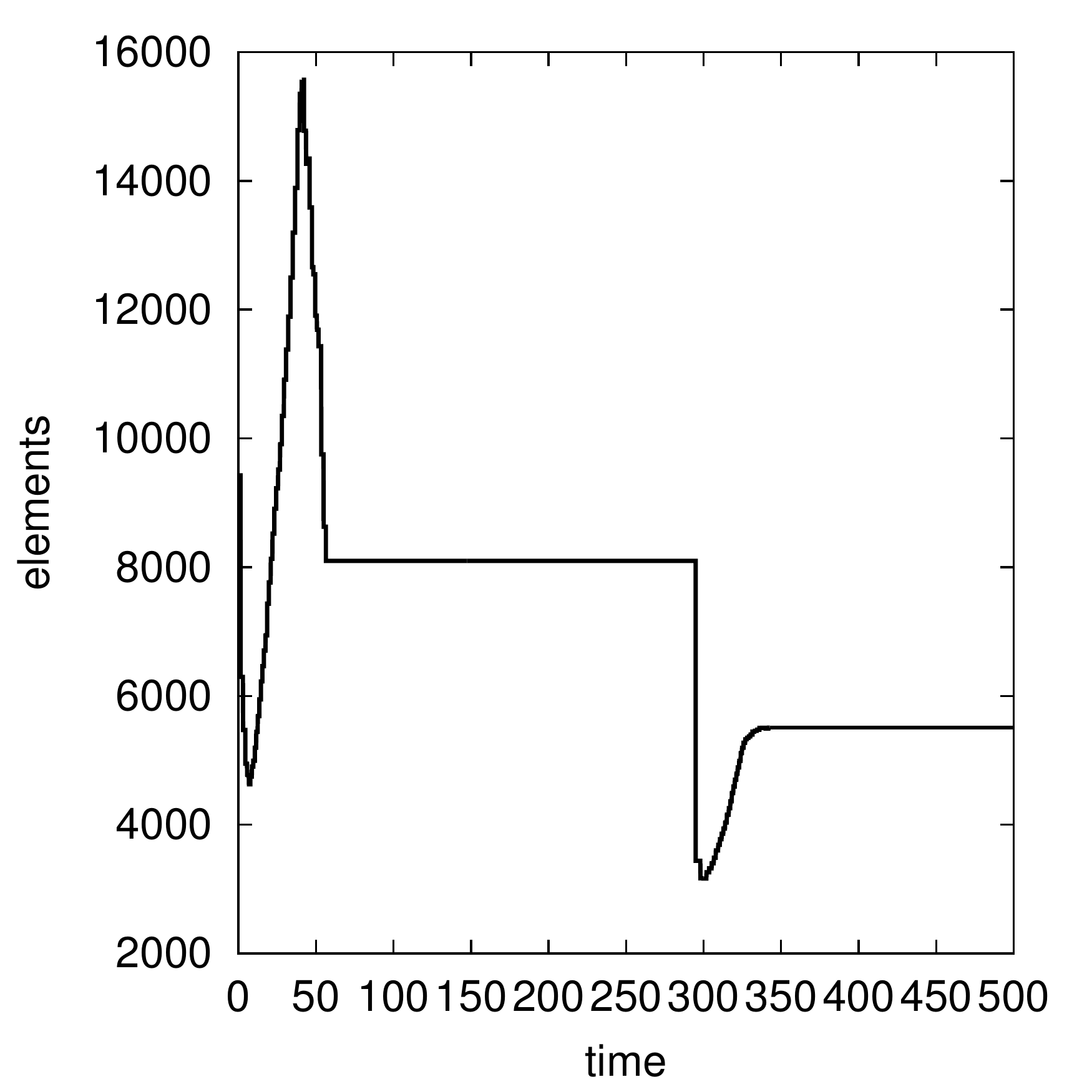}
  }
  \subfloat{
    \includegraphics[width=0.32\textwidth]{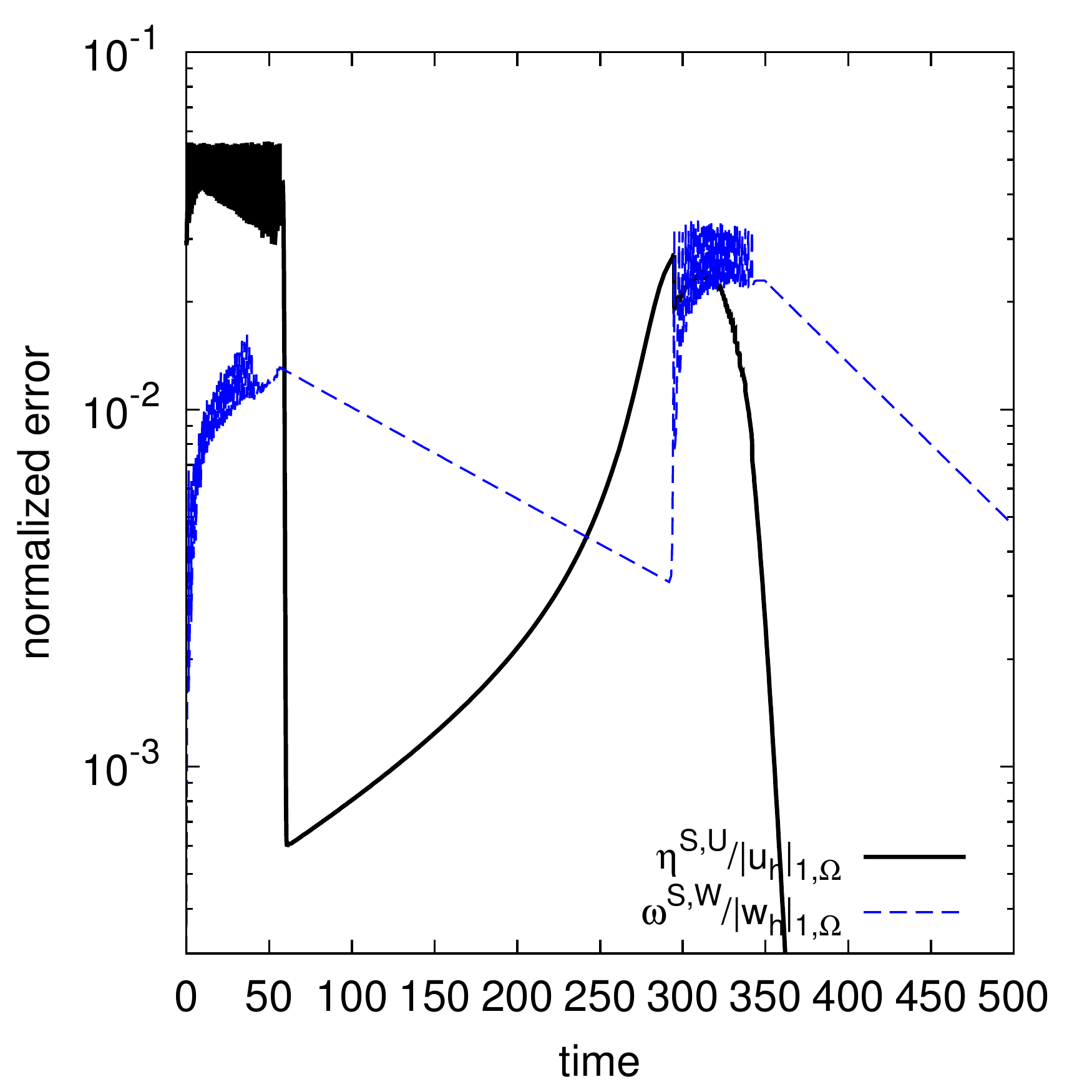}
  }
  \caption{Test case 4: evolution of the time step (left), the number of
    elements (center), and the estimated space error (right) for space-time adapted
    solution.}
  \label{fig:MS_tanh100_long-step-elements-err}
\end{figure}

\begin{figure}[ht]
  \centering
  \subfloat{
    \includegraphics[width=0.45\textwidth]{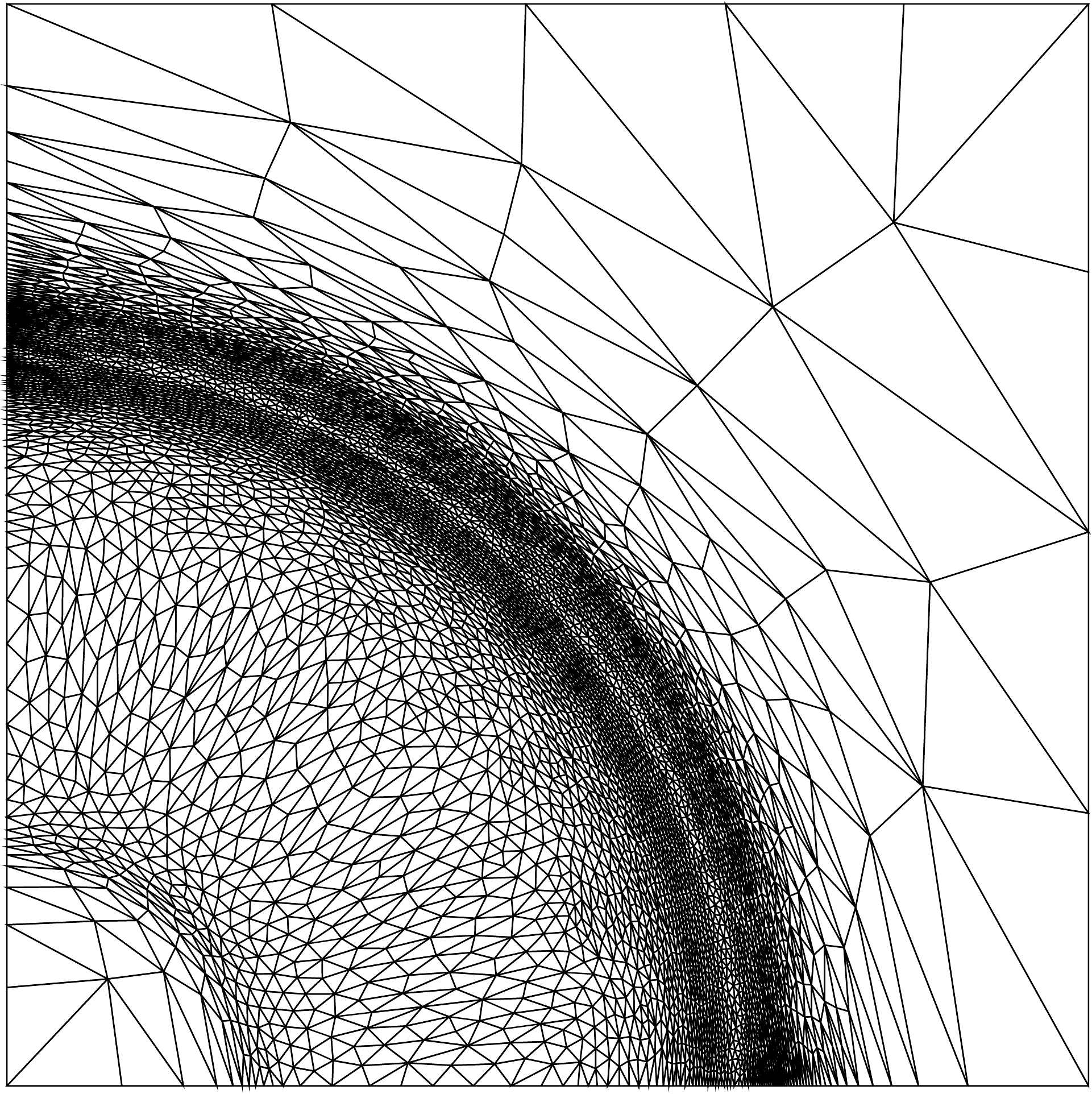}
  }
  \subfloat{
    \includegraphics[width=0.45\textwidth]{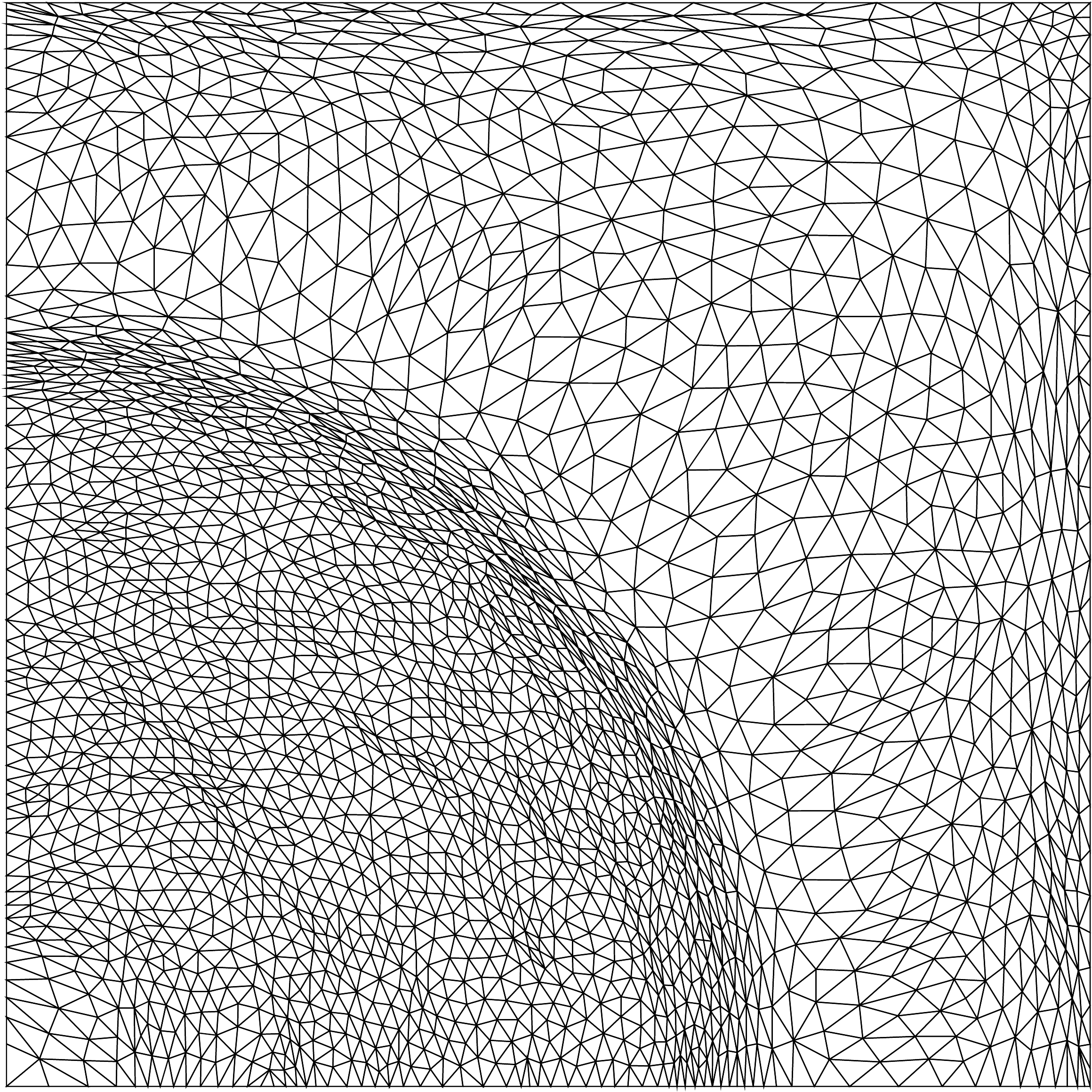}
  }
  \caption{Test case 4: adapted mesh at $t=30$ (left) during the depolarization
    phase and at $t=316$ (right) during the repolarization phase.}
  \label{fig:MS_tanh100_long-mesh}
\end{figure}

As a final note, the efficiency of applying the space-time adaptation algorithm
for this example requires a thorough study. This could for instance be done by
computing a reference solution as is done for Test case 3, and
therefore approximating the exact error. However, due to the scale of this
problem, a different approach may be required.


\section{Conclusion}
\label{sec:conc-mono}

In this paper we introduced an anisotropic residual error estimator for a scalar
reaction-diffusion problem and for the monodomain system, discretized with $P_1$
finite elements in space, and the variable step BDF2 method in time. The
estimator is shown to give an upper bound for the error in the energy norm for
the parabolic variables, and in $L^2(0,T;L^2(\Omega))$ norm for the second
variable. It was found that the residual estimator for the second variable of
the monodomain problem could not be used in practice for mesh adaptation
purposes, due to a fundamental difference in behaviour between parabolic PDEs
and ODEs. In particular, the estimator does not provide a suitably local
representation of the error. Instead, a simplified estimator was proposed for
the ODE variable, which is based on interpolation estimates combined with a
gradient recovery operator. Numerical computations are carried out, confirming
the reliability of the estimator. A space-time adaptation method is proposed to
simultaneously control the error in space and time. The mesh is adapted using a
metric to control the anisotropic nature of the error. For the scalar problem,
it was found that the space-time method is at least as efficient in terms of
achieving a global level of error in a given CPU time compared to applying mesh
adaptation with a constant time step, and significantly more efficient than
computing with a uniform mesh and constant time step. For the monodomain
problem, improved efficiency was observed controlling the error for the
transmembrane potential when applying the error estimator guided mesh adaptation
method. The space-time adaptation method is applied to a problem exhibiting
large variation in time scales. While the results appear promising, more work is
required to accurately assess the efficiency of the algorithm.

\section*{Acknowledgements}

The authors would like to acknowledge the financial support of an Ontario
Graduate Scholarship (OGS) and by Discovery Grants of the Natural Sciences and
Engineering Research Council of Canada (NSERC).  The authors wish to thank the
professionals and researchers at GIREF, in particular Thomas Briffard, {\'E}ric
Chamberland, Andr{\'e} Fortin, and Cristian Tibirna, for making available their
code MEF++ and for their assistance in using the code during visits to the
laboratory at Universit\'e Laval and email correspondences.


\FloatBarrier
\section*{References}
\bibliographystyle{plain}
\bibliography{../../bib/list}

\end{document}